\documentclass{article}
\usepackage{graphicx}
\usepackage{amsmath,amsthm,amssymb}
\usepackage{authblk}
\usepackage{setspace}
\usepackage{mathtools}
\usepackage[margin=1.25in]{geometry}
\usepackage{hyperref}
\usepackage{color}
\usepackage{algorithmic}

\newcommand{\both}[1]{\mathbf{#1}}
\newcommand{\bi}{\both{i}}
\newcommand{\bj}{\both{j}}
\newcommand{\bk}{\both{k}}
\newcommand{\bI}{\both{I}}
\newcommand{\Lip}{\operatorname{Lip}}
\newcommand{\Mat}{\operatorname{Mat}}
\newcommand{\PP}{\bar{P}}
\newcommand{\xx}{\mathrm{x}}

\newtheorem{theorem}{Theorem}[section]
\newtheorem{lemma}[theorem]{Lemma}
\newtheorem{proposition}[theorem]{Proposition}

\newtheorem{condition}[theorem]{Condition}
\newtheorem{remark}[theorem]{Remark}
\newtheorem{definition}[theorem]{Definition}
\newtheorem{algorithm}[theorem]{Algorithm}

\newcommand{\slow}{\mbox{\tiny \textup{slow}}}
\newcommand{\fast}{\mbox{\tiny \textup{fast}}}
\newcommand{\unst}{\mbox{\tiny \textup{unst}}}

\newcommand{\eps}{\epsilon}

\newcommand{\R}{\mathbb{R}}
\newcommand{\C}{\mathbb{C}}
\newcommand{\N}{\mathbb{N}}

\newcommand{\Z}{\mathbb{Z}}

\newcommand{\kQ}{\mathbf{Q}}
\newcommand{\kP}{\mathbf{P}}
\newcommand{\kR}{\mathbf{R}}
\newcommand{\cC}{\kQ^2}

\newcommand{\cA}{\mathcal{A}}
\newcommand{\cB}{\mathcal{B}}

\newcommand{\cE}{\mathcal{E}}

\newcommand{\cG}{\mathcal{G}}
\newcommand{\cH}{\mathcal{H}}
\newcommand{\cL}{\mathcal{L}}

\newcommand{\cN}{\mathcal{N}}
\newcommand{\cO}{\mathcal{O}}
\newcommand{\cQ}{\mathcal{Q}}
\newcommand{\cT}{\mathcal{T}}

\newcommand{\fL}{\mathfrak{L}}

\newcommand{\loc}{\mbox{\scriptsize \textup{loc}}}

\usepackage{fancyhdr}
\title{Validated Numerical Approximation of  Stable Manifolds \\ for Parabolic Partial Differential Equations}
\author[a]{Jan Bouwe van den Berg
	\thanks{This work is part of the research program \emph{Connecting Orbits in Nonlinear Systems} with project number NWO-VICI 639.033.109, which is (partly) financed by the Dutch Research Council (NWO).}
	}
\affil[a]{\small
	Department of Mathematics, Vrije Universiteit Amsterdam, de Boelelaan 1111, 1081HV Amsterdam, the Netherlands. \url{janbouwe@few.vu.nl}
}

\author[b]{
	Jonathan Jaquette \thanks{
		Present Address:	Department of Mathematics and Statistics, Boston University,
		Boston, MA 02215, USA. \url{jaquette@bu.edu}. 
		}\thanks{
		This material is based upon work supported by the National Science Foundation under Grant No. DMS-1440140 while JJ was in residence at the Mathematical Sciences Research Institute in Berkeley, California, during the Fall 2018 semester.}
}
\affil[b]{ 
	Department of Mathematics, Brandeis University,
	Waltham, MA 02453, USA. 
	}

\author[c]{
		J.D. Mireles James 	
	\thanks{JDMJ was partially supported by National Science Foundation grant DMS - 1813501 during work on this project.}
}

\affil[c]{
	Department of Mathematical Sciences, Florida Atlantic University, 
	Boca Raton, FL 33431, USA. \url{jmirelesjames@fau.edu}
}

\begin{document}

\maketitle

\begin{abstract}
This paper develops validated computational methods for studying infinite dimensional stable manifolds at equilibrium solutions of parabolic PDEs, synthesizing disparate errors resulting from numerical approximation. To construct our approximation, we decompose the stable manifold into three components: a finite dimensional slow component, a fast-but-finite dimensional component, and a strongly contracting infinite dimensional ``tail''. We employ the parameterization method in a finite dimensional projection to approximate the slow-stable manifold, as well as the attached finite dimensional invariant vector bundles. This approximation provides a change of coordinates which largely removes the nonlinear terms in the slow stable directions. In this adapted coordinate system we apply the Lyapunov-Perron method, resulting in mathematically rigorous bounds on the approximation errors. As a result, we obtain significantly sharper bounds than would be obtained using only the linear approximation given by the eigendirections. As a concrete example we illustrate the technique for a 1D Swift-Hohenberg equation.
\end{abstract}

\bigskip
\centerline{{\bf Keywords  }}
\centerline{
	parabolic partial differential equations, stable manifold, Lyapunov-Perron method,
}
\centerline{
	 parameterization method, rigorous numerics, computer assisted proof
}

\bigskip
\centerline{{\bf AMS subject classifications}}
\medskip
\centerline{
  35B40,	
  35B42, 
  35K55, 
  37L15, 
  37L25, 
  37L65, 
  37M21 
 }

\bigskip 


\clearpage



\section{Introduction}

In this paper we develop a novel method for representing the 
\emph{infinite dimensional} stable manifold of an equilibrium solution of a 
parabolic PDE.  The method 
makes extensive use of numerical calculations, 
 results in an approximation valid in an explicitly prescribed 
neighborhood of the equilibrium, and comes equipped with  
mathematically rigorous bounds on all truncation 
and discretization errors.  The work is motivated by our intention to 
use this method as an ingredient in further mathematically rigorous
computer assisted proofs (see also Section \ref{sec:motivation}).
The method is able to provide validated 
bounds on the linear approximation of the stable manifold
by the stable eigenspace, but gives dramatically improved 
results when combined with  a nonlinear change of coordinates which
``flattens out'' a finite dimensional slow stable manifold.
The main tools used here are the Lyapunov-Perron method,
a parameterization method for slow-stable manifolds and their invariant 
normal bundles (see \cite{MR3541499}),
and an iterative strategy for 
bootstrapping Gronwall's inequality in subspaces associated with
various linear growth rates.  

We remark first on the need for the present work,
noting that while the abstract theory for 
invariant manifolds of compact semi-flows is well developed, there are 
obstacles preventing its direct application in
computer assisted proofs. 
One complication stems from the fact that in a given example
we generally do not have explicit formulas for either the 
equilibrium or the eigendecomposition 
of the linearized operator:  instead we have approximations. 
To perform computer assisted proofs, these approximation 
errors must be incorporated into the set-up from the start.

A second difficulty concerns 
localizing the estimates, which is necessary 
because the nonlinearities are not globally
 Lipschitz. Moreover, in infinite dimensions 
 we do not generally have access to smooth cut-off functions.  
 Finally, even in situations where it is possible 
 to apply the general theory, this typically leads  
 to bounds that are valid in an inconveniently small neighborhood of the 
 equilibrium.

To overcome these difficulties, 
we project the Lyapunov-Perron operator
into various judiciously chosen subspaces, corresponding to collections of
approximate eigendirections.  The assumption that the PDE is parabolic 
gives that the spectrum is comprised entirely of 
isolated eigenvalues (of finite multiplicity) 
which ``accumulate to minus infinity''.  More precisely,
 for any $M \in \mathbb{R}$ there are only finitely 
many eigenvalues with real part greater than $M$.
 We choose an approximation of the 
(finite dimensional) unstable subspace, 
and split the approximate stable space 
into finite dimensional ``slow'' and infinite dimensional ``fast'' parts.  
As a subtle refinement,
we further decompose the finite dimensional stable
eigenspace into slow-finite dimensional stable and fast-finite dimensional 
stable subspaces.

We remark that the Lyapunov-Perron operator acts on candidate
functions $\alpha$, 
which map (an approximation of) the linear stable eigenspace to the 
(approximate) unstable eigenspace.  The main
technical difficulty is to choose 
the domain of the candidate
functions so as to maximize the portion of the manifold
represented, while minimizing the final error bounds.
To manage this problem we take domains which are products
 of balls, having aspect ratios determined by the 
growth rates in the various subspaces.
We perform an explicit 
change of coordinates, which may be linear or nonlinear,
and which provides more 
flexibility in choosing a good domain for the stable manifold approximation.

To show that the Lyaponuv-Perron operator is a contraction
we need explicit bounds on the  
projections of the nonlinearities 
onto the specified subspaces. 
To obtain effective bounds, 
i.e.\ bounds that guarantee contraction for functions defined on a reasonably 
large neighborhood of the equilibrium, a naive Gronwall estimate does not suffice.
Instead we take a more refined approach, in which we bootstrap a system of 
Gronwall inequalities (roughly, decomposed along eigendirections) 
exploiting the different decay rates in different directions.  
The applications to 
computer assisted proofs of transverse connecting orbits we
have in mind (see again Section \ref{sec:motivation}),
introduce the additional technical complication that we
would like a $C^{1,1}$ description of the stable manifold.


\begin{figure}[t!]
	\centering
	\includegraphics[width=0.7\linewidth]{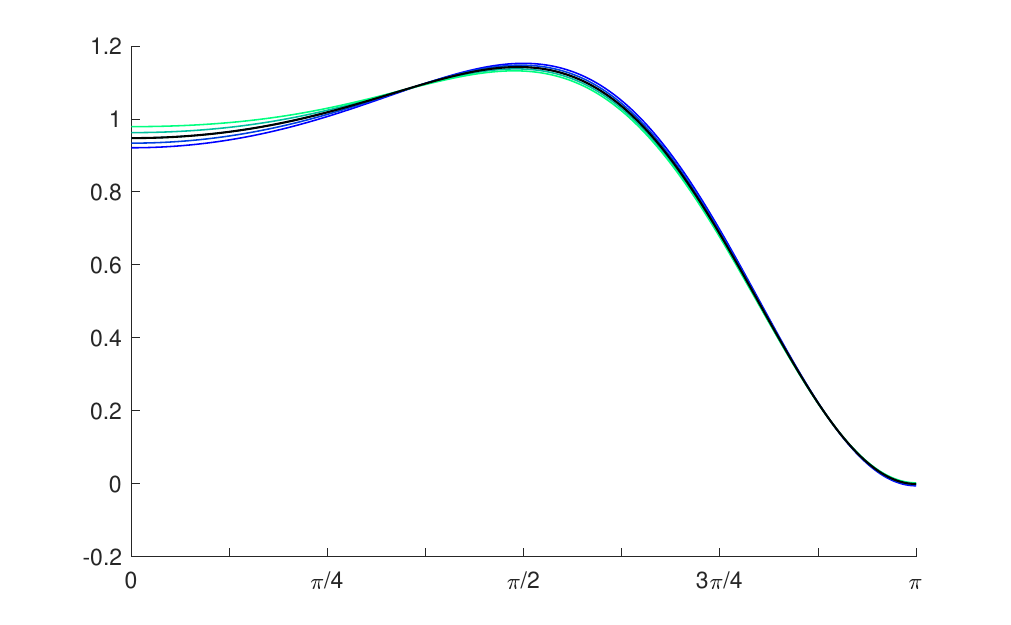} 
	\caption{A verified numerical approximation of an unstable 
	equilibrium (black curve) for the Swift-Hohenberg 
	PDE~\eqref{eq:SwiftHohenberg} with $\beta_1=0.05$ and 
	$\beta_2=-0.35$ and several (numerical approximations of) 
	``points'' -- that is functions -- 
	along its verified slow stable manifold. Near this slow stable 
	manifold we find a description of the full, co-dimension~1, 
	stable manifold, with validated computer assisted error bounds.
}
	\label{fig:SlowManifoldEndPoints}
\end{figure}

\subsection{Example results for Swift-Hohenberg} \label{sec:SH_eq}
The utility of the method is best illustrated through application 
to an explicit example.  To this end we 
 provide a complete numerical 
implementation of our method for the Swift-Hohenberg PDE
\begin{equation}\label{eq:SwiftHohenberg}
	u_t = - \beta_1  u_{xxxx} + \beta_2 u_{xx} + u -u^3,	
\end{equation}
posed on a one-dimensional spatial domain $x \in [0,\pi]$
with Neumann boundary conditions
\[
u_x(0)=u_x(\pi)=0 \qquad \text{and} \qquad u_{xxx}(0)=u_{xxx}(\pi)=0 .
\]
The parameters of the problem are $\beta_1 > 0$ and $\beta_2 \in \R$.  
For comparison, we illustrate the use of our method for 
 both a linear, and a nonlinear change
 of variables near the equilibrium. 
As a result, we obtain  stable manifold theorems of 
varying accuracy, and in neighborhoods of the equilibrium 
having various sizes and shapes. 

For example, in Theorem \ref{prop:CAP_Linear} we focus
on a non-trivial equilibrium solution of Swift-Hohenberg
with Morse index 1.  The equilibrium solution is illustrated in Figure 
\ref{fig:SlowManifoldEndPoints}.  
To obtain the results described in  Theorem \ref{prop:CAP_Linear},
we represent the local stable manifold 
as the graph of a function over the stable eigenspace.
We take a 31 dimensional Galerkin projection, so that 
the stable eigenspace is decomposed into a 30 dimensional
finite part, and an infinite dimensional remainder.  The domain of the graph 
is taken to be the product of a box of radius $2.2\times 10^{-2}$ in 30 
dimensional subspace,  and a box of radius $10^{-5}$ in the tail.
The chart for the local stable manifold has $C^0$ norm bound by 
$3.36 \times 10^{-3}$.  That is, the true stable manifold 
has distance no more than
$3.36 \times 10^{-3}$ away from the stable eigenspace, over 
the box just described.

Contrast this with the results described in 
Theorem~\ref{prop:NonlinearComputerAssistedProof}.
In this case we use the nonlinear change of coordinates discussed in 
Section \ref{sec:goodCoordinates}, and represent the local stable 
manifold as the graph of a function over a one dimensional
 slow-stable manifold and its 29 dimensional invariant stable vector 
 bundles.  This time the domain 
 of the graph is the product of three boxes: a box of radius $3.18 \times 10^{-2}$
 in the slow stable direction, a box of radius $10^{-6}$ in the remaining
 29 dimensions of the finite dimensional eigenspace, and a
 box of radius $10^{-10}$ in the tail.   
The chart for the local stable manifold has $C^0$ norm bound by 
$7.34 \times 10^{-12}$.  That is, the true stable manifold is 
$7.34 \times 10^{-12}$ close to the slow stable manifold and 
its stable vector bundles over the box just described. 
 
Comparing the results of 
Theorem \ref{prop:CAP_Linear}  with the results of 
 Theorem~\ref{prop:NonlinearComputerAssistedProof}
illustrate the power of the techniques developed in the present work.
The two representaitons of the infinite dimensional stable manifold
are valid in neighborhoods having size on the order of
$10^{-2}$ away from 
the equilibrium (in some directions).  
Exploiting the nonlinear change of variables improves the
validated error bounds by nine order of magnitude in the 
unstable directions (bounds on the graph) and by five orders of
magnitude in the stable tail directions.  These are by far the most
accurate mathematically rigorous computer assisted error bounds
for an infinite dimensional manifold appearing in the literature
up until now.
More details and comparisons are found in Sections~\ref{sec:Linear_Results} 
and~\ref{sec:ConclusionAndNumerics}.

%

\subsection{Motivation: saddle-to-saddle connects for parabolic PDEs}
\label{sec:motivation}

When viewed as ODEs on Banach spaces, nonlinear parabolic PDEs
 fit well within the qualitative theory of dynamical systems.  Theorems 
 regarding the stability of equilibria, periodic orbits, and their attached invariant 
 manifolds follow in analogy with the finite dimensional case. Connecting orbits 
 between invariant sets serve as a kind of a road map to the global dynamics,
 illuminating  transitions between distinct regions of the phase space and signaling 
 global bifurcations. Such  orbits are 
main ingredients in forcing theorems like those of Smale 
and Shilnikov: theorems which guarantee the existence of rich dynamics.
Connecting orbits are 
essential for defining geometric chain groups and boundary operators 
in the homology theories of Witten and Floer.

Precisely because of their global and nonlinear nature, 
connecting orbits are difficult to work with analytically.  
These difficulties are compounded in infinite dimensional settings. 
In specific applications researchers typically perform numerical 
calculations to gain insights into the properties of important invariant objects. 
Recent progress in computer-assisted methods of proof for infinite dimensional systems brings the mathematically rigorous quantitative study of  connecting 
orbits for PDEs within the realm of possibility. 

We refer for example to the work of \cite{MR3773757,reinhardt2019fourier}
for some examples of computer assisted proofs for connecting orbits
in PDEs.  In particular the authors study connections frome saddle 
to attracting equilibrium solutions. 
The works just mentioned study the finite dimensional unstable 
manifold attached to an equilibrium,
 and develop mathematically rigorous tools for extending this 
manifold into a trapping neighborhood of a sink. 
Similarly, in a nonconservative
nonlinear Schr\"{o}dinger equation, the work  \cite{jaquette2020global} 
computes  connecting orbits from  saddle equilibria to a center equilibrium.
 In each of the studies just mentioned the authors obtain 
explicit and mathematically rigorous bounds on the basin of attraction of the 
limiting equilibrium -- which is an open set.  

Controlling the asymptotic behavior of a connecting orbit requires an explicit 
description of the local stable and unstable manifolds of the equilibrium solutions
(or other limiting invariant sets). 
The major obstacle to extending the methods of 
\cite{MR3773757,jaquette2020global,reinhardt2019fourier} to the general
case of a saddle-to-saddle connection is obtaining an explicit description 
of the local stable manifold.  
It is worth mentioning that rigorous numerical integration of a PDE is a 
nontrivial task, and invariably suffers from the so called wrapping effects resulting 
from the accumulation of numerical error. 
Consequently, in computer assisted arguments involving connecting orbits it is 
desirable to minimize integration time by absorbing as much of the connecting 
orbit into the local stable and unstable manifolds as possible.  
This motivates out interest in the nonlinear coordinate changes utilized 
in the present work.

We refer the interested reader also to the related work of
\cite{MR2136516},
where saddle-to-saddle connections are established using topological 
methods based on Conley Index theory and its connection matrix. Being 
topological in nature these methods require much less in the way of $C^1$
 information, resulting in a softer description of the dynamics. The challenge 
 in applying these methods is the rigorous calculation of index information for
  macroscopic regions in the infinite dimensional phase space. 

The computational framework developed here is rather general, and will be  
useful for describing invariant manifolds in a variety of other settings.  
We have in mind examples such as (un)stable and center-(un)stable 
manifolds in delay differential equations and partial differential equations
 on domains in $\R^n$, as well as stable and unstable manifolds in strongly 
 indefinite problems, where both the dimension and the co-dimension of the 
 manifold are infinite dimensional (e.g. \cite{cheng2020stable}).
In \cite{takayasu2019rigorous} a similar methodology is used to construct  
a local representation for a co-dimension~0 center-stable manifold 
of the homogeneous equilibrium 
in a complex-valued nonlinear heat equation. 

\subsection{Related work} \label{sec:relatedWork}

The present work grows out of the thriving literature on methods 
of computer assisted proof in dynamical systems theory going back to 
the first proofs of the Feigenbaum conjectures 
\cite{MR759197,MR775044,MR883539,MR727816},
the first proofs of chaotic motions in the Lorenz equations 
\cite{MR1276767,MR1808460,MR1459392,MR1661345} and for Chua's circuit 
\cite{MR1453709}, and the computer assisted resolution of Smale's 
14th problem \cite{MR1701385,MR1870856}. 
In particular, we build on the substantial literature on 
computer assisted proofs for 
studying the dynamics of parabolic PDEs.    
A thorough review of this literature 
beyond the scope of the present work, 
and we refer the reader to the work of
\cite{MR944817,MR1100582,MR1131109,MR1838755,MR2443030,
MR2776917,MR3377562,MR2679365,MR2852213,MR3904424}.
See also the book of 
\cite{MR3971222},  and the review articles  
\cite{MR3990999,MR3444942,MR759197}

A number of techniques for computer assisted proofs involving finite 
dimensional invariant manifolds have emerged from this literature. 
One family of methods for proving existence of unstable manifolds involves 
checking a number of geometric covering and cone conditions 
near the equilibrium in the same spirit as  Fenichel theory
 \cite{zgliczy2009covering,capinski2011cone,capinski2015geometric}. 
Since time reversal is well defined for ODEs, equivalent bounds for stable 
manifolds follow as a trivial corollary.  
Applications of these methods to the study
	of stable manifolds for PDEs requires substantial modification and 
	have -- to the best of our knowledge -- not yet appeared in the literature. 
We refer the interested reader to the recent work of \cite{danielKS_chaosProof}
where, following 
\cite{MR1276767,MR1808460,MR1459392,MR1661345,MR1453709},
 the authors bypass consideration of  stable/unstable manifolds 
and provide a direct computer 
assisted proof of the existence of a geometric horseshoe in the 
Kuramoto-Sivashinsky equation, by studying covering relations in a 
Poincar\'e section.

Another technique for obtaining validated bounds on invariant manifolds
which has been applied successfully in a number of finite dimensional 
settings is the parametrization method 
\cite{cabre2003parameterizationI,cabre2003parameterizationII,
cabre2005parameterizationIII}, see also to the book 
\cite{haro2016parameterization} 
for detailed discussions of the method and its applications. 
Briefly, the idea is to study a conjugacy equation between the 
dynamics on the manifold and the linear dynamics in an eigenspace.
The conjugacy equation is reduced to a set of linear homological 
equations via recursive power matching, and one obtains a high order 
Taylor expansions for the manifold, as well as remainder estimates on the 
truncation errors in the tail of the series. This method 
recovers both the embedding of the manifold and the dynamics on it, and 
is very effective for representing invariant manifolds far beyond a small 
neighborhood of the equilibrium, periodic orbit, or invariant torus, where 
the linear approximation is valid. 

There is a substantial literature devoted to validated numerics based 
on the parameterization method for invariant manifolds of ODEs.
We refer the interested reader to the works of 
\cite{MR3281845,MR2773294,MR3437754,MR3518609,MR3871613,MR3792792}
for more a complete discussion.
Such methods have also been extended for studying finite dimensional 
invariant manifolds of infinite dimensional systems.  The case of compact 
infinite dimensional maps is treated in \cite{MR3735860}, the case of PDEs 
is studied in \cite{reinhardt2019fourier},
and DDEs are considered in \cite{groothedde2017parameterization,validatedParmMethod_ddes}. 

However, there is an obstruction to applying the parameterization 
method to infinite dimensional manifolds in PDEs, which is that the existence of 
a conjugacy depends certain non-resonance conditions between the 
eigenvalues.  
There are techniques to deal with the case of a finite  number of 
resonant eigenvalues \cite{cabre2003parameterizationI,parmChristian}. 
Nonetheless, to describe an infinite dimensional manifold one will have an 
infinite number of resonance conditions to check, which seems to be a 
major obstruction.  Indeed, there is no good reason to think that 
a parabolic PDE can in practice satisfy infinitely many non-resonance
conditions.

Instead, we consider the 
two widespread approaches for studying 
infinite dimensional invariant manifolds in Banach spaces: these are
the graph transform method   (e.g.~see  \cite{bates1998existence}),
 and the Lyapunov-Perron  method   (e.g.~see  \cite{chow1988invariant}). 
 We refer to~\cite[Section 1.4]{eldering2013} for a  comparison of these 
 methods, but the important point to mention here is that the 
graph transform method is most natural for discrete time dynamical systems. 
Indeed, in ~\cite{Llave2016connecting},
 a graph transform-type argument was used to obtain
 validated computer assisted error bounds for the infinite
 dimensional stable manifold of a compact infinite dimensional 
 map generated by convolution against a smooth kernel.
 The result just cited was a significant motivation for the present work.
%
%
%
 The graph transform method applies to continuous 
 time systems by considering the implicitly defined time-$1$ map generated 
 by the semi-flow. But this requires direct access to the time-$1$ maps, 
 which are defined only implicitly by the PDE.  Because of this,
  we have opted to work with the 
 Lyapunov-Perron method.  
The present work extends the work of \cite{Llave2016connecting}
  to parabolic PDEs, exploiting geometric techniques in the 
  projection space which allow us to obtain validated results on 
  much larger domains.

\subsection{Organization of the present work}

The outline of the paper is as follows. 
In Section \ref{sec:Notation} we discuss the notation to be used in this paper, and the level of generality to be considered. 
Abstractly, we assume that our approximate (un)stable eigenspaces are decomposed into further subspaces, with (potentially) different time scales. 
This corresponds to our plan to develop distinct methods of approximation along the slow-stable, fast-but finite-stable, and infinite-stable eigenvalues. 
We intend to compute $C^{1,1}$ bounds on our manifold, and here we define a number of constants relating to our nonlinearity $ \cN$.

In Section \ref{sec:ExponentialTracking} we discuss how we  explicitly bootstrap Gronwall’s inequality to get component-wise bounds on the  exponential tracking problem.  
This iterative bootstrapping of Gronwall's inequality is described in Algorithm \ref{alg:BootStrap}. 
The approach is quite versatile, and we apply the same procedure several times in different scenarios. 
A general description for where this approach can be taken is described in Algorithm \ref{alg:GeneralBootStrap}.  

In Section \ref{sec:LyapunovPerron} we discuss the Lyapunov-Perron Operator $\Psi$, which is given in Definition  \ref{def:LyapunovPerron}. 
We formulate conditions for when $\Psi$ maps a ball of $C^{0,1}$ functions into itself in 
Theorem \ref{prop:Endomorphism}, and 
for when $\Psi$ maps a ball of $ C^{1,1}$ functions into itself in 
Theorem \ref{prop:LambdaC2Bounds}.

In Section \ref{sec:Contraction} we obtain the necessary estimates to show that the Lyapunov-Perron Operator is a contraction mapping. 
In Definition \ref{def:SemiNorm} we define a norm in which we wish to prove we have a contraction mapping. 
We then give conditions for when we have a contraction in Theorem \ref{prop:Contraction},  and the results of Sections \ref{sec:ExponentialTracking}--\ref{sec:Contraction} are summarized in  Theorem  \ref{prop:UniqueStableManifold}.   

In Section \ref{sec:SH_Linear} we apply our results to the Swift-Hohenberg equation, obtaining the appropriate estimates for a linear change of variables at a nonlinear equilibrium. 
Finally in Section \ref{sec:SH_Nonlinear} we discuss how to get the estimates to work using a nonlinear change of coordinates at a nontrivial equilibrium. 
Computer assisted proofs of a stable manifold theorem using a linear approximation and a nonlinear approximation are given in Theorem \ref{prop:CAP_Linear} and Theorem \ref{prop:NonlinearComputerAssistedProof} respectively, and the source code is available online \cite{StableManifoldCode}.


\section{Background and Notation}
\label{sec:Notation}

A useful first step in studying stable/unstable manifolds 
is to perform a change of coordinates taking the equilibrium to zero 
and aligning the (possible generalized) eigendirections with the coordinate axes.  
For ordinary differential equations (ODEs) such a transformation always 
exists.  Nevertheless, in a particular problem it may be impractical to compute 
this transformation exactly due to the lack of explicit formulas and
the finite numerical precision. 
For PDEs, the situation is even worse, 
as the desired change of coordinates is infinite dimensional.
In the present work we settle for coordinate transformations which move the 
origin approximately to zero, and approximately align the coordinate axes with 
eigendirections.  This is achieved by computing good numerical 
approximations of the equilibrium and the eigendata for a finite dimensional 
Galerkin projection, and approximating the eigendata in
 the infinite dimensional complement
via the linearization of the homogeneous equilibrium.  
To obtain mathematically rigorous results it is necessary to quantify these 
 errors, and  formalizing this discussion requires 
a good deal of notation.  


\subsection{Parabolic PDEs and Semigroup Operators} 
\label{sec:background_PDE}

Let $X$ be a Banach space with norm $|\cdot|=|\cdot|_X$, 
and consider the differential equation
\begin{align}
	\dot{\xx} = \tilde{\Lambda} \xx
	 + \tilde{\cN} (\xx) ,
	\label{eq:PreOriginalODE}
\end{align}
where $\tilde{\Lambda}: \mbox{Dom}(\tilde{\Lambda}) \subseteq  X \to X$ is a densely defined 
linear operator with bounded inverse, 
and  $ \tilde{\cN} \in C^2_{\text{loc}} ( X , X)$.
We will need explicit bounds on $D \tilde{\cN}(0)$ and a 
local (uniform) bound on the second derivative(s).  
See Proposition~\ref{prop:Initial_Constant_Bounds} below. 
Assume that  $ \tilde{h} \in X$ is a hyperbolic equilibrium solution of
Equation \eqref{eq:PreOriginalODE}, where we think of $\tilde{h}$ as being small.
Making the change of variables $\xx \to \xx + \tilde{h}$ leads to the 
differentail equation 
\begin{align}
\dot{\xx} &= \Lambda \xx + L \xx + \hat{\cN}(\xx) .
\label{eq:OriginalODE}
\end{align} 
where
%
\begin{align}
\Lambda  &:= \tilde{\Lambda}, 
&
L &:= 
D \tilde{\cN}(\tilde{ h }),
&
\hat{\cN}(\xx) &:= 
 \tilde{\cN}(\tilde{h}+\xx) - \tilde{\cN}(\tilde{ h }) - D \tilde{\cN}(\tilde{ h })\xx .
\label{eq:ImplicitNonlinearity}
\end{align}
Equation \eqref{eq:OriginalODE}
 has that the origin is an equilibrium solution 
 and that  $\hat{\cN}(0)=0$ and $ D \hat{\cN}(0)=0$.  


\begin{definition}[Stable and unstable decomposition] \label{def:indiexSetSpaces}
Let $X = X_s \times X_u$ denote the decomposition of 
$X$ into stable and unstable eigenspaces of the operator $\Lambda$.
Fix integers $m_s, m_u \in \N $, and define two index 
sets $I := \{ 1, 2, \dots , m_s\}$ and $I' := \{ 1', 2', \dots , m_u'\}$. 
For $i \in I$ and $ i' \in I'$, assume that
$ X_i \subseteq X_s$ and $ X_{i'} \subseteq X_u$ 
are closed subspaces of $X$ with:  
\begin{align*}
	X_s &:= \prod_{1 \leq i \leq m_s} X_i ,&
	X_u &:= \prod_{1' \leq i' \leq m_u'} X_{i'}  .
\end{align*}  
\end{definition}

\begin{remark}[primed and un-primed indices] \label{rem:primes}
Throughout the paper
we use a primed notation, such as $i'$ or $j'$, to index over
the unstable eigenspace $X_u$ and un-primed indices for the 
stable. It is sometimes convenient to have an index ranging over all 
stable and unstable indices, so we define 
$\bI  := I \cup I'$ and write $\bi \in \bI$ to signify that 
$\bi$ may be a primed or un-primed index.
\end{remark}
 
For the projections onto the subspaces $X_i$, $X_{i'}$, $X_s$ and $X_u$ we use the notation $\pi_i$, $\pi_{i'}$, $\pi_s$ and $\pi_u$, respectively. 
Since these subspaces are closed, the projection maps are  
bounded linear operators.  That is, there exist constants $ p_s$, $p_u$, $p_\bi < \infty$ 
so that  
\begin{align} \label{eq:ProjectionBound}
\| \pi_s \| & \leq p_s 
&
\| \pi_u \| & \leq p_u 
&
\| \pi_\bi \| & \leq p_\bi .
&
\end{align}
We use the notation, $\xx_{\bi} = \pi_{\bi} \xx$, $\xx_s =\pi_s \xx$, etc,
hence $\xx = \xx_s + \xx_u$, $\xx_s = \sum_{i\in I} \xx_i$
 and $ \xx_u = \sum_{i' \in I' } \xx_{i'}$, as well as $\xx= \sum_{\both{i} \in \both{I}} \xx_{\both{i}}$.

Assume that $ \Lambda$ is invariant along the subspaces $ X_i$, $X_{i'}$.  
That is to say, assume that there exist $\Lambda_{i} : X_i \to X_i$ and $ \Lambda_{i'} : X_{i'} \to X_{i'}$ such that 
\[
\Lambda \xx = \sum_{i \in I} \Lambda_i \xx_i + \sum_{i' \in I'} \Lambda_{i'} \xx_{i'} .
\]
Furthermore, assume there are constants $\lambda_i <0$ such that for $1\leq i \leq m_s$
\begin{align}
|e^{\Lambda_i     t} \xx_{i} | &\leq  e^{\lambda_{ i} t } |\xx_{i}|,  &
t \geq 0&, \xx_{i} \in X_i ,
\label{eq:stableEigenvalueEstimate}
\end{align}
and $\lambda_{i'} >0$ such that for $1' \leq i' \leq m_u'$
\begin{align}
|e^{\Lambda_{i'} t} \xx_{i'} | &\leq  e^{\lambda_{ i'} t } |\xx_{i'}|,  &
t \leq 0&, \xx_{i'} \in X_{i'}.
\label{eq:UnstableEigenvalueEstimate}
\end{align}
In particular, this implies that the norm on $X$ aligns well with flow of $\Lambda$ on the subspaces $X_i$ in the sense that the vector field $\Lambda_i$ points inwards on the boundary of the unit ball in $X_i$. 

The linear operator $L$ is decomposed in the following manner: 
for all $\both{i},\both{j} \in \both{I}$, define the bounded linear 
operators $L_{\both{i}}^{\both{j}} : X_{\both{j}} \to X_{\both{i}}$ by 
\begin{align*}
[L \xx]_{\both{i}}  &= \sum_{\both{j} \in \both{I}} L_{\both{i}}^{\both{j}} \xx_{\both{j}}.
\end{align*}
Restricting $ \Lambda$ and $L$ to $X_s$ and~$X_u$ gives operators
\begin{align*}
\Lambda_s \xx_s &: X_s \to X_s &  
L_s^s \xx_s &: X_s \to X_s& 
L_s^u \xx_u &: X_u \to X_s\\
\Lambda_u \xx_u &: X_u \to X_u &  
L_u^s \xx_s &: X_s \to X_u &  
L_u^u \xx_u &: X_u \to X_u &  
\end{align*}
defined by 
\begin{align*}
 \Lambda_s \xx_s &:= \sum_{i \in I} \Lambda_i \xx_i &  
  L_s^s \xx_s &:= \sum_{i,j \in I} L_i^{j} \xx_j & 
  L_s^u \xx_u &:= \sum_{i \in I, j'\in I'} L_i^{j'} \xx_{j'} \\
   \Lambda_u \xx_u &:= \sum_{i' \in I'} \Lambda_{i'} \xx_{i'} &  
  L_u^s \xx_s &:= \sum_{i' \in I, j \in I} L_{i'}^{j} \xx_j & 
  L_u^u \xx_u &:= \sum_{i' \in I', j'\in I'} L_{i'}^{j'} \xx_{j'}  .
\end{align*}
Assume that $ -( \Lambda_u + L_u^u)$ and $(\Lambda_s + L_s^s)$ are negative operators,
 in the sense that there exist constants $ C_s,C_u$ and $ \lambda_s < 0 $ and $ \lambda_u >0$ 
 so that  
\begin{alignat}{2}
|e^{(\Lambda_s + L_s^s)t } \xx_{s} | &\leq  C_s e^{\lambda_s t } |\xx_{s}|,  &\qquad&
t \geq 0, \xx_{s} \in X_s ,
\label{eq:TOTALstableEigenvalueEstimate}  \\
|e^{(\Lambda_u + L_u^u)t} \xx_{u} | &\leq  C_u e^{\lambda_u t } |\xx_{u}|,  &\qquad&
t \leq 0, \xx_{u} \in X_{u}.
\label{eq:TOTALUnstableEigenvalueEstimate}
\end{alignat}
Calculation of these constants is discussed in 
Section \ref{sec:EstimatingEigenvalues}, and an explicit 
example is given in Section~\ref{sec:SH_Linear}. 

\begin{remark}
For both the prime and non-prime spatial indices we 
employ Einstein summation notation, writing
\[
	L_i^{j} \xx_{j} \equiv \sum_{j \in I} L_i^{j} \xx_{j},
	\qquad\text{and}\qquad
	L_i^{j'} \xx_{j'} \equiv \sum_{j' \in I'} L_i^{j'} \xx_{j'}.
\]	  
For other indices, for example sums over $\both{I}=I\cup I'$, we write the summation explicitly.  
\end{remark}

We now project the nonlinear terms into the subspaces just defined, and write  
$ \hat{\cN}_{\both{i}} := \pi_{\both{i}} \circ \hat{\cN}(\xx) $  for $\both{i} \in \both{I}$.  
Then $ \hat{\cN}_s(\xx) := \pi_s \circ \hat{\cN}(\xx)$ and 
$ \hat{\cN}_u(\xx) := \pi_u \circ \hat{\cN}(\xx)$. 
For $\bi \in \bI$ let 
\begin{align}
	 \cN_\bi(\xx_s,\xx_u) &:= 
L_\bi^j \xx_j  +L_\bi^{j'} \xx_{j'} 
	+ \hat{\cN}_\bi(\xx_s,\xx_u).
	\label{eq:N_no_hat_def}
\end{align}
We write 
\begin{align*}
	 \cN_s &:= \sum_{i \in I} \cN_i, & 
	  \cN_u &:= \sum_{i' \in I'} \cN_{i'} ,&
	  \cN &:= \cN_s + \cN_u.
\end{align*}
Equation \eqref{eq:OriginalODE} becomes 
\begin{align}
	\dot{\xx}_i &= \Lambda_i \xx_i + \cN_i ( \xx_s , \xx_u), 
		\label{eq:ProductODE_stab} \\
	\dot{\xx}_{i'} &= \Lambda_{i'} \xx_{i'} + \cN_{i'} ( \xx_s , \xx_u).
	\label{eq:ProductODE_unst}
\end{align}
 

We study functions defined on certain a certain products of 
balls containing the origin in the various subspaces.

\begin{definition} \label{def:Ball}
Fix positive vectors $ r_s \in \R^{m_s}$ and $ r_u \in \R^{m_u}$, and  
define  the closed balls  $B_s(r_s) \subseteq X_s$ and $B_u(r_u) \subseteq X_u$ 
given by 
\begin{align*}
		B_s(r_s) &:= 
	\left\{
	\xx_s\in X_s : 
	|  \xx_i | \leq r_i  \mbox{ for } i \in I  
	\right\} 
	\\
		B_u(r_u) &:= 
	\left\{
	\xx_u\in X_u : 
	|  \xx_{i'} | \leq r_{i'}  \mbox{ for } i' \in I'
	\right\} .
\end{align*}
\end{definition}

When the vectors $ r_s$, $r_u$ are understood, we abbreviate to 
$B_s \equiv B_s(r_s)$ and $ B_u \equiv B_u ( r_u) $. 
Below we define bounds on our nonlinearity $\cN$ over balls of fixed radius. 

\begin{definition} \label{def:derivatives}
	Suppose  $r_s \in \mathbb{R}^{m_s}$ and $r_u \in \mathbb{R}^{m_u}$.
	For $\xx_s \in B_s(r_s)$,  $\xx_u \in B_u (r_u)$ 
	and $\both{i},\both{j} , \both{k} \in \both{I}$ define 
	\begin{align*}
	\cN_{\both{j}}^{\both{i}}( \xx_s, \xx_u)  &:= 
	\frac{\partial}{\partial \xx_{\both{i}}} \cN_{\both{j}}( \xx_s, \xx_u), 
	& 
	\| \cN_{\both{j}}^{\both{i}} \|_{(r_s,r_u)} &:=
	\sup_{\xx_s \in B_s(r_s)} \sup_{\xx_u \in B_u(r_u)}
	\| \cN_\both{j}^\both{i}(   \xx_s, \xx_u) \| \\
	\cN_{\both{j}}^{\both{i}\both{k}}( \xx_s, \xx_u)  &:=
	\frac{\partial^2}{\partial \xx_\both{i} \partial \xx_\both{k}}
	\cN_\both{j}(  \xx_s, \xx_u),
	& 
	\| \cN_{\both{j}}^{\both{i}\both{k}} \|_{(r_s,r_u)} &:=
	\sup_{\xx_s \in B_s(r_s)} \sup_{\xx_u \in B_u(r_u)} \|\cN_\both{j}^{\both{i}\both{k}}(   \xx_s, \xx_u) \|.
	\end{align*}
\end{definition}


\begin{proposition}
\label{prop:Initial_Constant_Bounds}
	 Fix $r_s \in \mathbb{R}^{m_s}$, and $r_u \in \mathbb{R}^{m_u}$, 
	 and suppose  that $ | \tilde{h}_\bi| < \epsilon_\bi$. 
	 Assume that the constants $ \tilde{D}^\bi_\bj $  and $\tilde{C}_\bj^{\bi\bk}$ satisfy
\begin{align*}
\tilde{D}^\bi_\bj   &\geq 	\| \tilde{\cN}_\bj^\bi (0,0) \|, 
	& 
\tilde{C}^{\bi\bk}_\bj 	&\geq \| \tilde{\cN}_\bj^{\bi\bk} \|_{(r_s + \epsilon_s,r_u + \epsilon_u)} .
\end{align*}
  For $ \bi,\bj,\bk \in I \cup I'$   define constants $ \hat{ C }_\bj^\bi, D_\bj^\bi,C_\bj^\bi$, and $C_\bj^{\bi\bk}$ as below:
\begin{align*}
D_{\bj}^{\bi}  & := \tilde{D}_\bj^\bi + \tilde{C}_\bj^{\bi l} \epsilon_l +
\tilde{C}_\bj^{\bi l'} \epsilon_{l'}  ,
&
C_\bj^{\bi\bk} &:= \tilde{C}_\bj^{\bi\bk} \\
\hat{  C}_\bj^\bi &:=  \tilde{C}_\bj^{\bi l} r_l +  \tilde{C}_\bj^{\bi l' } r_{l'} 
&
C_\bj^\bi &:= \hat{  C}_\bj^\bi +  D_\bj^\bi .
\end{align*}
Then for $ L$ and $\hat{\cN} $ defined in \eqref{eq:ImplicitNonlinearity} and $ \cN$ defined in \eqref{eq:N_no_hat_def} we have the bounds 
\begin{subequations}
\label{e:boundsoncN}
	\begin{align}
D_{\bj}^{\bi} &\geq \| L_\bj^\bi\| 
&
C_{\bj}^{\bi\bk} &\geq \|  \cN_\bj^{\bi\bk}\|_{(r_s,r_u)} \\
\hat{C}_{\bj}^\bi &\geq \| \hat{\cN}_\bj^{\bi}\|_{(r_s,r_u)}
&
	C_{\bj}^\bi &\geq \| \cN_\bj^{\bi}\|_{(r_s,r_u)}  .
	\end{align}
	\end{subequations}
\end{proposition}
\noindent The proof follows directly from the definitions.  
 
\subsection{Regularity of the candidate functions
}

Our goal is to find a chart $\alpha : B_s  \to X_u$ such that the graph
$\{ ( \xi, \alpha (\xi) ) : \xi \in B_s  \}$  
is a local stable manifold attached to 
the origin of the differential equation \eqref{eq:OriginalODE}.
The desired chart is formulated as a fixed point of the Lyapunov-Perron operator 
in Section \ref{sec:LP_Overview}. 
In preparation for that formulation we now specify 
the appropriate spaces of candidate functions.


%
%
%

\begin{remark}
In Section~\ref{sec:background_PDE} there is notational symmetry between the stable and unstable eigenspaces. For the \emph{stable} manifold the main parameter is the stable radius $r_s$, which determines the domain of the chart $\alpha$. On the other hand, the unstable radius $r_u$ in the codomain of $\alpha$ follows from a Lipschitz assumption on the chart. To highlight this distinction, in the contexts of the Lyapunov-Perron operators and the associated charts we denote the radius in the stable subspace by the parameter $\rho$.
\end{remark}

Let $\rho \in \mathbb{R}^{m_s}$ and $\alpha \in C^0(B_s(\rho), X_u)$.
Define the Lipschitz constants of  $\alpha$ relative to the subspaces $ X_i$ and $X_{i'}$ by
\begin{align*}
\Lip( \alpha  )_{i'}^{i} := \sup_{\xi \in B_s} \sup_{\substack{  0 \neq \zeta_i \in X_i \\ \xi+\zeta_i \in B_s }} 
\frac{| \alpha_{i'}(\xi+\zeta_i) - \alpha_{i'}(\xi) | }{|\zeta_i|}.
\end{align*}
Observe that if $\alpha$ is Fr\'{e}chet differentiable, then
$ \sup_{\xi \in B_s(\rho)} \| \alpha_{i'}^i(\xi) \| = \Lip( \alpha  )_{i'}^{i}$. 
Here we employ the notation of Definition \ref{def:derivatives}, so that
superscripts attached directly to $\alpha$ denote partial derivatives.
Let $C^{0,1}(B_s(\rho), X_u)$ denote the set of all Lipschitz continuous functions 
on $B_s(\rho)$, taking values in $X_u$. Similarly, let 
$C^{1,1}(B_s(\rho), X_u) \subset C^{0,1}(B_s(\rho), X_u)$ 
denote the set of all 
continuously differentiable functions whose derivative 
is Lipschitz continuous. 

\begin{definition} \label{def:Ball_of_Functions}
Fix positive tensors $\rho \in \R^{m_s}$,  $P \in \R^{m_s} \otimes \R^{m_u}$   and 
$\PP \in (\R^{m_s})^{\otimes 2} \otimes \R^{m_u}$, and 
define the function spaces 
	\begin{align*}
	\cB_{\rho,P}^{0,1} &:= \{ \alpha \in C^{0,1}(B_s(\rho),X_u) : \alpha(0) = 0, \; \Lip( \alpha  )_{i'}^{i} \leq P_{i'}^i \}  ,
	\\
	\cB_{\rho,P,\PP}^{1,1} &:= \{ \alpha \in C^{1,1}(B_s(\rho),X_u) :
	\alpha(0) = 0,
	\;
	\Lip( \alpha  )_{i'}^{i} \leq P_{i'}^i , 	\;
	\Lip( \partial_i \alpha  )_{i'}^{j} \leq \PP_{i'}^{ij} \}.
	\end{align*} 
\end{definition}

Note that for all  $ \alpha \in \cB_{\rho,P}^{0,1}$ and  $ \xi, \zeta \in B_s$ we have: 
$
| \alpha_{i'}(\xi) - \alpha_{i'}(\zeta)  |   \leq 
 P^i_{i'}
| \xi_i - \zeta_i  | $.  
For a positive vector~$\rho$ and positive tensor $P$, the range of the $ \alpha \in \cB_{\rho,P}^{0,1}$ lies in a ball $B_u(r_u)$ with $r_u$ given by $r_{i'} = P_{i'}^i \rho_i$.  

\begin{definition}
		\label{def:H}
	Let the vector $\rho$ and tensor $P$ be as in Definition~\ref{def:Ball_of_Functions}.  
	Define $r_u$ by $ r_{i'} := P_{i'}^i \rho_{i}$. 
	For constants $C_j^i$, $\hat{C}_j^i$ and $D^i_j$ such that the bounds~\eqref{e:boundsoncN} hold with $r_s=\rho$, 
	 define positive tensors 
	\begin{align*}
	H_{j}^i &:= C_{j}^i + C_j^{i'} P_{i'}^i , & 	
	H_{j'}^i &:= C_{j'}^i + C_{j'}^{i'} P_{i'}^i , &
	\hat{H}_j^i&:=
	\hat{C}_j^i +( \hat{C}_j^{i'} + D_j^{i'})P_{i'}^i ,
	\end{align*}
	and the positive scalar:
	\[
		\hat{\cH} \ \, := 
	\sup_{\alpha \in \cB_{\rho,P}^{0,1}}
	\sup_{\xx_s \in B_s(\rho)}
	\| 
	\tfrac{\partial}{\partial \xx_s}  L_s^u \alpha( \xx_s) +
	\tfrac{\partial}{\partial \xx_s}  
	\hat{\cN}_s(   \xx_s, \alpha(\xx_s))
	\| 
	.
	\]
\end{definition}

The tensor $H$ provides the following bound:  
fix $\rho$, $P$ and 
$\alpha \in \cB_{\rho,P}^{0,1}$, $\xi,\zeta \in B_s(\rho)$. Then for each $\bj \in \bI$ we have
\begin{align}
|\cN_\bj(\xi,\alpha(\xi)) - \cN_\bj(\zeta,\alpha(\zeta)) | \leq   H_\bj^i  | \xi_i - \zeta_i  |.
\label{eq:NHestimate}
\end{align}


\begin{proposition}
	\label{prop:ComputeHhat}
	Fix $ \rho $ and $ P$ as in Definition~\ref{def:H}. 
If the norm on $X$ has $ | \xx | = \sum_{\bi \in \bI} | \xx_\bi |$, then  $
\hat{\cH} \leq \max_{i\in I} \sum_{j \in I}   \hat{H}_j^i $.
\end{proposition}
\begin{proof}
	Fix $ \alpha \in \cB_{\rho,P}^{0,1}$ and $ \xx_s \in B_s(\rho)$. 
Then
	\begin{align*}
\left\| 	\tfrac{\partial}{\partial \xx_i}  L_s^u \alpha( \xx_s)  \right\| 
&= \left\| \sum_{j\in I} 	\tfrac{\partial}{\partial \xx_i}  L_j^{n'} \alpha_{n'}^i ( \xx_s)  \right\| 
\leq \sum_{j\in I} D_j^{n'}  P_{n'}^i  , 
\\
\left\| \tfrac{\partial}{\partial \xx_i} 
\hat{\cN}_s \left( \xx_s,\alpha(\xx_s)\right)  \right\|
&\leq  \left\| 
\sum_{j\in I} 
	\hat{\cN}_j^i(   \xx_s, \alpha(\xx_s)) +
	\hat{\cN}_j^{n'}(   \xx_s, \alpha(\xx_s))\alpha_{n'}^i(\xx_s)
  \right\|  
\leq \sum_{j\in I}  \hat{C}_j^{i}  + \hat{C}_j^{n'}  P_{n'}^i .
\end{align*}
It now follows from the hypothesis on the norm of $X$ 
that $ \| \pi_\bi \| =1$ for all $ \bi \in \bI$. 
Then 
\begin{align*}
\left\| 	\tfrac{\partial}{\partial \xx_s}  L_s^u \alpha( \xx_s)  
+
\tfrac{\partial}{\partial \xx_s} 
\hat{\cN}_s \left( \xx_s,\alpha(\xx_s)\right) 
 \right\| 
 & = 
 \sup_{u \in X_s , |u| =1} 
 \left|  
 \sum_{i \in I}
 \left( 	\tfrac{\partial}{\partial \xx_i}  L_s^u \alpha( \xx_s)  
 +
 \tfrac{\partial}{\partial \xx_i} 
 \hat{\cN}_s \left( \xx_s,\alpha(\xx_s)\right) 
 \right) u_i
 \right|  
 \\
 &\leq 
  \sup_{u \in X_s , |u|=1 }  
 \sum_{i,j\in I}  \left( D_j^{n'}  P_{n'}^i  +   \hat{C}_j^{i}  + \hat{C}_j^{n'}  P_{n'}^i \right) |u_i|.
\end{align*}
In the righthand side of the previous inequality we recognize $ \hat{H}_j^i$. 
Hence
\begin{align}
\sum_{i,j \in I}
\hat{H}_j^i  |u_{i}|
= 
\sum_{i\in I}
\Big(\sum_{j \in I}
\hat{H}_j^i \Big) |u_{i}|
\leq
\sum_{i\in I }
\Big(
 \max_{n\in I} \sum_{j \in I}   \hat{H}_j^n \Big)
|u_{i}|
=
\Big( \max_{i\in I} \sum_{j \in I}   \hat{H}_j^i \Big) \, |u|.
\label{eq:GlobalAPrioriBound}
\end{align}
Taking the $\sup$ over $ u \in X_s, | u | = 1$ gives
\[
\left\| 	\tfrac{\partial}{\partial \xx_s}  L_s^u \alpha( \xx_s)  
+
\tfrac{\partial}{\partial \xx_s} 
\hat{\cN}_s \left( \xx_s,\alpha(\xx_s)\right) 
\right\|  \leq \max_{i\in I} \sum_{j \in I}   \hat{H}_j^i.
\qedhere \]
\end{proof}

\subsection{Overview of the Lyapunov-Perron Approach}
\label{sec:LP_Overview}
Having established the necessary notation, we are prepared to formalize the discussion. 
Namely, we transform the problem of studying the local stable manifold into the problem of 
finding a fixed point of the Lyapunov-Perron operator.  Excellent general 
references on the Lyapunov-Perron approach include books \cite{chicone2006ODE,henry1981geometric,sell2002dynamics}.

This operator is an endomorphism on charts $ \alpha \in \cB_{\rho,P}^{0,1}$.  
Given such an $\alpha$,  define $ x( t , \xi, \alpha)$ to be the solution of 
the projected differential equation
\begin{align}
\dot{\xx}_s &= \Lambda_s \xx_s + \cN_s(\xx_s,\alpha(\xx_s)) ,
\label{eq:ProjectedSystem} 
\end{align}
with initial condition $\xi \in B_s(\rho)$ at time $t=0$.  
In Section \ref{sec:ExponentialTracking} we 
show that if $\Lambda_s$ sufficiently dominates the nonlinearity $\cN_s$, then solutions of the 
projected system \eqref{eq:ProjectedSystem} do not blow up for any $\alpha \in \cB_{\rho,P}^{0,1}$.
In fact, solutions of the projected system approach~$0$ as $t \to \infty$.  

Assuming for the moment this is true, 
consider the pair $ ( x( t, \xi , \alpha),  \alpha ( x( t, \xi , \alpha) )) $.
 If equation  \eqref{eq:ProductODE_unst} is satisfied for all $ i' \in I'$, then by construction equation  \eqref{eq:ProductODE_stab} is satisfied for all $ i \in I$. 
Hence the pair $ ( x( t, \xi , \alpha),  \alpha ( x( t, \xi , \alpha) )) $ is a solution to the full system~\eqref{eq:OriginalODE},   
and moreover the map $ \xi \mapsto ( \xi , \alpha(\xi))$ is a 
chart for a local invariant manifold of the origin.


To find $\alpha$ solving equation  \eqref{eq:ProductODE_unst} 
for all $ i' \in I'$, we exploit the variation of constants formula and defining the Lyapunov-Perron operator.
\begin{definition}
	\label{def:LyapunovPerron}
	Fix a positive vector $\rho \in \R^{m_s}$  and  a positive tensor $P$. 
	The Lyapunov Perron operator 
	$\Psi: \cB_{\rho,P}^{0,1} \to \Lip(B_s(\rho),X_u)$ 
	is given by 
	\begin{equation}
	\Psi[\alpha](\xi) := - \int_0^\infty e^{-\Lambda_u t} \cN_{u}(x(t,\xi,\alpha), \alpha(x(t,\xi,\alpha))) dt,
	\quad \quad \quad \mbox{for all } \alpha  \in  \cB_{\rho,P}^{0,1}.
	\label{eq:LPO}
	\end{equation}
\end{definition}

%

\begin{remark}[Dynamics on the graph of $\alpha$] \label{rem:innerDynamics}
A fixed point of $\Psi$ is a coordinate chart for a local invariant manifold of the origin. 
Showing this is the stable manifold requires an additional argument. 
This is part of the power of the approach, as by modifying the assumptions one can 
study other attached invariant manifolds like center and center-stable manifolds.  
For an example involving computer assisted proofs see \cite{takayasu2019rigorous}.

Let $ \mathbb{E}_s, \mathbb{E}_u  \subseteq X$ denote the stable  and unstable eigenspaces 
of the operator $ \Lambda + L$. 
If either $\dim(X_s)=  \dim( \mathbb{E}_s)   <\infty$ or $\dim(X_u)= \dim( \mathbb{E}_u)  <\infty$,
 then $\alpha = \Psi[\alpha]$ is a chart for a local stable manifold of the origin. 
In practice this is established by correctly counting with multiplicity the finite number of 
stable/unstable eigenvalues of $\Lambda+ L$. 
We consider this case in Sections \ref{sec:SH_Linear} and \ref{sec:SH_Nonlinear}. 

If, on the other hand, both $ \dim( \mathbb{E}_s)  = \infty$ and $  \dim( \mathbb{E}_u)  = \infty$, 
then the desired result is obtained by showing that the family of operators $ \Lambda + s L$ 
does not have any eigenvalues crossing the imaginary axis for $s \in [0,1]$.
This is the approach taken in~\cite{BGLV} and it could be extended
to studying strongly indefinite problems 
 as typically appear in elliptic problems, see e.g.~\cite{cheng2020stable}. 
\end{remark}

In Section \ref{sec:LyapunovPerron} we show that, for an appropriate 
choice of constants, $ \Psi$ is simultaneously an endomorphism on the balls
$ \cB_{\rho,P}^{0,1}$ and $ \cB_{\rho,P,\PP}^{1,1}$.  
In Section \ref{sec:Contraction} we show that $\Psi$ is a contraction
in a $C^0$-like norm (see Definition \ref{def:SemiNorm}) and use the 
Banach Fixed Point Theorem to establish the existence of a unique fixed point.

\subsection{Good Coordinates: Parameterization of Slow Stable Manifolds and Attached Invariant Frame Bundles} \label{sec:goodCoordinates}
%
In this section we describe a method for high order computation
of slow stable manifolds, as well as some attached invariant frame bundles 
describing the stable and unstable directions normal to the slow stable manifold.
Our approach is based on the parameterization method of 
\cite{cabre2003parameterizationI,cabre2003parameterizationII,cabre2005parameterizationIII}, 
and especially on the notion of slow spectral submanifolds discussed in the references
just cited.  See also the works of \cite{MR3541499,MR3562433,MR3672647,MR3832468,MR3800257},
and the book  \cite{haro2016parameterization}.

The theorem below is extracted from the results of
 \cite{cabre2003parameterizationI,cabre2005parameterizationIII}.
The version we state assumes that the eigenvalues are real and have geometric multiplicity 
one.  These assumptions are not necessary, but simplify the presentation.  In 
the applications considered in Section \ref{sec:SH_Nonlinear},
these assumptions have to be checked.
In slight abuse of notation, to align with the existing literature we use $P$ to 
denote the parametrizaton of a slow stable manifold; this should not be confounded with the 
positive tensor denoted by the same symbol in previous subsection.

\begin{theorem}[Slow-stable manifold parameterization] \label{thm:slowParm}
Let $F \colon \mathbb{R}^d \to \mathbb{R}^d$ be a real analytic vector field, 
and $p_0 \in \mathbb{R}^d$
be a hyperbolic equilibrium point whose differential $DF(p_0)$ is diagonalizable.  Let 
$\lambda_1, \ldots, \lambda_d \in \mathbb{R}$ denote the eigenvalues of $DF(p_0)$ and 
suppose that $\lambda_1, \ldots, \lambda_{m_{\slow}}$ with $m_{\slow} < d$
are the slow stable eigenvalues.  Let $\xi_1, \ldots, \xi_{m_{\slow}} \in \mathbb{R}^d$ denote
the associated slow stable eigenvectors.     
Write 
\[
\Lambda_{\slow} = \left(
\begin{array}{ccc}
\lambda_1 & \ldots & 0 \\
\vdots & \ddots & \vdots \\
0 & \ldots & \lambda_{m_{\slow}}
\end{array}
\right),
\qquad\text{and}\quad
\Lambda = \left(
\begin{array}{ccc}
\lambda_1 & \ldots & 0 \\
\vdots & \ddots & \vdots \\
0 & \ldots & \lambda_d
\end{array}
\right),
\]
to denote respectively the $m_{\slow} \times m_{\slow}$
and $d \times d$ matrices of the slow stable eigenvalues and all the
eigenvalues of $DF(p_0)$. 
Suppose that $P \colon [-1, 1]^{m_{\slow}} \to \mathbb{R}^d$
is a smooth solution of the invariance equation 
\begin{equation} \label{eq:invEq}
F(P(\theta)) = DP(\theta) \Lambda_{\slow} \theta,  \quad \quad \quad \quad \quad  \theta \in [-1, 1]^{m_{\slow}},
\end{equation}
subject to the first order constraints 
$P(0) = p_0$  and $\partial_j P(0) = \xi_j$, 
$1 \leq j \leq m_{\slow}$.
Then $P$ parameterizes the $m_{\slow}$ dimensional smooth slow manifold attached
to $p_0$.
\end{theorem}

It follows from the results of \cite{cabre2003parameterizationI} that 
Equation \eqref{eq:invEq} has analytic solution as long as 
for all $(m_1, \ldots, m_{\slow}) \in \mathbb{N}^{m_{\slow}}$ with 
$m_1 + \ldots + m_{\slow} \geq 2$, 
the non-resonance conditions
$m_1 \lambda_1 + \ldots + m_{\slow} \lambda_{m_{\slow}} \neq \lambda_j$ for
$1 \leq j \leq d$, 
are satisfied.  Observe that this reduces to a finite number of conditions.  
Moreover, the solution is unique up to the choice of the scalings of the eigenvectors
$\xi_1, \ldots, \xi_{m_{\slow}}$.

To control the fast dynamics we  exploit the 
``slow manifold Floquet theory'' developed in \cite{MR3541499}.
The idea is to study certain linearized invariance equations describing the 
stable/unstable bundles attached to the slow stable manifold. These invariant 
bundles describe the linear approximation of the full stable manifold near the slow stable manifold, and in addition 
they provide control over the normal and tangent directions.  
Combining the stable, unstable, and tangent bundles provides a frame bundle for the 
phase space in a tubular region surrounding the slow manifold -- the ``good coordinates''
exploited in Section \ref{sec:SH_Nonlinear}.
The idea is illustrated in Figure \ref{fig:slowBundles}.

Computation of the invariant frame bundles is facilitated by the following theorem, 
the main result of \cite{MR3541499}.  Note that we apply this theorem only in a
finite dimensional Galerkin projection of our PDE.  

\begin{theorem}[Slow-stable manifold Floquet normal form] \label{thm:slowFloquet}
Let $F \colon \mathbb{R}^d \to \mathbb{R}^d$, $p_0 \in \mathbb{R}^d$, $DF(p_0)$,
$\lambda_1, \ldots, \lambda_{d}$, $\xi_1, \ldots, \xi_{d}$, $m_{\slow} < d$,
$\Lambda_{\slow}$, $\Lambda$, and $P \colon [-1, 1]^{m_{\slow}} \to \mathbb{R}^d$
be as in Theorem \ref{thm:slowParm}.
Assume that  for $1 \leq j \leq d$
the functions $q_j \colon  [-1, 1]^{m_{\slow}} \to \mathbb{R}^d$
are smooth solutions of the equations
\begin{equation}\label{eq:bundleInvEq}
DF(P(\theta)) q_j(\theta) = \lambda_j q_j(\theta) + D q_j(\theta) \Lambda_{\slow} \theta, 
\end{equation}
for $\theta \in [-1, 1]^{m_{\slow}}$, subject to the constraints
$q_j(0) = \xi_j$.
Let $GL(\mathbb{R}^d)$ denote the collection of all non-singular $d \times d$ 
matrices with real entries.  Define $Q \colon [-1, 1]^{m_{\slow}} \to GL(\mathbb{R}^d)$
by 
\[
Q(\theta) = \left[ q_1(\theta) | \ldots | q_d(\theta)\right].
\] 
Then 
\begin{enumerate}
\item For all $\theta \in [-1, 1]^{m_{\slow}}$ the collection of vectors 
$q_1(\theta)$, $\ldots$, $q_d(\theta)$ span $\mathbb{R}^d$.  That is, $Q$ takes values in 
$GL(\mathbb{R}^d)$ and hence parameterizes a frame bundle.
\item For all $t \geq 0$ and for all $\theta \in [-1, 1]^{m_{\slow}}$, the derivative of the flow along 
the slow stable manifold factors as 
\begin{equation}\label{e:MQ}
M(t) = Q(e^{\Lambda_{\slow} t}  \theta) e^{\Lambda t} Q^{-1}(\theta),
\end{equation}    
where $M(t)$ is the solution of the equation of first variation for $F$ along $P(\theta)$:
\[
M'(t) = DF(P(\theta)) M(t), \qquad \text{for all } t \geq 0 ,
\]   
with $M(0)$ the identity matrix.
\end{enumerate}
\end{theorem}

\begin{figure}[t!]
	\centering
	\includegraphics[width=0.6\linewidth]{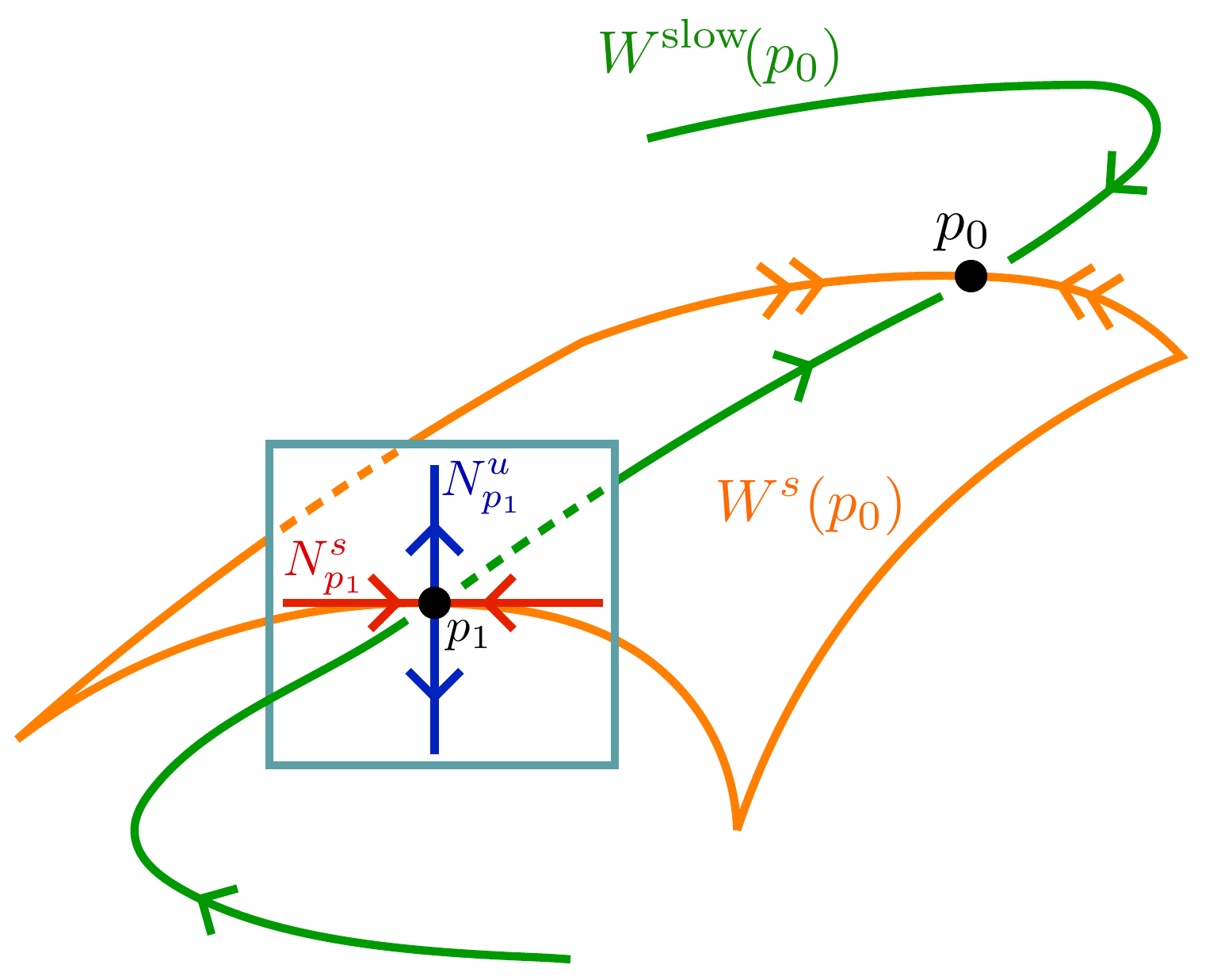}
	\caption{\textbf{Slow stable manifold and attached frame bundles:} the figure
	illustrates an equilibrium solution $p_0$ and its slow stable manifold in green.
	The orange surface illustrates the full stable manifold, of which the slow manifold
	is a submanifold.  At each point on the slow manifold there are invariant 
	stable/unstable normal bundles.  The stable normal bundle describes the 
	stable manifold of $p_0$ near $W^{\slow}$.  Taking the stable, unstable, 
	and tangent bundles gives a frame for the entire space.  
	Theorem \ref{thm:slowFloquet} provides an explicit method for computing these
	structures.  }
	\label{fig:slowBundles}
\end{figure}

Considering~\eqref{e:MQ} 
one column at a time gives that the frame 
bundles $q(\theta)_j$, $1 \leq j \leq d$ satisfy the invariance equation
\[
M(t) q_j(\theta) = e^{\lambda_j t} q_j \left(e^{\Lambda_{\slow} t} \theta \right), \qquad\text{for } \theta \in [-1, 1]^{m_{\slow}}.
\] 
This says that the flow along $P(\theta)$ leaves the direction of $q_j$ invariant (maps the bundle into 
itself) but expands vectors at an exponential rate of $\lambda_j$.  
It follows that if  $q_{m_{\slow}+1}(\theta), \ldots, q_{m_s}(\theta)$ are the parameterized vector 
bundles associated with the stable eigenvalues which have not been designated as slow (the so called
\textit{fast stable} directions), then for each $\theta \in  [-1, 1]^{m_{\slow}}$ these invariant bundles are 
the fastest contracting directions near $P(\theta)$, and hence they describe $W^s(p_0)$ near $P(\theta)$.

We now define a nonlinear change of coordinates which, to first order, 
diagonalizes the vector field $F$ near $P(\theta)$.
Let $d = m_{\slow} + m_{\fast} + m_{\unst}$.  
Define the coordinate change $K \colon [-1, 1]^{m_{\slow}} \times [- \epsilon_f, \epsilon_f]^{m_{\fast}}
\times [-\epsilon_u, \epsilon_u]^{m_{\unst}}
 \to \mathbb{R}^d$  
by 
\[
K(\theta, \phi_f, \phi_u) := P(\theta) + Q_f(\theta) \phi_f + Q_u(\theta) \phi_u,
\]
i.e.\ $K$ is a diffeomorphism with $K(0,0,0)= p_0$
and $DK(0,0,0)=Q(0)$, the matrix of eigenvectors.
Here $\theta$ is the coordinate in the slow stable manifold,
$Q_f$ and $\phi_f$ denote the fast stable directions, and $Q_u$ and $\phi_u$ denote the unstable directions.
Recall that the defining relations for $P$, $Q_f$ and $Q_u$ are
\begin{align}
	F(P(\theta)) &= DP(\theta) \Lambda_{\slow} \theta , \label{e:defrelP} \\
	DF(P(\theta))Q_f(\theta) &= DQ_f(\theta) \Lambda_{\slow} \theta + Q_f(\theta) \Lambda_{\fast} , \label{e:defrelQs} \\
	DF(P(\theta))Q_u(\theta) &= DQ_u(\theta) \Lambda_{\slow} \theta + Q_u(\theta) \Lambda_{\unst} .  \label{e:defrelQu} 
\end{align}

We use $K$ to pull back the vector field $F \colon \mathbb{R}^d \to \mathbb{R}^d$,
resulting in
\[
\left(
\begin{array}{c}
\theta' \\
\phi_f' \\
\phi_u'
\end{array}
\right)
=  DK^{-1}(\theta, \phi_f, \phi_u) \, F(K(\theta, \phi_f, \phi_u)) =  
\left(
\begin{array}{c}
\Lambda_{\slow} \theta + N_{\theta}(\theta, \phi_f, \phi_u) \\
\Lambda_{\fast} \phi_f + N_{\phi_f}(\theta, \phi_f, \phi_u) \\
\Lambda_{\unst} \phi_u + N_{\phi_u}(\theta, \phi_f, \phi_u)
\end{array}
\right),
\]
where each of the $N_k(\theta,\phi_f,\phi_u)$ is \emph{quadratic in $\phi_f$ and $\phi_u$},
for $k=\theta,\phi_f,\phi_u$.

To see this, and to obtain explicitly the form of $N_k$, expanding about $P(\theta)$ results in  
\begin{align}
F(K(\theta, \phi_f, \phi_u)) &= F(P(\theta) + Q_f(\theta) \phi_f + Q_u(\theta) \phi_u) \nonumber \\
&= F(P(\theta)) + DF(P(\theta)) \left[ Q_f(\theta) \phi_f + Q_u(\theta) \phi_u \right]
+ R(\theta,\phi_f,\phi_u) , \label{eq:fKexpanded}
\end{align}
where the remainder term $R$ is  quadratic in $\phi_f$ and $\phi_u$. 
For the first two terms in~\eqref{eq:fKexpanded} we use the defining 
relations for $P$, $Q_f$ and $Q_u$ as well as the definition of $K$ to rewrite
\begin{align*}
	F(P(\theta)) + DF(P(\theta)) \left[ Q_f(\theta) \phi_f + Q_u(\theta) \phi_u \right] 
&= DP(\theta) \Lambda_{\slow} \theta \\
& \qquad
+ DQ_f(\theta) (\Lambda_{\slow} \theta, \phi_f) + 
Q_f(\theta) \Lambda_{\fast} \phi_f \nonumber \\
& \qquad + DQ_u(\theta) (\Lambda_{\slow} \theta, \phi_u) +  
Q_u(\theta) \Lambda_{\unst} \phi_u  \label{eq:form1} \\
& = DK(\theta, \phi_f, \phi_u) 
\left(
\begin{array}{c}
\Lambda_{\slow} \theta  \\
\Lambda_{\fast} \phi_f \\
\Lambda_{\unst} \phi_u
\end{array}
\right). 
\end{align*} 
Then 
\[
DK^{-1}(\theta, \phi_f, \phi_u) \, F(K(\theta, \phi_f, \phi_u)) = 
\left(
\begin{array}{c}
\Lambda_{\slow} \theta  \\
\Lambda_{\fast} \phi_f \\
\Lambda_{\unst} \phi_u
\end{array}
\right)
+ DK^{-1}(\theta, \phi_f, \phi_u) \, R(\theta,\phi_f,\phi_u),
\] 
hence
\[
N(\theta, \phi_f, \phi_u) = 
DK(\theta, \phi_f, \phi_u)^{-1} R(\theta,\phi_f,\phi_u),
\]
As $R$ is quadratic in $\phi_f$ and $\phi_u$, so is $N$.  
Once again we refer to Figure \ref{fig:slowBundles} for the geometric 
interpretation of the coordinate change.

Note that the invariance equation \eqref{eq:invEq} and the 
invariant bundle equations \eqref{eq:bundleInvEq} do not have to be solved exactly.  
Given any approximate solutions, defects are defined by 
considering the invariance equations defining the objects. 
The numerical approximations exploit formal power series methods which have been 
discussed in many places.  In particular, we use the numerical schemes discussed in 
\cite{MR3541499} freely throughout Section \ref{sec:SH_Nonlinear}.

\section{Exponential Tracking}
\label{sec:ExponentialTracking}

\begin{remark}
Throughout this section,   $ \rho \in \R^{m_s}$ denotes a positive vector 
(the radius of the domain of the local stable manifold chart candidates)
and $ P \in \R^{m_s} \otimes \R^{m_u}$ denotes a positive tensor 
(bounding the subspace-Lipschitz constants of our charts). 
\end{remark}

To begin the analysis we first derive estimates on $x(t,\xi,\alpha)$, the solution of the projected system~\eqref{eq:ProjectedSystem}.   
\begin{proposition}
	\label{prop:InitialBound}
Let $ \xi,\zeta \in B_s (\rho)$.
	If $x(t,\xi,\alpha) $ and $ x(t,\zeta,\alpha) $ stay inside $B_s$ for all $ t \in [0,T]$,  then
	\[
	|  x(t,\xi,\alpha) - x(t,\zeta,\alpha)  | \leq  C_s |\xi -\zeta|  e^{(\lambda_{s} +C_s \hat{\cH}) t} 
	\qquad\text{for all }  t \in [ 0, T].
	\]
\end{proposition}

\begin{proof}
	Recall from \eqref{eq:ProjectedSystem} that
	\[
	\dot{\xx}_s= \Lambda_s \xx_s + L_s^s \xx_s + L_s^u \alpha(\xx_s) + \hat{\cN}_s(\xx_s,\alpha(\xx_s))  .
	\]
	Define $ x(t) = x(t,\xi ,\alpha)$ and $ z(t) = x(t, \zeta,\alpha)$. 
	By variation of constants, we have that 
	\[
	x(t) = e^{(\Lambda_s +L_s^s) t} \xi + \int_0^t e^{(\Lambda_s +L_s^s)(t-\tau)} 
	\left( L_s^u \alpha(x(\tau)) +
	\hat{\cN}_s \left(x(\tau),\alpha(x(\tau))\right) \right) d\tau .
	\]	
	From  \eqref{eq:TOTALstableEigenvalueEstimate},
	 we have that $|e^{(\Lambda_s +L_s^s) t} \xi_s| \leq C_s |e^{\lambda_s t} \xi_s | $. 
         Let $ U(t) = |x(t) - z(t)|$, so that  
	\begin{align}
	e^{-\lambda_s t } U(t) &\leq C_s  |\xi -\zeta|  + \int_0^t C_s e^{- \lambda_s \tau} 
	\left|L_s^u ( \alpha(x(\tau)) -\alpha(z(\tau))) \right| d\tau \nonumber\\
	& \hspace{3cm} + \int_0^t C_s e^{- \lambda_s \tau} \left|
	\hat{\cN}_s \left(x(\tau),\alpha(x(\tau))\right) -
	\hat{\cN}_s \left(z(\tau),\alpha(z(\tau))\right)\right|d\tau .
	\label{e:gronwallproof}
	\end{align} 
	Recall from Definition~\ref{def:H} the definition of $ \hat{\cH}$.
	Applying the mean value theorem gives
	\begin{align*}
\left|L_s^u ( \alpha(x(\tau)) -\alpha(z(\tau))) \right| +    \left|
\hat{\cN}_s \left(x(\tau),\alpha(x(\tau))\right) -
\hat{\cN}_s \left(z(\tau),\alpha(z(\tau))\right)\right| 
&\leq \hat{\cH} | x(\tau) - z(\tau)|  .
	\end{align*}
%
%
Plugging this bound into~\eqref{e:gronwallproof} gives 
	\begin{equation*}
	e^{-\lambda_{s} t} U(t)
	\leq 
	C_s |\xi -\zeta| 
	+ \int_0^t   C_s
	\hat{\cH}  e^{-\lambda_{s} \tau}   U(\tau) d\tau.
	\end{equation*}
	By Gronwall's inequality, it follows that
	\(
	e^{-\lambda_{s} t} U(t) \leq C_s |\xi -\zeta|  \exp \{ C_s \hat{\cH} t \}
	\), 
	which we rewrite as 
	\[
	U(t) \leq C_s |\xi -\zeta|  e^{(\lambda_{s} +C_s \hat{\cH}) t} .
	\qedhere
	\]
\end{proof}

From the proof of Proposition \ref{prop:InitialBound},
it is clear that $ \lambda_s + C_s \hat{\cH} <0$ implies the
solution limits to zero. 
Taking $ \zeta =0$, this shows that points in 
$ B_s( \tfrac{1}{C_s} \rho )$ stay in $B_s(\rho)$ for all time.   
A sharper version of Proposition  \ref{prop:InitialBound} 
follows by taking into account the rates in the different subspaces of $ X_s$. 
Consider for example the decomposition $ X_s =   X_{\slow} \times X_{\fast} $ 
and the initial condition $ \xi = ( \xi_{\slow} , \xi_{\fast} ) \in   X_{\slow} \times X_{\fast}$. 
Solving the linear system, and exploiting the 
bound from \eqref{eq:stableEigenvalueEstimate}, gives that
$ | e^{\Lambda_{\slow} t} \xi_{\slow}| \leq  e^{\lambda_{\slow} t} |\xi_{\slow}| $, and that 
$ | e^{\Lambda_{\fast} t} \xi_{\fast}| \leq  e^{\lambda_{\fast} t} |\xi_{\fast}| $.  
If $0 >  \lambda_{\slow}   \gg  \lambda_{\fast}$, 
we expect that solutions of Equation \eqref{eq:ProjectedSystem} 
have a  component  $x_{\fast}(t,\xi,\alpha)$ that initially decreases very quickly. 

This intuition motivates the definition of 
the characteristic ``control'' rates,  arising from each subspace in the 
stable eigenspace,  by which solutions to \eqref{eq:ProjectedSystem} grow/shrink. 
The effect of coupling the various subspaces together
is controlled by the constant 
$ \gamma_0 = \lambda_s + C_s \hat{  \cH}$, the exponent derived in Proposition \ref{prop:InitialBound}. 

\begin{definition}
	\label{def:Gamma}
	For integers $0 \leq k \leq m_s $,
	define constants $ \gamma_k $  (control rates) as 
	\begin{equation*}
	\gamma_k := 
	\begin{cases}
	\lambda_s + C_s \hat{\cH}& \mbox{if } k=0 \\
	\lambda_{k} + H_{k}^{k}  & \mbox{otherwise.}
	\end{cases}
	\end{equation*}
 Assume the ordering $\gamma_k > \gamma_{k+1}$. 
\end{definition}
In practice the ordering of $\gamma_k$ is always satisfied by suitably (re)arranging the subspaces $X$.  
The strictness of the ordering indicates that on the balls chosen, the nonlinearities do not spoil the subspace splitting.
Using these exponential rates,  we estimate the components of $| x(t,\xi,\alpha) | $ using tensors $G_{j,k}^{n}$ defined as follows. 
\begin{condition}
	\label{cond:G_bounding_function} 
	A tensor $ G \in (\R^{m_s})^{\otimes 2}  \otimes \R^{m_s+1} $ satisfies Condition \ref{cond:G_bounding_function} on the interval $ [0,T]$ if:
\begin{equation}\label{e:condG}
|  x_j(t,\xi,\alpha)  - x_j(t,\zeta,\alpha)  |  
 \leq \sum_{\substack{n\in I \\ 0 \leq k \leq m_s}} e^{\gamma_k t} G_{j,k}^n | \xi_n - \zeta _n|  ,
\end{equation}
for all $ t \in [0,T]$, all $ \xi , \zeta \in B_s(\rho) $ and all $ \alpha \in \cB_{\rho,P}^{0,1}$. 
\end{condition}

\begin{remark}\label{remark:startingpoint}
Since $|x_j| \leq p_j  |x|$, with $p_j$ defined in~\eqref{eq:ProjectionBound}, by Proposition \ref{prop:InitialBound} the tensor 
\begin{equation*}
 \widehat{G}_{j,k}^n := \begin{cases} p_jC_s  & \text{for } k=0,\\
  0 & \text{for } k \neq 0,
  \end{cases}
\end{equation*}
satisfies Condition \ref{cond:G_bounding_function}. 
\end{remark}
Note that while  this tensor  $\widehat{G}$ is non-negative, a  generic tensor $G$ satisfying  Condition \ref{cond:G_bounding_function} can, and in practice will, have negative components.

Additionally, we remark that while this estimate is typically initially 
worse than the bound given by Proposition \ref{prop:InitialBound}, 
an explicit bootstrapping argument allows us to obtain tighter 
component-wise bounds on solutions of Equation \eqref{eq:ProjectedSystem}. 
The bootstrapping argument
applies variation of constants to  Equation \eqref{eq:ProjectedSystem} 
in each subspace, focusing on improving the bound one component at a time. 
To begin, we first prove the following proposition.  
\begin{proposition}
	\label{prop:InitialClaim}
	Let $\alpha \in \cB_{\rho,P}^{0,1}$ and $ \xi,\zeta \in B_s(\rho)$.
	Define $
	u_{i}(t) :=  |   x_i(t,\xi,\alpha)  - x_i(t,\zeta,\alpha)  |$ for $i\in I$. 
	If $x(t, \xi , \alpha) , x(t, \zeta , \alpha) \in B_s(\rho)$ for $ t \in [0,T]$, then for each $j \in I$ and all $ t \in [0,T]$ we have 
	\begin{equation}
	\label{eq:InitialInequality}
	e^{-\lambda_j t} u_j(t) \leq |\xi_j - \zeta_j| 
	+ \int_0^t e^{-\lambda_j \tau} \sum_{i\in I}  H_j^i   u_{i}(\tau) d\tau .
	\end{equation}
\end{proposition} 
\begin{proof} 
	By variation of constants
	\[
	x_j(t,\xi,\alpha) = e^{\Lambda_j t} \xi_j + \int_0^t e^{\Lambda_j (t-\tau)} \cN_j \left(x(\tau,\xi,\alpha),\alpha(x(\tau,\xi,\alpha))\right) d\tau .
	\]
	Then
	\begin{align*}
	|\cN_j(x(t,\xi,\alpha),\alpha(x(t,\xi,\alpha))) - \cN_j(x(t,\zeta,\alpha),\alpha(x(t,\zeta,\alpha))) | 
	&\leq  
	H_j^i u_{i}(t)
	\qquad\text{for all } t \geq 0. 	
	\end{align*}
	Together with the estimate $|e^{\Lambda_j t} \xi_j| \leq   e^{\lambda_j t} |\xi_j|$ for $t\geq 0$ 
	this gives
	\begin{align*}
	e^{- \lambda_j t}	u_{j}(t) 
	&\leq   
	| \xi_j - \zeta_j | +
	\int_0^t  e^{- \lambda_j \tau}  
	\sum_{i\in I} H_j^i  u_{i}(\tau)  d\tau. \qedhere
	\end{align*}	
\end{proof}

Given a tensor $G$ satisfying Condition \ref{cond:G_bounding_function}, 
we obtain sharper component-wise estimates by the following theorem.   
 
\begin{theorem}
	\label{thm:LipschitzTracking}
	Let $\alpha \in \cB_{\rho,P}^{0,1}$ and let $ \xi,\zeta \in B_s(\rho)$. 
Suppose  $G$ satisfies Condition \ref{cond:G_bounding_function}, and fix $j \in I$. 
If 	$ G^{n}_{i,j}=0$ for all $ n \in I$ and $i \in I - \{ j\}$,
	then 
	\begin{equation}
	|  x_j(t,\xi,\alpha)  - x_j(t,\zeta,\alpha)  |   \leq | \xi_j-\zeta_j| e^{\gamma_j t} + 
	\sum_{\substack{n, i \in I, i\neq j \\
			0 \leq m \leq m_s, m \neq j }} 
	\frac{e^{\gamma_m t} - e^{\gamma_j t} }{\gamma_m - \gamma_j} H_j^i G_{i,m}^n
	| \xi_n - \zeta_n| .
	\label{eq:UBootBound}
	\end{equation}
	That is,  for $ j \in I$  and $\cT_j : ( \R^{m_s})^{\otimes 2}  \otimes 
	\R^{m_s+1} \to    \R^{m_s} \otimes \R^{m_s+1}$ defined by 	
	\begin{align} \label{eq:Tbasic}
	\left[\cT_j(G ) \right]_{k}^n  &:= 
	\begin{dcases}
	\qquad  \qquad
	\sum_{	n, i\in I, i \neq j }
	(\gamma_k - \gamma_j)^{-1}H_j^i G_{i,k}^{n} 	& \mbox{if } k \neq j , \\
	\delta_{k}^{n} - \sum_{\substack{
			n,i\in I, i \neq j \\
			0\leq m \leq m_s, m \neq j }} 
	(\gamma_m-\gamma_j)^{-1} H_j^i G_{i,m}^n  & \mbox{if } k=j,
	\end{dcases}
	\end{align}
	 replacing $ G_{j,k}^n$ by  $\left[\cT_j(G ) \right]_{k}^n $
	 results in a new tensor $G$ satisfying Condition  \ref{cond:G_bounding_function}.
\end{theorem}

Two lemmas aid in the proof.

\begin{lemma}[see {\cite[p.4]{lakshmikantham1988stability}}]
	Let $u,V,h \in C^0\big([0,\infty), [0,\infty) \big)$
	and suppose that 
	\[
	u(t) \leq V(t) + \int_0^t h(s) u(s) d s .
	\]
	If $V$ is differentiable, then 
	\begin{equation*}
	u(t) \leq  V(0) \exp\left\{ \int_0^t h(s) ds \right\} + \int_0^t V'(s) \exp\left\{ \int_s^t h(\tau) d \tau \right\} ds.
	\end{equation*}
	\label{lem:Gronwall}
\end{lemma}

\begin{lemma}
	\label{prop:Fundamental}
	Fix constants $c_0,c_1,c_2 \in \R$ with $c_1,c_2 \geq 0$ and  define $ \mu_0 = c_0 + c_2$. 
	For constants $\mu_k,a_k$ with $ \mu_k \neq \mu_0$ for $k=1,\dots,K$, 
	we set
	\[
	v(s) =  \sum_{k=1}^K  e^{\mu_k s} a_k.
	\]	
	Suppose that $ v(t) \geq 0 $ for $ t \geq 0$, and assume
	\begin{equation*}
	e^{-c_0 t} u_0 (t) \leq \left( c_1 + \int_0^t e^{-c_0 s} v(s) ds \right) 
	+ \int_0^t c_2 e^{-c_0 s} u_0(s) ds.
	\end{equation*}
	Then 
	\begin{equation}
	u_0(t) \leq c_1 e^{\mu_0 t }  +  
	\sum_{k =1}^K \frac{a_k }{\mu_k -\mu_0} \left( e^{\mu_k t} - e^{\mu_0 t}\right).
	\label{eq:Fundamental}
	\end{equation}
	Furthermore, the sum in the righthand side is non-negative for all $t \geq 0$. 
\end{lemma}
\begin{proof}
	Lemma \ref{lem:Gronwall} gives
	\begin{align}
	e^{-c_0 t} u_0(t) &\leq c_1 e^{c_2 t } + \int_0^t e^{-c_0 s} v(s) e^{c_2 ( t-s)} ds. \nonumber \\
	 &=  c_1 e^{c_2 t } + e^{c_2 t } \int_0^t  \sum_{k=1}^n a_k e^{(\mu_k -c_0 - c_2) s } ds 
	 \nonumber\\
	&=   c_1 e^{c_2 t } +  e^{c_2 t } \sum_{k=1}^n \frac{a_k}{\mu_k - \mu_0 } \left( e^{(\mu_k -\mu_0)t} -1 \right) .
		\label{eq:FundamentalFirst}
	\end{align}	
Multiplying each side by $e^{c_0 t}$ gives the desired inequality \eqref{eq:Fundamental}. 
	Since $v(t)$ is nonnegative, so is the integrand. Hence the sum
	in the righthand side of~\eqref{eq:FundamentalFirst} 
	is non-negative for all $t \geq 0$. 
\end{proof}

\begin{proof}[Proof of Theorem \ref{thm:LipschitzTracking}]
	Fix $j\in J$ and  rewrite \eqref{eq:InitialInequality} as
	\begin{equation}
	e^{-\lambda_j t} u_j(t) \leq | \xi_j - \zeta_j|
	+ \sum_{i \in I, i\neq j} \int_0^t  e^{-\lambda_j s} H_j^i   u_{i}(s) ds 
	+ \int_0^t   e^{-\lambda_j s}  H_j^j   u_{j}(s) ds .
	\label{e:track1}
	\end{equation}
	Since $G$ satisfies Condition \ref{cond:G_bounding_function} we have
	\begin{align}
	\sum_{i \in I, i\neq j}     H_j^i    u_{i}(t) 
	&\leq \sum_{i\in I, i\neq j}     H_j^i   \sum_{\substack{n\in I \\ 0 \leq m \leq m_s}}
	e^{\gamma_m t  } 
	G_{i,m}^{n} 
	\left|  \xi_n - \zeta_n  \right| \nonumber\\
	&= \sum_{0\leq m \leq m_s}  	e^{\gamma_m t  }  \sum_{n,i \in I, i\neq j}    H_j^i  
	G_{i,m}^{n} 
	\left|  \xi_n - \zeta_n  \right|  
	\nonumber\\
		&= \sum_{0\leq m \leq m_s, m \neq j}  	e^{\gamma_m t  }  \sum_{n,i \in I, i\neq j}    H_j^i  
		G_{i,m}^{n} 
		\left|  \xi_n - \zeta_n  \right|,
	 \label{e:track2}
	\end{align}
	where the final equality follows from the assumption that
	$ G^{n}_{i,j}=0$ whenever $i \neq j$.
	Defining 
	\begin{equation*}
	v(s) = \sum_{0\leq m\leq m_s, m\neq j} e^{\gamma_m s} a_m, 
	\qquad\text{with}\qquad a_m := \sum_{n,i \in I, i\neq j}     H_j^i  G_{i,m}^{n} 
	\left|  \xi_n - \zeta_n  \right| ,
	\end{equation*}
	and combining~\eqref{e:track1} with~\eqref{e:track2} leads to 
	\begin{align*}
	e^{-\lambda_j t} u_j(t) & \leq
	 | \xi_j - \zeta_j| +  \int_0^t   e^{-\lambda_j s} \sum_{0\leq m\leq m_s, m\neq j} e^{\gamma_m s} a_m ds + \int_0^t  e^{-\lambda_j s}  H_j^j   u_{j}(s) ds.
	\\
	&= | \xi_j - \zeta_j| +  \int_0^t   e^{-\lambda_j s}   v(s) ds + \int_0^t    H_j^j  e^{-\lambda_j s} u_{j}(s) ds.
	\end{align*} 
	Now apply Lemma \ref{prop:Fundamental} with 
	$u_0=u_j$, $c_0=\lambda_j$, $c_1=| \xi_j - \zeta_j|$, $c_2=H_j^j$. 
	Re-indexing $\{\mu_k\}_{1 \leq k \leq K} = \{\gamma_m\}_{0\leq m\leq m_s, m\neq j}$,
	we see that $\gamma_{m} \neq \lambda_j + H_j^j = \gamma_j$ for $m \neq j$
follows from the strict ordering assumption of Definition~\ref{def:Gamma}.
Then the assumption in Lemma \ref{prop:Fundamental} is satisfied. Applying Lemma \ref{prop:Fundamental} is justified, and leads to the result \eqref{eq:UBootBound}. 
\end{proof}

Theorem~\ref{thm:LipschitzTracking} lets us pick a $j \in I$, and replace a bound of the form~\eqref{e:condG} with the same bound, where 
$G_{j,k}^n$ is replaced by $\left[\cT_j(G ) \right]_{k}^n$, possibly producing a sharper bound. 
Note that in Theorem \ref{thm:LipschitzTracking},  we impose that for a fixed $ j \in I$
 we have $ G_{i,j}^n =0$ for all $ n \in I$ and $ i \in I - j$. 
Without this assumption, we would end up with terms of the form $t e^{\gamma_j t}$ in~\eqref{eq:UBootBound}.  We choose to avoid this, 
as we prefer to work with a finite set of exponentially decaying functions as the basis 
of our estimates. 

However, we  also need to deal with the case $ G_{i,j}^n \neq 0$ for some
$i\neq j$ and some $n\in I$. 
This problem is solved by modifying 
such an ``ill-conditioned'' $G$ 
before replacing it with $\cT_j(G )$. Namely,
if $G_{i,j}^n \neq 0$ then, depending on the sign of $G_{i,j}^n$
we estimate $(G_{i,j}^n)  e^{\gamma_j t}$ from above by either 
$G_{i,j}^n e^{\gamma_{j-1} t}$ or  $G_{i,j}^n  e^{\gamma_{j+1} t}$ for $t \geq 0$.
Here we use the ordering $\gamma_0 > \dots > \gamma_{m_s}$ 
asserted in Definition~\ref{def:Gamma}. 
%
%
%
%
To be precise, for any fixed $j \in I$, define the modified tensor 
\begin{align}
	[\cQ_j(G)]_{i,k}^n &:= 
	\begin{cases}
	0 		&	\mbox{if } k=j \\
	G_{i,k}^n + G_{i,j}^n &	\mbox{if } k=j-1, \mbox{and } G_{i,j}^n >0 \\
	G_{i,k}^n + G_{i,j}^n &	\mbox{if } k=j+1, \mbox{and } G_{i,j}^n <0 \\
	G_{i,k}^n & \mbox{otherwise.}
	\end{cases}		
\end{align}
Note that if  $ j=m_s$ and $G_{i,j}^n < 0$, then we are effectively 
employing the estimate $ G_{i,j}^n  e^{\gamma_{m_s}t} <0$. 

The following lemma summarizes the  preceding discussion. 
\begin{lemma}
Fix $ j \in I$. 
If $ G$ satisfies Condition \ref{cond:G_bounding_function}, then $\cQ_j(G)$ satisfies Condition \ref{cond:G_bounding_function}.
\end{lemma}

Thus, starting from an initial bound of the form~\eqref{e:condG} with tensor $\widehat{G}$ given in Remark~\ref{remark:startingpoint},
we iteratively  improve the bound using the following algorithm.
\begin{algorithm}
\label{alg:BootStrap}
Let $N_{\text{bootstrap}} \in \mathbb{N}$ be a computational parameter.
%
\begin{algorithmic}
	\STATE $G \gets \widehat{G}$
	\FOR {$1 \leq i \leq  N_{\text{bootstrap}}$} 
	\FOR {$1 \leq j \leq m_s $}
	\STATE $G_{j,k}^{n} 	\gets
	 \left[\cT_j \circ \cQ_j (G) \right]_{k}^{n}$
	\ENDFOR
	\ENDFOR 
	\RETURN G 
\end{algorithmic}
\end{algorithm}
In practice 
Algorithm \ref{alg:BootStrap} quickly converges to a fixed tensor $G$.
For example $N_{\text{bootstrap}} \leq 5$ is sufficient for the applications to folllow.

\begin{theorem}
	\label{prop:BootstrapBounds}
	Let $\alpha \in \cB_{\rho,P}^{0,1}$, and suppose that the
	coefficients $ G_{j,k}^n $ are output by Algorithm \ref{alg:BootStrap}.
	 Fix initial conditions $ \xi,\zeta \in B_s(\rho)$.
	If $x(\tau,\xi,\alpha)$ and $ x(\tau,\zeta,\alpha)$ stay inside 
	$ B_s(\rho)$ for all $ t \in [0,T]$, then
	\begin{align}
	| x_j(t,\xi,\alpha) - x_j(t,\zeta,\alpha) |
	&\leq  
	\sum_{\substack{n\in I \\ 0 \leq k \leq m_s}}
	e^{\gamma_k t  } \cdot
	G_{j,k}^{n} 
	\left|   \xi_n - \zeta_n  \right| 
	\qquad \text{for all }  t \in [0,T].
	\label{eq:AlgBound}
	\end{align}
	Furthermore, if $\alpha$ is differentiable then 
	$
	\left\| 
	\tfrac{\partial}{\partial \xi_n}
	x_j(t,\xi,\alpha)
	\right\| 
	\leq 
	\sum_{0 \leq k \leq m_s}
	e^{\gamma_k t  } 
	G_{j,k}^{n}
$ for all $ t \in [0,T]$. 
\end{theorem}

The proof of Theorem \ref{prop:BootstrapBounds} is by induction on $N_{\text{bootstrap}}$, with Proposition \ref{prop:InitialBound} taking care of the base case ($N_{\text{bootstrap}}=0$), and Theorem \ref{thm:LipschitzTracking}  taking care of the inductive step. 
We omit the details. 

Now, in Proposition \ref{prop:InitialBound}  the assumption that $ \gamma_0 <0$ gives 
 only that points $\xi \in B_s(C_s^{-1} \rho )$ have solutions to \eqref{eq:ProjectedSystem} 
staying in $B_s(\rho)$ for all $t \geq 0$. 
The following proposition gives conditions which extend the result to all points $ \xi \in B_s(\rho)$. 



\begin{proposition}
	\label{prop:StayInsideBall}
	Suppose that $\gamma_0 <0$  and that $G_{j,k}^n$  
	is the output of Algorithm~\ref{alg:BootStrap}. 
	If 
	\begin{align}
	\rho_j \geq \sum_{\substack{n\in I \\ 0 \leq k \leq m_s}} e^{\gamma_k t}G_{j,k}^n \rho_n ,
	\label{eq:StayInsideBall}
	\end{align}
	for all $ t\geq 0$, then for all $ \xi \in B_s(\rho)$ and $t \geq 0$ we have $ x( t, \xi, \alpha) \in B_s(\rho)$ for all $ \alpha \in \cB_{\rho,P}^{0,1}$.  
\end{proposition}
\begin{proof}  
	Fix $\alpha \in \cB_{\rho,P}^{0,1}$,  $0 <  \epsilon <1$, and  $\xi \in B_s(\epsilon \rho)$. 
Define  $ T = \sup \{t \geq 0 : x(t,\xi,\alpha) \in B_s(\rho) \}$. 
Assume that $ T  < + \infty$. We show by contradiction that $T = + \infty$.  

Since $ x(0,\xi,\alpha) \in B_s(\epsilon \rho)$ and $ x(t,\xi,\alpha)$ is continuous in $t$, it follows that $ T>0$. 
 	By Proposition \ref{prop:BootstrapBounds} we have for all $ t \in [0,T)$ that 
	\begin{align*}
	| x_j(t,\xi,\alpha) |
	&\leq  
	\sum_{0 \leq k \leq m_s}
	e^{\gamma_k t  }  
	G_{j,k}^{n} 
	\left|   \xi_n    \right|  
	\leq \epsilon
	\sum_{0 \leq k \leq m_s}
	e^{\gamma_k t  }  
	G_{j,k}^{n} \rho_n 
	\leq \epsilon \, \rho_j .
	\end{align*}
	Hence $x(t,\xi,\alpha) \in B_s( \epsilon \rho)$ for all $ t \in [0,T)$, 
	and so by continuity   $x(T,\xi,\alpha) \in B_s(\epsilon  \rho)$. 
	Since $x(T,\xi,\alpha) $ is in the interior of $B_s(\rho)$,  the solution of
	\eqref{eq:ProjectedSystem} starting at $ x(T,\xi,\alpha) $ stays inside the 
	ball $B_s(\rho)$ for some positive amount of time.  
	But this contradicts the definition of $T$ 
	as the supremum of $ \{t \geq 0 : x(t,\xi,\alpha) \in B_s(\rho) \}$. 
	Hence, if $ 0 < \eps <1$ and  $\xi \in B_s(\epsilon \rho)$, then  
	$x(t,\xi,\alpha) \in B_s(  \rho)$ for all  $ t\geq0$.  
	
	By continuity of solutions, 
	this result extends to initial conditions on the boundary of $B_s(\rho)$. 
\end{proof}

\begin{remark}
In practice we verify the hypothesis of Proposition  \ref{prop:StayInsideBall} in three steps: 
	\begin{enumerate}
		\item  For  some $T_2 > 0 $, we check that $\rho_j  >  \sum_{n\in I, 0 \leq k \leq m_s} e^{\gamma_k T_2} |G_{j,k}^n| \rho_n $, and hence \eqref{eq:StayInsideBall} is satisfied for all $ t \geq T_2$. 
		 \item  For some $ 0 < T_1 < T_2$, we use interval arithmetic to verify 
		 the inequality \eqref{eq:StayInsideBall} for $ T_1 \leq t \leq T_2$. 
		\item 
		To prove inequality \eqref{eq:StayInsideBall} for $ t \in [0,T_1]$, we both prove that the inequality  holds at $t=0$ (explained below), and show using interval arithmetic that the derivative of the right-hand side of~\eqref{eq:StayInsideBall} is negative:
		\begin{align*}
		\sum_{\substack{n\in I \\0 \leq k \leq m_s}} \gamma_k  e^{\gamma_k t}G_{j,k}^n \rho_n < 0 
		 \qquad \text{for } t \in [0,T_1].
		\end{align*}
		To prove that inequality \eqref{eq:StayInsideBall} holds at $t=0$, we fix $ j \in I$. 
		If $G$ is the final output of Algorithm~\ref{alg:BootStrap}, then  there is a tensor $ \widetilde{G} \in (\R^{m_s})^{\otimes 2}  \otimes \R^{m_s+1} $    for which $G_{j,k}^n \leftarrow  	 \big[\cT_j \circ \cQ_j (\widetilde{G}) \big]_{k}^{n}  $.  It is assigned at  step $j$ of the inner for-loop of the algorithm, and at step $ N_{\text{bootstrap}}$ of the outer for-loop. 
Letting $\bar{G} := \cQ_j (\widetilde{G})$, it follows 
		%
		from the definition of $\cT_j$ in 
		\eqref{eq:Tbasic}
		that	
\begin{equation*}
\sum_{\substack{n\in I \\ 0 \leq k \leq m_s}}  	e^{\gamma_k t}   G_{j,k}^n | \xi_n  | 
=
| \xi_j| e^{\gamma_j t} + 
\sum_{\substack{n,i \in I, i\neq j \\0 \leq k \leq m_s, k \neq j}} 
\frac{e^{\gamma_k t} - e^{\gamma_j t} }{\gamma_k - \gamma_j} H_j^i \bar{G}_{i,k}^n
| \xi_n  | .
\end{equation*}
		Evaluating at $ t =0 $, we have 
		\[
		|  x_j(0,\xi,\alpha)   |  = | \xi_j| =
		\sum_{0 \leq k \leq m_s}     G_{j,k}^n
		| \xi_n  | .
		\] 
		Taking  $ |\xi_n |=\rho_n$ for all $n \in I$,  it follows that  
		$\rho_j =  \sum_{0 \leq k \leq m_s}  G_{j,k}^{n} \rho_n$. 
		Hence \eqref{eq:StayInsideBall} is satisfied at $ t = 0$ for   all $ j \in I$.

	\end{enumerate}	

\end{remark}
 
 \begin{remark}
 	
When inequality \eqref{eq:StayInsideBall} fails to be true,  we cannot be sure that all solutions of Equation~\eqref{eq:ProjectedSystem} stay inside the ball $B_s(\rho)$ for all time. 
 	There are two common reasons for why this happens:
first, the nonlinearity may be too large and solutions leave the ball never to return; second, solutions to  Equation \eqref{eq:ProjectedSystem} may  temporarily leave the ball,  reenter, and then converge to zero. 
 	
 	If inequality \eqref{eq:StayInsideBall} fails to be true because of the first  reason, then $\rho$ should be made smaller.
 	If inequality \eqref{eq:StayInsideBall} fails to be true because of the second reason, it is often because  $B_s(\rho)$ is too wide in one direction and too thin in another. 
 	If we suspect this to be true, 
 	then to better align the box with the flow,  we iteratively select a new value 
	of $\rho$ using the map   
	$ \rho_j \mapsto   \sup_{0 \leq t \leq T} \sum_{k} e^{\gamma_k t}G_{j,k}^n \rho_n $. 
	In practice, this heuristic is effective for finding a value of $\rho$ for which 
	\eqref{eq:StayInsideBall} is satisfied. 
  	%
	
 \end{remark}
 
Algorithm \ref{alg:BootStrap} can be applied in more general situations. 
The two conditions necessary to construct such an algorithm are Condition \ref{cond:G_bounding_function} and Proposition \ref{prop:InitialClaim}. 
These are all generalized  in Appendix \ref{sec:GeneralBootstrap} 
leading to an algorithm used in Section \ref{sec:SecondDerivative} 
to obtain bounds on $\frac{\partial}{\partial \xi_i} x(t,\xi,\alpha)$, 
and in Section \ref{sec:Contraction} to construct bounds on 
$|x(t,\xi,\alpha) - x(t,\xi,\beta) |$ for charts $ \alpha, \beta \in \cB_{\rho,P}^{0,1}$.   

%
%

\section{Lyapunov-Perron Operator} 
\label{sec:LyapunovPerron}

 
In this section we show that the Lyapunov-Perron operator $ \Psi$ is 
an endomorphism on balls $ \cB_{\rho,P}^{0,1}$ and $ \cB_{\rho,P,\bar{P}}^{1,1}$ for appropriately chosen constants. 

\begin{remark} \label{rem:FixedConstants1}
	Throughout this section, we fix a positive vector $ \rho \in \R^{m_s}$ and a
	 positive tensor $ P \in \R^{m_u} \otimes \R^{m_s}$, and fix $ G \in 
	 (\R^{m_s})^{\otimes 2}  \otimes \R^{m_s+1}$ as the output of Algorithm
	  \ref{alg:BootStrap} taken with $ N_{bootstrap} \geq 1$.  
	Furthermore, we assume that the hypotheses of Proposition 
	\ref{prop:StayInsideBall} are satisfied, and in particular that 
	inequality~\eqref{eq:StayInsideBall}
	holds for all $ t\geq 0$. 
	Hence $G$ satisfies Condition \ref{cond:G_bounding_function} on the 
	interval~$[0,\infty)$. 
\end{remark}

Throughout this section we adopt Einstein summation convention for 
indices of $I$ and $I'$.

%

\subsection{Endomorphism on $\cB_{\rho,P}^{0,1}$}


The next theorem provides a straightforward bound on 
$\Lip(\Psi[\alpha])$ for $\alpha \in \cB_{\rho,P}^{0,1}$.

\begin{theorem}
	\label{prop:Endomorphism}  
	 Define $\tilde{P} \in \R^{m_u} \otimes \R^{m_s}$ component-wise by:
	\begin{equation*}
	\tilde{P}^n_{i'} :=
	\sum_{0 \leq k \leq m_s} (\lambda_{i'}- \gamma_k )^{-1} H_{i'}^i G_{i,k}^{n} .
	\end{equation*}
	If  $\alpha \in \cB_{\rho,P}^{0,1}$, 
	then $\Lip(\Psi[\alpha])_{i'}^n \leq \tilde{P}^n_{i'}$.  
	If $\tilde{P}^j_{j'} \leq P_{j'}^j  $ 
	then $\Psi: \cB_{\rho,P}^{0,1} \to \cB_{\rho,P}^{0,1}$ is well defined. 
\end{theorem}

\begin{proof}
	Fix $\alpha \in  \cB_{\rho,P}^{0,1}$ and
	$\xi , \zeta \in B_s(\rho)$.  Define $ x(t) := x(t,\xi,\alpha)$ and 
	$ z(t) := x(t,\zeta,\alpha)$.  
	Our goal is to prove that 
	$|\Psi[\alpha]_{i'}(\xi) - \Psi[\alpha]_{i'}(\zeta) |  \leq   \tilde{P}^n_{i'} 	
	\left| \xi_n - \zeta_n \right| $.  	
	From the definition of $\Psi$ we have
	\[
	\Psi[\alpha](\xi) - \Psi[\alpha](\zeta)  = 
	- \int_0^\infty e^{-\Lambda_u t} 
	\left [ \cN_{u}(x(t), \alpha(x(t))) - \cN_{u}(z(t),\alpha(z(t))) \right] dt . 
	\]
	Using the bound~\eqref{eq:NHestimate}, and the fact that $G$ satisfies Condition~\ref{cond:G_bounding_function} on $[0,\infty)$, we obtain  
	\begin{align*}
	|\Psi[\alpha]_{i'}(\xi) - \Psi[\alpha]_{i'}(\zeta) |  
	&\leq  \int_0^\infty   e^{-\lambda_{i'} t}  
H_{i'}^i
	| x_i(t)- z_i(t) |  dt \\
	&\leq  
	\int_0^\infty  e^{-\lambda_{i'}t}  
	\sum_{0 \leq k \leq m_s}
	e^{\gamma_k t } H_{i'}^i
	G_{i,k}^{n} 
	\left|  \xi_n - \zeta_n \right|  dt \\
	&=
	\sum_{0 \leq k \leq m_s} (\lambda_{i'}- \gamma_k )^{-1} H_{i'}^i G_{i,k}^{n}
	\left| \xi_n - \zeta_n \right| .  
	\end{align*}
	For  $\tilde{P}^n_{i'}$ as defined above, 
	 it follows that
	\begin{equation*}
		|\Psi[\alpha]_{i'}(\xi) - \Psi[\alpha]_{i'}(\zeta)| \leq   \tilde{P}^n_{i'} 	
		\left| \xi_n - \zeta_n \right|   .
	\end{equation*}
	Hence $\Lip(\Psi[\alpha])_{i'}^n \leq \tilde{P}_{i'}^n$. 
	Since $\cN(0)=0$,  direct evaluation shows that $ \Psi[\alpha](0)=0$.
	 Hence $ \Psi[\alpha] \in \cB_{\rho,P}^{0,1}$. 
\end{proof}

\begin{remark}
	\label{rem:ChoiceOfP}
	Ideally, we would like to choose a tensor $P$ as small as possible while still satisfying the inequality  $\tilde{P}^j_{i'} \leq P^j_{i'}$. 
	In practice, we find a nearly optimal $P$ by iteratively 
	mapping  $P^j_{i'} \mapsto  \tilde{P}^j_{i'}$. 
	This has the effect that if  $\tilde{P}^j_{i'} \leq P^j_{i'}$, 
	then the new value of $P$ will be smaller. 
	Since the bounds for $H$ and $G$ improve with smaller $P$, 
	the inequality $\tilde{P}^j_{i'} \leq P^j_{i'}$ will likely be satisfied for
	 the new $P$.  
	On the other hand, if $P$ is too small and  $\tilde{P}^j_{i'} \leq P^j_{i'}$ is not 
	satisfied, then the new value of $P$ will be larger, and the inequality will hopefully 
	 be satisfied at the next iterate of the algorithm. 
	 
	Note that the definitions of $H$ and $G$  depend on $P$, 
	and so these constants need to be recomputed every time.
	Nevertheless, this iterative process provides an effective, 
	algorithmic method for selecting appropriate $P_{i'}^j$. 
\end{remark}

	 Using  second derivative bounds on $ {\cN}_u$ sharpens
	  Theorem \ref{prop:Endomorphism} as below.  
\begin{proposition}
 	\label{prop:Endomorphism_2}
	 Define $\tilde{P} \in \R^{m_u} \otimes \R^{m_s}$ component-wise by:
	\begin{align*}
	\tilde{P}^n_{i'} &:=
	\left( 
	D_{i'}^i + D_{i'}^{j'} P_{j'}^i 
	\right) 
	\sum_{0 \leq k \leq m_s}
	(\lambda_{i'}-\gamma_k)^{-1}  
	G_{i,k}^{n}  \nonumber
 \\&\qquad + 
	\left( \hat{C}_{i'}^{ij} + \hat{C}_{i'}^{j'j} P_{j'}^i 		\right)  
	\sum_{0 \leq k_1,k_2 \leq m_s}
	(\lambda_{i'}-\gamma_{k_1}-\gamma_{k_2} )^{-1}
	G_{j,k_1}^{m} 
	G_{i,k_2}^{n}  \rho_m   .
	\end{align*}
	If  $\alpha \in \cB_{\rho,P}^{0,1}$, 
	then $\Lip(\Psi[\alpha])_{i'}^n \leq \tilde{P}^n_{i'}$.  
	If $\tilde{P}^j_{j'} \leq P_{j'}^j  $ 
	then $\Psi: \cB_{\rho,P}^{0,1} \to \cB_{\rho,P}^{0,1}$ is well defined. 
\end{proposition}
\begin{proof}
By the mean value theorem we have (recall that $\cN_{i'}^i= \frac{\partial}{\partial \xx_{i}} \cN_{i'}$)
	\begin{align*}
		|\cN_{i'} (x,\alpha(x)) - \cN_{i'}(z,\alpha(z)) | 
		&\leq 
		\left[ \sup_{\substack{y \in B_s(\rho) , \; j \in I\\
			|y_j| \leq \max \{  |x_j|,|z_j|\} }}
		 \| \cN_{i'}^i (y,\alpha(y))\| \right]
		|x_i-z_i| .
		\end{align*}
		We  estimate $\max \{  |x_j(t)|,|z_j(t)|\}$ using the  tensor $G$ (which satisfies Condition~\ref{cond:G_bounding_function}), and since  $\max\{ |\xi_m|,|\zeta_m|\} \leq \rho_m$, we have
		\begin{align*}
 \sup_{\substack{y \in B_s(\rho) ,\; j \in I\\
		|y_j| \leq \max \{  |x_j(t)|,|z_j(t)|\} }}\| \cN_{i'}^i (y,\alpha(y))\|   
		& \leq 
		D_{i'}^i + D_{i'}^{j'} P_{j'}^i + (\hat{C}_{i'}^{ij} + \hat{C}_{i'}^{j'j} P_{j'}^i ) \max\{ |x_j(t)|, |z_j(t)| \}  \\
		&\leq 
		D_{i'}^i + D_{i'}^{j'} P_{j'}^i + (\hat{C}_{i'}^{ij} + \hat{C}_{i'}^{j'j} P_{j'}^i )
		 \sum_{0 \leq k \leq m_s}
		e^{\gamma_k t } 
		G_{j,k}^{m} \rho_m.
		\end{align*}
Using Condition~\ref{cond:G_bounding_function} gives 				
		\begin{align*}
	|\cN_{i'} (x,\alpha(x)) - \cN_{i'}(z,\alpha(z)) | 
		&\leq 
		\left( 
		D_{i'}^i + D_{i'}^{j'} P_{j'}^i 
		\right) 
		 \sum_{0 \leq k \leq m_s}
		e^{\gamma_k t }  
		G_{i,k}^{n} 
		\left|  \xi_n - \zeta_n \right| 
		\\&\qquad + 	
		\left( \hat{C}_{i'}^{ij} + \hat{C}_{i'}^{j'j} P_{j'}^i 		\right)  
		\sum_{0 \leq k_1,k_2 \leq m_s}
		e^{(\gamma_{k_1}+\gamma_{k_2} )t } 
		G_{j,k_1}^{m} 
		G_{i,k_2}^{n} \rho_m
		\left|  \xi_n - \zeta_n \right| .
	\end{align*}
We obtain the desired result by integration:
	\begin{align*}
	|\Psi[\alpha]_{i'} (\xi) - \Psi[\alpha]_{i'} (\zeta) | 
	&\leq  \int_0^\infty   e^{-\lambda_{i'} t}  
	|\cN_{i'} (x,\alpha(x)) - \cN_{i'}(z,\alpha(z)) | \,  dt    \\
	&\leq 
	\left( 
	D_{i'}^i + D_{i'}^{j'} P_{j'}^i 
	\right) 
	\sum_{0 \leq k \leq m_s}
	(\lambda_{i'}-\gamma_k)^{-1}  
	G_{i,k}^{n} 
	\left|  \xi_n - \zeta_n \right| 
	\\&\qquad + 	
	\left( \hat{C}_{i'}^{ij} + \hat{C}_{i'}^{j'j} P_{j'}^i 		\right)  
	\sum_{0 \leq k_1,k_2 \leq m_s}
	(\lambda_{i'}-\gamma_{k_1}-\gamma_{k_2} )^{-1}
	G_{j,k_1}^{m} 
	G_{i,k_2}^{n} \rho_m 
	\left|  \xi_n - \zeta_n \right| .
	\end{align*}

\end{proof}

\subsection{Endomorphism on $\cB_{\rho,P,\bar{P}}^{1,1}$}
\label{sec:SecondDerivative}

We now bound the Lipschitz constant of the derivative of the local stable manifold. 
To do this, we show that $\Psi$ maps  $\cB_{\rho,P,\bar{P}}^{1,1}$, a ball of functions with Lipschitz derivative, into itself. 
Hence, if there are any fixed points 
$\Psi[\alpha] =  \alpha  \in  \cB_{\rho,P,\bar{P}}^{1,1}$, 
then  by Definition \ref{def:Ball_of_Functions} they satisfy
$\Lip( \partial_i \alpha  )_{i'}^{j} \leq \bar{P}_{i'}^{ij}$. 
To show that $ \Psi : \cB_{\rho,P,\bar{P}}^{1,1} \to \cB_{\rho,P,\bar{P}}^{1,1}$ 
we first derive bounds on the difference 
$ \frac{\partial}{\partial \xi_i} x_j(t , \eta, \alpha)
  -  \frac{\partial}{\partial \xi_i} x_j(t , \zeta, \alpha)  $ for $ i,j\in I$. 
 In particular, we are interested in finding a tensor $K$ as follows.

 \begin{condition}
 	\label{cond:SecondDerivative}
 	Define  $ \{\mu_k \}_{k=1}^{N_\mu} = \{ \gamma_k \}_{k=0}^{m_s} \cup \{ \gamma_{k_1} + \gamma_{k_2}  \}_{k_1,k_2=0}^{m_s}$. 
 	A tensor $K \in (\R^{m_s})^{\otimes 3}  \otimes \R^{N_{\mu}} $ is said to satisfy Condition \ref{cond:SecondDerivative} if  
 	\begin{align*}
 	\left\| \frac{\partial}{\partial \xi_i} x_j(t,\eta,\alpha) - \frac{\partial}{\partial \xi_i} x_j(t,\zeta,\alpha) \right\|  \leq \sum_{k=1}^{N_\mu} e^{\mu_k t}  
 	K_{j, k }^{i l} | \eta_{l} - \zeta_{l} | ,
 	\end{align*}
 	for all $ \alpha \in \cB_{\rho,P,\bar{P}}^{1,1}$ and $ \eta,\zeta \in B_s(\rho)$ and $i , j \in I$. 
 \end{condition}

 The bound is obtained  using an approach analogous  to the one
 discussed in Section \ref{sec:ExponentialTracking}.  
 Since we use this approach in Sections \ref{sec:ExponentialTracking}, \ref{sec:LyapunovPerron}, and \ref{sec:Contraction},  we present in Appendix \ref{sec:GeneralBootstrap} a generalization which encompasses all cases. 
 In Proposition \ref{lem:LipschitzDerivativeEst} we define a tensor $S$
  analogous to $H$ given in Definition~\ref{def:H}. 
In Proposition \ref{prop:InitialSecondDerivativeBound} we derive 
an \emph{a priori} bound, constructing an initial  tensor  $K$ satisfying Condition \ref{cond:SecondDerivative} (cf. Proposition \ref{prop:InitialBound}). 
In Proposition \ref{prop:LipDerivative} we derive a system of integral inequalities (cf. Proposition~\ref{prop:InitialClaim} and Condition \ref{cond:GeneralGbounds}).  
Then, as described in Theorem \ref{prop:SecondDerivativeAlgorithmResult},  we apply Algorithm \ref{alg:GeneralBootStrap} (cf. Algorithm \ref{alg:BootStrap}) to bootstrap Gronwall's inequality, and obtain successively sharper tensors $K$ satisfying Condition \ref{cond:SecondDerivative}. 
Finally, in Proposition \ref{prop:LambdaC2Bounds}, we give conditions  guaranteeing that $\Psi : \cB_{\rho,P,\bar{P}}^{1,1} \to  \cB_{\rho,P,\bar{P}}^{1,1}$ is a well defined map. 

\begin{proposition}
	\label{lem:LipschitzDerivativeEst}
	Let
	$ \alpha \in \cB^{1,1}_{\rho,P,\bar{P}}$  and $\eta , \zeta \in B_s(\rho)$. 
	Define $ x = x(t,\eta ,\alpha)$, 
	$ z = x(t,\zeta,\alpha)$,
        $ x_j^i = \tfrac{\partial}{\partial \xi_i} x_j(t,\eta,\alpha)$,
         and likewise for $z_j^i$.  
	Fix $\both{j} \in \both{I}$, and define
	\begin{align*}
	S_\both{j} ^{nm} &:= 
	(C_\both{j} ^{nm} + C_\both{j} ^{nm'}P_{m'}^m) + 
	C_\both{j} ^{n'} P_{n'}^{nm} +
	(C_\both{j} ^{n'm} + C_\both{j}^{n'm'}P_{m'}^m) P_{n'}^n.
	\end{align*}
	Then 
	\begin{align*}
	\left\| 
	\frac{\partial}{\partial \xi_i}
	\big(  \cN_\both{j}(x,\alpha(x)) - 
	\cN_\both{j}(z,\alpha(z)  )
	\big)
	\right\|
	&\leq S_\both{j}^{nm} | x_m  - z_m|  \, \|z_n^i \|  +  
	H_\both{j}^n \| x_n^i - z_n^i \| .
	\end{align*}
\end{proposition}
\begin{proof}
	We have
	\begin{align}
	\frac{\partial }{\partial \xi_i} \cN_\both{j}(x,\alpha(x))
	&=
	\left( \cN_\both{j}^n ( x, \alpha(x) ) +  
	\cN_\both{j}^{n'} ( x, \alpha(x) ) \alpha_{n'}^n(x)  \right) \cdot x_{n}^i,
	\label{eq:D2Expansion}
	\end{align}
        and split the estimate into four parts:
	\begin{align*}
	\frac{\partial }{\partial \xi_i} \Big( \cN_\both{j}(x,\alpha(x)) -  \cN_\both{j}(z,\alpha(z)) \Big) 
	&=  
	\left( 
	\cN_\both{j}^n ( x, \alpha(x) ) 
	-
	\cN_\both{j}^n ( z, \alpha(z) )
	\right) \cdot z_{n}^i 
	\\[-1mm]
	& \qquad + 
		 \cN_\both{j}^{n'}(x,\alpha(x)) \left( \alpha_{n'}^n(x) - \alpha_{n'}^n(z) \right) z_{n}^i \\
	& \qquad + 
	    \left(  \cN_\both{j}^{n'} (x,\alpha(x)) - \cN_\both{j}^{n'} (z, \alpha(z)) \right)  \alpha_{n'}^n(z)   z_n^i 
		\\
	& \qquad + 
	\left( \cN_\both{j}^n ( x, \alpha(x) ) +  
	\cN_\both{j}^{n'} ( x, \alpha(x) ) \alpha_{n'}^n(x)  \right) \cdot (x_{n}^i - z_n^i).
	\end{align*} 
Each term is bound separately, as
	\begin{align*}
 \left( \cN_\both{j}^n ( x,\alpha(x)) - \cN_\both{j}^n(z,\alpha(z)) \right) \cdot z_n^i 
 & \leq
 (C_\both{j}^{nm} + C_\both{j}^{nm'}P_{m'}^m) |x_m - z_m| \, \| z_n^i\| 
 , \\
 \cN_\both{j}^{n'}(x,\alpha(x)) \left( \alpha_{n'}^n(x) - \alpha_{n'}^n(z) \right) z_{n}^i
 & \leq
  C_\both{j}^{n'} P_{n'}^{nm} |x_m - z_m| \, \| z_n^i \|
 , \\
 \left(  \cN_\both{j}^{n'} (x,\alpha(x)) - \cN_\both{j}^{n'} (z, \alpha(z)) \right)  \alpha_{n'}^n(z)   z_n^i 
 & \leq
 (C_\both{j}^{n'm} + C_\both{j}^{n'm'}P_{m'}^m) P_{n'}^n  |x_m - z_m| \, \| z_n^i \|
 , \\
\left( \cN_\both{j}^n(x,\alpha(x)) + \cN_\both{j}^{n'}(x,\alpha(x))   \alpha_{n'}^n (x)\right)  ( x_n^i - z_n^i ) 
 & \leq
 (C_\both{j}^n + C_\both{j}^{n'} P_{n'}^n) \,  \| x_n^i -z_n^i\|
 .
\end{align*}
	The result follows by collecting all terms.
\end{proof}

\begin{proposition} 
	\label{prop:InitialSecondDerivativeBound} 
	Define a tensor $\widetilde{K} \in (\R^{m_s})^{\otimes 3}  \otimes (\R^{m_s+1})^{\otimes 2} $ as 
	\[
	\widetilde{K}_{j, k_1 k_2}^{i l} = \left(
	\gamma_{k_1} + \gamma_{k_2} - \gamma_0 \right)^{-1}  C_s
	p_j 
	S_{j}^{nm} G_{m,k_1}^{l} G_{n,k_2}^{i} .
	\]
	Then    we have 
	\begin{align*}
	\left\| \tfrac{\partial}{\partial \xi_i}x(t,\eta,\alpha ) - \tfrac{\partial}{\partial \xi_i} x(t,\zeta ,\alpha ) \right\|  
	&\leq 
	\sum_{\substack{0 \leq k_1,k_2 \leq m_s \\ j \in I}} \left(e^{(\gamma_{k_1} + \gamma_{k_2})t} - e^{\gamma_0 t} \right) 
	\widetilde{K}_{j, k_1 k_2}^{i l} | \eta_{l} - \zeta_{l} | ,
	\end{align*}
	for all $ \alpha \in \cB_{\rho,P,\bar{P}}^{1,1}$ and $ \eta,\zeta \in B_s(\rho)$ and $ i \in I$. 
\end{proposition}

The indices in tensor notation $\widetilde{K}_{j, k_1 k_2}^{i l}$ 
are interpreted as follows. The superscripts correspond to derivatives, 
the subscript to the left of the comma corresponds to subspace projections, 
and the subscript to the right of the comma correspond to exponentials. 
\begin{proof}
	Define $ x = x(t,\eta ,\alpha)$ and 
	$ z = x(t,\zeta,\alpha)$.  
	Let $ x^i = \tfrac{\partial}{\partial \xi_i} x(t,\eta,\alpha)$ and likewise for $z^i$.
	By variation of constants, we have that
		\begin{align}
	x^i(t) -z^i(t) &= 
	\int_0^t e^{(\Lambda_s +L_s^s) (t-\tau)} 
 \frac{\partial }{\partial \xi_i} L_s^u
 \big( 
  \alpha(x(\tau)) -   \alpha(z(\tau)) \big) 	d\tau .
 \nonumber  \\ &  \qquad + 
	\int_0^t e^{(\Lambda_s +L_s^s) (t-\tau)} 
	 \frac{\partial }{\partial \xi_i} \left( \hat{\cN}_s(x(\tau),\alpha(x(\tau)) ) -  \hat{\cN}_s(z(\tau),\alpha(z(\tau)) ) \right)
	d\tau .
	 \label{eq:difference_D2Bound}
	\end{align}
	Expanding the partial derivatives appearing in  
	\eqref{eq:difference_D2Bound}, and dropping the $\tau$ dependence 
	in the notation in  the right hand side, gives
	\begin{align*}
	\frac{\partial }{\partial \xi_i} L_s^u \alpha(x(\tau))
	&=
	\sum_{j \in I}
	L_j^{n'}   \alpha_{n'}^n(x)    x_{n}^i  \\
	\frac{\partial }{\partial \xi_i} \hat{\cN}_s\bigl(x(\tau),\alpha(x(\tau))\bigr) 
	&=
	\sum_{j \in I}
	\left( \hat{\cN}_j^n ( x, \alpha(x) ) +  
	\hat{\cN}_j^{n'} ( x, \alpha(x) ) \alpha_{n'}^n (x) \right) \cdot x_{n}^i.
	\end{align*}	
	In Proposition \ref{lem:LipschitzDerivativeEst}  we demonstrated how the tensor $S$ offers a $C^{1,1}$ bound on $  \cN_\bj = L_\bj^s + L_\bj^u + \hat{  \cN}_\bj $, for $\bj \in \bI$.  
	By using~\eqref{eq:TOTALstableEigenvalueEstimate} we obtain, in analogy with the proof of 	Proposition \ref{lem:LipschitzDerivativeEst}, 
	\begin{align*}
	e^{- \lambda_s t}	\|x^i - z^i \| 
	&\leq   
	\int_0^t C_s e^{-\lambda_s \tau} 
	\sum_{j \in I}
	p_j 
	S_j^{nm} |x_m - z_m| \, \|z_n^i\|
	d\tau  + 
	\int_0^t e^{-\lambda_s \tau} 
	C_s  \hat{\cH} \|x^i - z^i \| 
	d\tau  .
	\end{align*}
	It then follows from Proposition~\ref{prop:BootstrapBounds} that
	\begin{align*}	
	e^{- \lambda_s t}	\|x^i - z^i \| 
	&\leq 	
	\int_0^t C_s e^{-\lambda_s \tau}  
	\sum_{\substack{0 \leq k_1, k_2 \leq m_s \\ j \in I}}
	e^{(\gamma_{k_1} + \gamma_{k_2})\tau}
	p_j
	S_j^{nm} G_{m,k_1}^{l} G_{n,k_2}^{i} | \eta_{l} - \zeta_{l} |
	d\tau  \nonumber \\
	& \qquad \qquad \qquad\qquad\qquad\qquad\qquad+ 
	\int_0^t e^{-\lambda_s \tau} 
	C_s \hat{\cH} \|x^i - z^i \| 
	d\tau   .
	\end{align*} 
	By Lemma \ref{prop:Fundamental} we infer that
\[
	\|x^i - z^i \|  \leq   
	\sum_{\substack{0 \leq k_1, k_2 \leq m_s \\ j \in I}}
	\frac{e^{(\gamma_{k_1} + \gamma_{k_2})t} - e^{\gamma_0 t} }{\gamma_{k_1} + \gamma_{k_2} - \gamma_0}
	C_s p_j S_j^{nm} G_{m,k_1}^{l} G_{n,k_2}^{i} | \eta_{l} - \zeta_{l} |.  
	\qedhere
\]
\end{proof}

\begin{remark}
	\label{rem:InitialK}
	Define  $ \{\mu_k \}_{k=1}^{N_\mu} = \{ \gamma_{k_1} \}_{k_1=0}^{m_s} \cup \{ \gamma_{k_1} + \gamma_{k_2}  \}_{k_1,k_2=0}^{m_s}$, with $N_\mu= (m_s+1)(m_s+4)/2$.
	Let $ \widetilde{K}$ be defined as in Proposition \ref{prop:InitialSecondDerivativeBound}, and 
	define a tensor $\widehat{K} \in (\R^{m_s})^{\otimes 3}  \otimes \R^{N_{\mu}} $ by  
	\[
	\widehat{K}_{j,k}^{il} :=  
	\begin{cases}
p_j \sum_{m \in I } \widetilde{K}_{m,k_1k_2}^{il} +\widetilde{K}_{m,k_2k_1}^{il} 
&
\mbox{if } \mu_k = \gamma_{k_1} + \gamma_{k_2} \text{ for } 0\leq k_1,k_2 \leq m_s,
\vspace{.1cm}
\\
- p_j 
\sum_{m \in I }
\sum_{0 \leq k_1,k_2 \leq m_s}  \widetilde{K}_{m,k_1k_2}^{il} +\widetilde{K}_{m,k_2k_1}^{il} 
&
\mbox{if } \mu_k = \gamma _0,
\\
0 &
\mbox{if } \mu_k = \gamma_{k_1}, \text{ for } 1 \leq k_1\leq m_s.
	\end{cases}
	\]
It follows from Proposition~\ref{prop:InitialSecondDerivativeBound} that $\widehat{K}$ satisfies Condition \ref{cond:SecondDerivative}.
\end{remark}

We now establish componentwise Lipschitz bounds on the derivatives.

\begin{proposition}
	\label{prop:LipDerivative}
Let $ \alpha \in \cB_{\rho,P,\bar{P}}^{1,1}$ and define $ x(t) = x(t,\eta,\alpha)$ and 
	$ z(t) = z(t,\zeta,\alpha)$ for some $\eta,\zeta \in B_s(\rho)$.  
	Let $ x_j^i(t) = \tfrac{\partial}{\partial \xi_i} x_j(t,\eta,\alpha)$ and likewise for $z_j^i(t)$.  
	Then 
		\begin{align*}	
	e^{- \lambda_j t}	\|x_j^i - z_j^i \| 
	&\leq 	
	\int_0^t e^{-\lambda_j \tau}  
	\sum_{0 \leq k_1, k_2 \leq m_s}
	e^{(\gamma_{k_1} + \gamma_{k_2})\tau}
	S_j^{nm} G_{m,k_1}^{l} G_{n,k_2}^{i} | \eta_{l} - \zeta_{l} |
	d\tau  \nonumber \\
	& \qquad \qquad \qquad\qquad\qquad\qquad\qquad+ 
	\int_0^t e^{-\lambda_j \tau} 
	H_j^{n} \|x_n^i - z_n^i\| 
	d\tau   .
	\end{align*}
\end{proposition}
\begin{proof}

	By variation of constants, we have that
\begin{align*}
	x_j^i(t) = e^{\Lambda_j t} \delta_j^i + \int_0^t e^{\Lambda_j (t-\tau)} 
	\left( \frac{\partial }{\partial \xi_i} \cN_j(x(\tau),\alpha(x(\tau))) 
	\right)
	d\tau ,
\end{align*}
	where $\delta_j^i$ is the Kronecker delta.  
	Taking the difference $ x_j^i - z_j^i$ we obtain
		\begin{align*}
	x^i(t) -z^i(t) &=  
	\int_0^t e^{\Lambda_s   (t-\tau)} 
	\frac{\partial }{\partial \xi_i} 
	\Big(
	 {\cN}_j(x(\tau),\alpha(x(\tau)) ) -  {\cN}_j(z(\tau),\alpha(z(\tau)) ) 
	 \Big)
	d\tau .
	\end{align*}
	From Proposition \ref{lem:LipschitzDerivativeEst} we have 
	\begin{align*}
	e^{- \lambda_j t}	\|x_j^i - z_j^i \| 
	&\leq   
	\int_0^t e^{-\lambda_j \tau} 
	S_j^{nm} |x_m - z_m| \, \| z_n^i\| 
	d\tau   + 
	\int_0^t e^{-\lambda_j \tau} 
	H_j^{n} \|x_n^i - z_n^i  \| 
	d\tau  .
	\end{align*}
	Plugging in the bounds  on $ | x_m - z_m|$ and  $\|z_n^i\|$ from Proposition \ref{prop:BootstrapBounds}, we obtain the desired result.

\end{proof}

\begin{theorem}
	\label{prop:SecondDerivativeAlgorithmResult}
Let  $ \{\mu_k \}_{k=1}^{N_\mu}$ 
and let the tensor $\widehat{K} \in (\R^{m_s})^{\otimes 3}   \otimes \R^{N_\mu} $ be as defined in Remark \ref{rem:InitialK}. 
When $ K$ is the output of Algorithm \ref{alg:GeneralBootStrap} taken with input $\widehat{K} $ and some $ N_{bootstrap} \geq 1$, then $ K$ satisfies Condition \ref{cond:SecondDerivative}. 	
\end{theorem}

The proof of Theorem \ref{prop:SecondDerivativeAlgorithmResult} 
follows from the argument outlined in Appendix \ref{sec:GeneralBootstrap}, where 
Conditions \ref{anz:IntegralInequality}
and  \ref{cond:GeneralGbounds}  
correspond to 
Proposition  \ref{prop:LipDerivative}
and Condition \ref{cond:SecondDerivative} 
respectively. 

\begin{theorem}
	\label{prop:LambdaC2Bounds}
Let $ \bar{P} \in \R^{m_u} \otimes (\R^{m_s})^{\otimes 2}$ and assume $K \in (\R^{m_s})^{\otimes 3}   \otimes \R^{N_\mu}  $ satisfies Condition \ref{cond:SecondDerivative}.  Define the tensor $ \tilde{ P} \in \R^{m_u} \otimes (\R^{m_s})^{\otimes 2} $ as 
	\begin{align}
	\tilde{P}_{j'}^{il} &:=
	\sum_{0 \leq k_1, k_2 \leq m_s} 
	(\lambda_{j'} - \gamma_{k_1} - \gamma_{k_2})^{-1}
	S_{j'}^{nm} G_{m,k_1}^{l} G_{n,k_2}^{i}  
	+ \sum_{1 \leq k \leq N_\mu}  
	(\lambda_{j'} - \mu_k)^{-1}
	  H_{j'}^n K_{n, k}^{il}  .
	\label{eq:SecondDerivativeBound}
	\end{align}
	Then for all $ \alpha \in \cB_{\rho,P,\bar{P}}^{1,1}$ we have  $\Lip(\partial_i \Psi[ \alpha])_{j'}^{l}  \leq \tilde{P}_{j'}^{il}$. 
		If $\tilde{P}_{j'}^{il} \leq \bar{P}_{j'}^{il}$ 
	then $\Psi: \cB_{\rho,P,\bar{P}}^{1,1} \to \cB_{\rho,P,\bar{P}}^{1,1}$ is well defined.   
\end{theorem}

\begin{proof}
	Let $\eta , \zeta \in B_s(\rho)$ and define $ x(t) = x(t,\eta,\alpha)$ and 
	$ z(t) = x(t,\zeta,\alpha)$.  
	Define $ x_j^i(t) = \tfrac{\partial}{\partial \xi_i} x_j(t,\eta,\alpha)$ and 
	likewise for $z_j^i(t)$. 
	From Definition \ref{def:LyapunovPerron} we have
	\[
	\Psi[\alpha](\eta) - \Psi[\alpha](\zeta)  = - \int_0^\infty e^{-\Lambda_u t} \big( \cN_{u}(x(t), \alpha(x(t))) - \cN_{u}(z(t),\alpha(z(t))) \big) dt .
	\]
	Using Proposition \ref{lem:LipschitzDerivativeEst} gives 
	\begin{align*}
	\left\|
	\Psi[\alpha]_{j'}^i(\eta) - \Psi[\alpha]_{j'}^i(\zeta) 
	\right\| &\leq  \int_0^\infty e^{-\lambda_{j'} t} 
	\left(
	S_{j'}^{nm} |x_m -z_m | \, \|z_n^i \| + H_{j'}^n \|x_n^i - z_n^i \| 		
	\right) dt .
	\end{align*}
Plugging in the bounds  on $ | x_m - z_m|$ and  $\|z_n^i\|$ from Proposition \ref{prop:BootstrapBounds},  as well as the bounds on $|x_n^i - z_n^i|$   from Proposition \ref{prop:LipDerivative}, gives 
	\begin{align*}
	\left\|	\Psi[\alpha]_{j'}^i(\eta) - \Psi[\alpha]_{j'}^i(\zeta) 
	\right\| 
	&\leq 
	\int_0^\infty 
	e^{-\lambda_{j'} t} 
	\sum_{0 \leq k_1,k_2 \leq m_s}  
	e^{(\gamma_{k_1} + \gamma_{k_2})t}	
	S_{j'}^{nm} G_{m,k_1}^{l} G_{n,k_2}^{i}
	 | \xi_{l} - \zeta_{l} | dt	 \nonumber 
	\\
	& \qquad + 
	\int_0^\infty 
	e^{-\lambda_{j'} t} 
	\sum_{1 \leq k \leq N_\mu}  
	e^{\mu_{k} t}
	 H_{j'}^n K_{n ,k}^{il} 
 | \eta_{l} - \zeta_{l} | 	
	dt \\
	&=
	\tilde{P}_{j'}^{il} | \eta_{l} - \zeta_{l} | 	.
	\end{align*}
	Hence, we have obtained the desired  
	 bound $\Lip(\partial_i \Psi[ \alpha])_{j'}^{l}  \leq \tilde{P}_{j'}^{il}$. 
\end{proof}


\section{Contraction Mapping}
\label{sec:Contraction}


\begin{remark} \label{rem:FixedConstants2}
	Throughout this section, suppose all the assumptions on
	the positive vector  $\rho \in \R^{m_s}$, the positive tensor $ P \in \R^{m_u} \otimes \R^{m_s}$, and the tensor $ G \in (\R^{m_s})^{\otimes 2}  \otimes \R^{m_s+1}$
	made in Remark \ref{rem:FixedConstants1} are in force. 
	Additionally, fix a tensor $K \in (\R^{m_s})^{\otimes 3} \otimes \R^{m_u}  \otimes \R^{N_\mu}$ satisfying Condition \ref{cond:SecondDerivative}, 
	and a positive tensor $ \bar{P} \in \R^{m_u} \otimes (\R^{m_s})^{\otimes 2}$. 
	Assume the hypotheses of Theorem \ref{prop:Endomorphism_2} and  Theorem \ref{prop:LambdaC2Bounds} are satisfied, so that both $ \Psi : \cB_{\rho,P}^{0,1} \to \cB_{\rho,P}^{0,1} $  and  $ \Psi : \cB_{\rho,P,\bar{P}}^{1,1} \to \cB_{\rho,P,\bar{P}}^{1,1} $  are well defined maps. 
\end{remark}

\subsection{Bounding the Difference Between Two Projected Systems}

We show that the Lyapunov-Perron operator is a contraction mapping 
in an appropriate norm.   
Note that the norm is weaker than the one used to define $\cB_{\rho,P}^{0,1}$  in Definition \ref{def:Ball_of_Functions}. 
\begin{definition}
	\label{def:SemiNorm}
	For $ \alpha \in \cE := \{ \alpha \in Lip(B_s(\rho) , X_u) : \alpha(0 ) = 0 \} $  define the 
	semi-norms
	\begin{align*}
	\| \alpha \|_{i'\cE}^i :=
	\sup_{\xi \in B_s(r); \xi_i \neq 0} 
	\frac{| \alpha_{i'}(\xi) - \alpha_{i'}(\xi - \xi_i)|}{|\xi_i|},
	\end{align*}
	where $ i \in I$ and $ i' \in I'$.
	The semi-norms define a norm by 
	\[
	\| \alpha \|_{\cE} := \sum_{i \in I , i' \in I'} 	\| \alpha \|_{i'\cE}^i .
	\]
\end{definition}
Note that $ \| \alpha \|_{i'\cE}^i \leq Lip( \alpha )_{i'}^i$ and $ | \alpha( \xi) | 
\leq \sum_{i' \in I'}  	\| \alpha \|_{i'\cE}^i  |\xi_i| \leq  
\| \alpha \|_\cE | \xi| \big(\max_{i \in I} p_{i} \big)$. 
With this norm both $ \cB_{\rho,P}^{0,1}$ and $\cB_{\rho,P,\bar{P}}^{1,1}$ are complete metric spaces (cf. \cite[Chapter 4]{chicone2006ODE}).

Before showing that $ \Psi$ is a contraction, we need to derive estimates on 
$x(t,\xi,\alpha) - x(t,\xi,\beta)$,  the difference between   two solutions of the 
projected system of Equation \eqref{eq:ProjectedSystem} 
for two different maps $\alpha , \beta \in \cB_{\rho,P}^{0,1}$. 
Classically, this  results in an estimate of the form $|x(t,\xi,\alpha) - x(t,\xi,\beta) |  \leq k e^{ \gamma t} | \xi| \| \alpha - \beta \|_{\cE}$, for some constants $k$ and $\gamma$.  
This estimate can be notably tightened, as 
at time zero $ | x(0,\xi,\alpha) - x(0,\xi,\beta) | = | \xi - \xi | = 0$.  
A bound on $ | x(t,\xi,\alpha) - x(t,\xi,\beta) | $ is obtained below,
using a tensor $F$ as now described.

\begin{condition} \label{cond:Contraction}
	Fix some $ \gamma_{-1} > \gamma_0$ and define $ \{ \mu_k\}_{k =1}^{m_s + 2 } = \{ \gamma_k \}_{k=-1}^{m_s}$. 
	A tensor  $ F \in  \left( \R^{m_s} \right)^{\otimes 3} \otimes \R^{m_u}  \otimes \R^{m_s+2}$ is said to satisfy Condition \ref{cond:Contraction} if
	\[
	|x_m(t,\xi,\alpha) -  x_m(t,\xi,\beta) | \leq \sum_{-1 \leq k \leq m_s}
	e^{\gamma_k t } 
	F_{mi,k}^{ni'}
\| \alpha - \beta \|_{i' \cE}^i  |\xi_n|,
	\]
	for all $ \alpha , \beta \in \cB_{\rho,P}^{0,1}$ and $ \xi \in B_s(\rho)$ and $m \in I$. 
\end{condition}

We obtain the tensor $F$ by applying the bootstrapping method as in Sections  \ref{sec:ExponentialTracking} and  \ref{sec:LyapunovPerron}, 
which is presented in a general setting in Appendix \ref{sec:GeneralBootstrap}. 
However, in this section we encounter a resonance problem involving 
$\gamma_0$, and 
augment $\{\gamma_k\}_{k=0}^{m_s}$, defining 
\[
 \gamma_{-1} := \gamma_0 /2.
\]
 In this manner we obtain an indexed set 
 $\{ \mu_k \}_{k=1}^{N_{\mu}} = \{ \gamma_{k} \}_{k=-1}^{m_s}$. 
The exact choice of $ \gamma_{-1}$ is somewhat arbitrary; it should satisfy
 $\lambda_{1'}> \gamma_{-1} > \gamma_0$,
 and $ (\gamma_{-1} - \gamma_{0})^{-1}$ should not be too large. 
We augment the tensor $G $ fixed in Remark \ref{rem:FixedConstants1} 
by defining  $ G_{i,-1}^n =0$ for all $ i , n \in I $.
To overcome the resonance problem we use the map $\mathcal{Q}_0$
(following the notation convention from Appendix~\ref{sec:GeneralBootstrap})
defined as 
\begin{equation} \label{eq:Q0projection}
	\mathcal{Q}_0(G)^n_{i,k} = 
\begin{cases}
G^n_{i,0}  	&  \mbox{ if } k = -1 \\
0		  	&  \mbox{ if } k = 0 \\
G^n_{i,k}  	&  \mbox{ if } 1 \leq k \leq m_s 
\end{cases} 
\qquad\text{for } i,n \in I.
\end{equation}
In Proposition \ref{prop:InitialContractionBound} and Remark \ref{rem:InitialF} below, we identify an initial tensor $\widehat{F}$ satisfiying Condition~\ref{cond:Contraction}.

\begin{proposition} \label{prop:InitialContractionBound}
Fix  $ \alpha,\beta \in \cB_{\rho,P}^{0,1}$ and some $ \gamma_{-1} > \gamma_0$. 
Define $ \mathcal{Q}_0$ as in \eqref{eq:Q0projection},
and the tensor 
$ \widetilde{F} \in  \left( \R^{m_s} \right)^{\otimes 3} 
\otimes \R^{m_u}  \otimes \R^{m_s+2}$ as
\begin{align*}
	\widetilde{F}_{ji,k}^{ni'} :=
	\begin{cases}
C_s  
(\gamma_k - \gamma_0)^{-1} p_j
C_j^{i'}  \mathcal{Q}_0(G)^n_{i,k}  &
\mbox{if } k \neq 0,  \\
	0 & \mbox{if } k=0.
	\end{cases}
\end{align*}
Then
\[
|x(t,\xi,\alpha) -  x(t,\xi,\beta) | \leq 
\sum_{\substack{-1 \leq k \leq m_s  , j \in I }}
\left(e^{\gamma_k t} - e^{\gamma_0 t} \right) 
\widetilde{F}_{ji,k}^{ni'}
\| \alpha - \beta \|_{i' \cE}^i  |\xi_n|,
\]
for all $ \alpha , \beta \in \cB_{\rho,P}^{0,1}$, and $ \xi \in B_s(\rho)$. 
\end{proposition}

%
%

\begin{proof}
	Fix an initial condition $ \xi \in B_s(\rho)$ and define $x(t) := x(t,\xi,\alpha) $ and $y(t) :=  x(t,\xi,\beta)$. 
	Variation of constants gives 
	\begin{align*}
	 x(t)-y(t)  =& \int_0^t e^{(\Lambda_s+L_s^s)(t-\tau)} \left( 
	L_s^u  \alpha(x(\tau) ) +  
	\hat{\cN}_s(x(\tau),\alpha(x(\tau))) 
	- L_s^u \beta(y(\tau))-  \hat{\cN}_s(y(\tau),\beta(y(\tau)))  \right) d\tau. 
	\end{align*}	
	By the usual splitting 
	$ \alpha( x) - \beta(y) = [\alpha(x) - \alpha(y)] + [\alpha(y) - \beta(y)]$ and 
	the  definition of $ \hat{  \cH}$ we obtain
	\begin{align*}
\left| 	L_s^u  \alpha(x  ) + 
\hat{\cN}_s(x ,\alpha(x )) 
- L_s^u \beta(y )-  \hat{\cN}_s(y ,\beta(y ))  \right|
& \leq \hat{  \cH} | x - y |  
\\ & \quad 
+ \left| 	L_s^u  \alpha(y  ) + 
\hat{\cN}_s(y ,\alpha(y)) 
- L_s^u \beta(y )-  \hat{\cN}_s(y ,\beta(y ))  \right| .
	\end{align*}
	Set $
		E^i_{i'} := \| \alpha - \beta \|_{i' \cE}^i  $.
	Since $	|\alpha_{i'}( y) - \beta_{i'}(y)| 
	\leq 
	E^i_{i'} | y_i|   $ we have
	\begin{align*}
	\left| 
	L_{s}^{u} \alpha(y) + 
	\hat{\cN}_{s}(y ,\alpha(y)) 
	- L_{s}^{u} \beta(y )-  \hat{\cN}_{s}(y ,\beta(y )) 
	\right| & 
	\leq 
	\sum_{j \in 
	I}
 p_j
	(\hat{C}_j^{i'} +D_j^{i'} ) E_{i'}^i  | y_i| .
	\end{align*}
Combining these estimates gives 
	\begin{align*}
	e^{-\lambda_{s} t} |x(t)-y(t)| 
	\leq &
	\int_0^t C_s  e^{-\lambda_{s} \tau} 
	\sum_{j \in I} 
	p_j
	C_j^{i'} E_{i'}^i   
	|y_i(\tau)| 
	d\tau 
	+	   
	\int_0^t  C_s e^{-\lambda_{s} \tau} 	\hat{\cH} |x(\tau)-y(\tau)|
	d\tau .
	\end{align*}	
	We would like to use the bound $ |y_i(\tau)| \leq \sum_{0 \leq k \leq m_s} e^{\gamma_k \tau } G_{i,k}^n |\xi_n| $ from Theorem \ref{prop:BootstrapBounds}, 
	and apply Lemma \ref{prop:Fundamental}. However, 
	this integral inequality has a resonance when~$\gamma_0$.
	 The problem is overcome by replacing $G$ with~$\mathcal{Q}_0(G)$,
         so that 
	\begin{align*}
	e^{-\lambda_{s} t} |x(t)-y(t)| &
	\leq 
	\int_0^t  C_s e^{-\lambda_{s} \tau} 
	\sum_{\substack{-1 \leq k \leq m_s ; j\in I}} 
	p_j
	C_j^{i'} E_{i'}^i    
	e^{\gamma_k \tau} \mathcal{Q}_0(G)^n_{i,k} |\xi_n|
	d\tau \\ 
&\qquad\quad	+	   
	\int_0^t C_s  e^{-\lambda_{s} \tau}  
	\hat{\cH} |x(\tau)-y(\tau)|
	d\tau.
	\end{align*}
	By Lemma \ref{prop:Fundamental}, we infer that
	\[
	|x(t)-y(t)| 
	\leq 
	C_s  
	\sum_{\substack{-1 \leq k \leq m_s ; j\in I}} 
	\frac{e^{\gamma_k t} - e^{\gamma_0 t}}{\gamma_k - \gamma_0} 
	p_j
	C_j^{i'}  \mathcal{Q}_0(G)^n_{i,k}  E_{i'}^i    
	|\xi_n| . \qedhere
	\]
\end{proof}

\begin{remark}
	\label{rem:InitialF}
	For some fixed $\gamma_{-1} > \gamma_0$, 
	define the  tensor $ \widetilde{F} \in  \left( \R^{m_s} \right)^{\otimes 3} \otimes \R^{m_u}  \otimes \R^{m_s+2}$ as in Proposition \ref{prop:InitialContractionBound}. 
	Define the tensor 
	$ \widehat{F} \in  \left( \R^{m_s} \right)^{\otimes 3} \otimes \R^{m_u}  \otimes \R^{m_s+2}$
	by
	\[
	\widehat{F}_{mi,k}^{ni'}  :=  
	\begin{cases}
	p_m \sum_{j \in I } \widetilde{F}_{ji,k}^{ni'}
	&
	\mbox{if } k  \neq 0,
	\vspace{.1cm}
	\\
	- p_m 
	\sum_{j \in I }
	\sum_{-1 \leq k_1 \leq m_s}  \widetilde{F}_{ji,k_1}^{ni'}
	&
	\mbox{if } k=0.
	\end{cases}
	\]
	It follows that 	  $\widehat{F}$ satisfies Condition \ref{cond:Contraction}.  
\end{remark}

We refine the initial norm estimate from Proposition~\ref{prop:InitialContractionBound}
using the following auxiliary proposition.
\begin{proposition}
	\label{prop:TrackingIneq2}
	Fix  
	 $ \alpha,\beta \in \cB_{\rho,P}^{0,1}$ and an initial condition 
	 $ \xi \in B_s$. Define
	\begin{align*}
	u_i(t) 	&:= | x_i(t,\xi,\alpha) - x_i(t,\xi,\beta)| \\
		E^i_{i'} &:= \| \alpha - \beta \|_{i' \cE}^i \\
	V_j(t) 		&:=  
	\int_0^t   e^{-\lambda_{j}\tau} 	
	\sum_{0 \leq k \leq m_s} 
	e^{\gamma_k \tau  }
	E^i_{i'}  C_j^{i'} 
	G_{i,k}^{n} 
	\left|  \xi_n  \right|  
	d\tau  .
	\end{align*}
	Then
	\begin{equation}
	\label{eq:IntegralInequalityTwoMaps}
	e^{-\lambda_j t} u_j(t) \leq V_j(t) +   \int_0^t e^{-\lambda_j \tau}  H_j^i   u_{i}(\tau) d\tau \, .
	\end{equation}
\end{proposition}

\begin{proof}
	Let $x(t) := x(t,\xi,\alpha) $ and $y(t) :=  x(t,\xi,\beta)$. 
	By variation of constants we have 
	\begin{align*}
 	x_j(t)-y_j(t) &= \int_0^t e^{\Lambda_j(t-\tau)} \big(\cN_j(x(\tau),\alpha(x(\tau))) - \cN_j(y(\tau),\beta(y(\tau))) \big) d\tau,
	\end{align*}	
	and the triangle inequality gives
	\begin{align*}
	|\alpha_{i'}( x) - \beta_{i'}(y)| &\leq 
	| \alpha_{i'}(y) - \beta_{i'}(y)|  + 
	|\alpha_{i'}( x) - \alpha_{i'}(y)| \\
	&\leq 
	 \|\alpha - \beta\|_{i' \cE}^i | y_i| 
	 +
	 P_{i'}^i |x_i - y_i| ,
	\end{align*}
	hence
	\begin{align}\label{e:NxNy}
	\left| \cN_{j}(x,\alpha(x)) - \cN_{j}(y, \beta( y)) 
	\right|&   
	\leq 
	C_j^{i'} E_{i'}^i  | y_i|
	+
	  H^i_j | x_i - y_i |.
	\end{align}
	Applying the bound from Theorem \ref{prop:BootstrapBounds} gives
	\begin{align*}
	e^{-\lambda_{j} t} |x_j-y_j|
	\leq &
	\int_0^t   e^{-\lambda_{j} \tau}
	\left( 
		C_j^{i'} E_{i'}^i  | y_i|
	+
	H^i_j | x_i - y_i |  
	\right) d\tau \\
	= &
	\int_0^t   e^{-\lambda_{j} \tau} 
	C_j^{i'} E_{i'}^i   
	|y_i|  
	d\tau 
	+	   
	\int_0^t   e^{-\lambda_{j} \tau} 	H_j^i |u_i|
	d\tau \\
	\leq&
\int_0^t   e^{-\lambda_{j} \tau} 
\sum_{\substack{0 \leq k \leq m_s }} 
C_j^{i'} E_{i'}^i    
e^{\gamma_k \tau} G_{i,k}^n |\xi_n|
ds 
+	   
\int_0^t   e^{-\lambda_{j} \tau}  
H_j^i u_i(\tau)
d\tau .
	\end{align*}	
Recalling the definition of $V_j(t) $,  the above inequality is of the form stated in \eqref{eq:IntegralInequalityTwoMaps}. 	
\end{proof}

\begin{theorem}
	\label{prop:ContractionAlgorithmResult}
	Define $N_\lambda = m_s$ and  $ \{\mu_k \}_{k=1}^{N_\mu} = \{ \gamma_k \}_{k=-1}^{m_s} $. 
	Let  	$ \widehat{F} \in  \left( \R^{m_s} \right)^{\otimes 3} \otimes \R^{m_u}  \otimes \R^{m_s+2}$ denote the tensor defined in Remark  \ref{rem:InitialF}. 
	When $F$ is the output of Algorithm \ref{alg:GeneralBootStrap} taken with input $\widehat{F} $ and some $ N_{bootstrap} \geq 1$, then $F $ satisfies  Condition \ref{cond:Contraction}. 	
\end{theorem}
\begin{proof} 
	By Proposition \ref{prop:InitialContractionBound} the initial tensor $F$ satisfies Condition \ref{cond:Contraction}.  
	We note that Proposition \ref{prop:TrackingIneq2} is a special case of  Condition \ref{anz:IntegralInequality} and Condition \ref{cond:Contraction} is a special case of Condition \ref{cond:GeneralGbounds}. 
	Hence Proposition  \ref{prop:GeneralAlgorithmResult} applies, yielding the result.  
\end{proof}

%

\subsection{Contraction Mapping}

The tensor $J$ below, which takes $m_s \times m_u$ matrices to  $m_s \times m_u$ matrices, provides a bound on $ \| \Psi[\alpha ]- \Psi[\beta]\|_{i' \cE }^i $. 

\begin{definition} \label{def:J_Contraction}
	Define the tensor $J \in   \left(  \R^{m_s} \otimes \R^{m_u} \right)^{\otimes 2} $  by 
	\begin{align}
	\label{eq:ContractionMappingTensor}
	J_{j'i}^{i'n} := \sum_{-1 \leq k \leq m_s}
	(\lambda_{j'} - \gamma_k)^{-1}
	\left(
	C_{j'}^{i'}  G_{i,k}^n 
	+
	H_{j'}^m F_{mi,k}^{ni'} 
	\right) .
	\end{align}	
\end{definition}

\begin{theorem}
		\label{prop:Contraction} 
		If the tensor $ F \in  \left( \R^{m_s} \right)^{\otimes 3} \otimes \R^{m_u}  \otimes \R^{m_s+2}$ satisfies Condition \ref{cond:Contraction}, 
		then  $ 	\|\Psi[\alpha] - \Psi[\beta]\|_{j'\cE}^n   
	\leq 
	J_{j'i}^{i'n} \|\alpha - \beta\|_{i'\cE}^i  
	$ for all 
	$ \alpha  , \beta \in  \cB_{\rho,P}^{0,1}$.
\end{theorem}
\begin{proof}

	Fix charts $ \alpha, \beta \in \cB_{\rho,P}^{0,1}$ and
	choose $\xi \in B_s(\rho)$. Define 
	$x := x(t,\xi,\alpha)$, and 
	$y := x(t,\xi,\beta)$.  
     By the definition of the Lyapunov-Perron operator, we have 
	\[
	\Psi[\alpha](\xi) - \Psi[\beta](\xi) = - \int_0^\infty e^{-\Lambda_{u} t} \left[ \cN_{u}(x,\alpha(x)) - \cN_{u}(y, \beta( y)) \right] dt .
	\]
Using~\eqref{e:NxNy} with the estimates provided in 
Conditions~\ref{cond:G_bounding_function} and~\ref{cond:Contraction}, we
 obtain
	\begin{align*}
	|\Psi[\alpha]_{j'} (\xi) - \Psi[\beta]_{j'} (\xi)| 
	\leq &
	 \int_0^\infty  e^{-\lambda_{j'}t} 
	\left(
	C_{j'}^{i'} E_{i'}^i  | y_i|
	+
	H_{j'}^i | x_i - y_i | 
	\right)
	dt
	\\
\leq &
\int_0^\infty 
 e^{-\lambda_{j'}t} 
 \sum_{-1 \leq k \leq m_s} e^{\gamma_k t} 
 E_{i'}^i 
\left(
C_{j'}^{i'}  G_{i,k}^n 
+
H_{j'}^m F_{mi,k}^{ni'} 
\right) 
|\xi_n|
dt .
	\end{align*}
Integrating gives
	\begin{align*}
 			|\Psi[\alpha]_{j'}(\xi) - \Psi[\beta]_{j'}(\xi)|   
 			&\leq 
 			  E_{i'}^i  
 			  J_{j'i}^{i'n}
 			  | \xi_n | ,
	\end{align*}
	where  the coefficients  $J_{j'i}^{i'n}$ are defined as in \eqref{eq:ContractionMappingTensor}.
	 It follows that
	  $
	\|\Psi[\alpha] - \Psi[\beta]\|_{j'\cE}^n   
	\leq 
	E_{i'}^i  
	J_{j'i}^{i'n}
	$. 
\end{proof}

\begin{remark} \label{rem:JContraction}
The tensor $J$ is a linear operator which maps $m_s \times m_u$ 
matrices to $m_s \times m_u$ matrices. 
If we represent an $m_s \times m_u$ matrix $E$ as an $m_s\cdot m_u$~dimensional  vector 
$\tilde{E}$  with components $ \tilde{E}_{(i'-1)m_s+i} = E_{i'}^i$, then the action of $J$ can be represented as a $m_s  m_u \times m_s  m_u$ matrix $\tilde{J}$ with 
components $\tilde{J}_{(j'-1)m_s+n}^{(i'-1)m_s+i} \equiv  J_{j'i}^{i'n}$. 
	
	We are principally interested in whether the Lyapunov-Perron operator 
	$\Psi$ has a unique fixed point. 
	By Theorem \ref{prop:Contraction}, this will be true if an iterative application of $J$ to any $m_s \times m_u$ matrix $E $ limits to zero, that is
	\[
	\lim_{k \to \infty}  	\underbrace{J \circ \dots \circ J}_{k} \cdot E =0 .
	\]
	This limits to zero  if and only if the spectral radius of $J$, denoted 
	by $\rho(J)$, is less than $1$.  Since $J$ is finite dimensional,  $\rho(J)$ is equal 
	to the absolute value of the eigenvalue with largest magnitude. 
	This is bounded as $\rho(J) \leq \| J^k \|^{1/k}$ for any positive integer $k \geq 1$,
	 and any matrix norm  $ \| \cdot \|$. 
\end{remark}

	The theorem below collects the major results thus far.

\begin{theorem}
\label{prop:UniqueStableManifold} 
Take the assumptions made in Remarks \ref{rem:FixedConstants1} and \ref{rem:FixedConstants2}.  
Suppose the tensor $ F \in  \left( \R^{m_s} \right)^{\otimes 3} \otimes \R^{m_u}  \otimes \R^{m_s+2}$ satisfies Condition \ref{cond:Contraction} and define $J \in \left(  \R^{m_s} \otimes \R^{m_u} \right)^{\otimes 2} $ as in Definition~\ref{def:J_Contraction}. 
If the spectral radius of $J$ is less than $1$, then there exists a unique fixed point $ \alpha \in  \cB_{\rho,P,\bar{P}}^{1,1}$ for which  $\Psi[\alpha ] = \alpha$. 
 Furthermore, the graph
\[
M_{\loc} :=  
\{
( \xx_s, \alpha( \xx_s ) ) \in X_s \times X_u : \xx_s \in B_s(\rho)
\} 
\]
is an invariant manifold under the flow  \eqref{eq:OriginalODE}, and points 
in $M_{\loc}$ 
converge asymptotically to $ 0$.

In addition, suppose that  $ \tilde{h}$ is an equilibrium solution to \eqref{eq:PreOriginalODE} satisfying  $ | \tilde{h}_\bi| < \epsilon_\bi$ for $ \bi \in \bI$, and that $\epsilon_i < \rho_i$ for $i \in I$.
Define $ \tilde{\alpha}(\xx_s) := {\alpha}(\xx_s - \tilde{h}_s)+\tilde{h}_u$. 
The graph 
\[
\widetilde{M}_{\loc} := 
\{
( \xx_s, \tilde{\alpha}( \xx_s ) ) \in X_s \times X_u : \xx_s \in B_s(\rho - \eps_s)
\}
\]
is an invariant manifold under the flow  \eqref{eq:PreOriginalODE}, and points 
in $\widetilde{M}_{\loc}$ 
converge asymptotically to $\tilde{ h }$. 
Moreover, we have the estimates
\begin{align*}
| \tilde{\alpha}_{i'}(\xx_s)  | &\leq P_{i'}^i ( | \xx_i| + \epsilon_i) + \epsilon_{i'}
&\|  \tilde{\alpha}_{i'}^i  (\xx_s)\| &\leq P_{i'}^i 
& \Lip(\partial_i   \tilde{\alpha})_{i'}^{j}  \leq \bar{P}_{i'}^{ij},
\end{align*}
for all $ \xx_s \in B_s(\rho - \eps_s)$ and $ i ,j \in I$ and $ i' \in I'$. 

\end{theorem}
\begin{proof} 
	We infer from the assumptions made 
	in Remarks \ref{rem:FixedConstants1} and \ref{rem:FixedConstants2}, all of which can be verified \emph{a posteriori}, that the map 	$ \Psi : \cB_{\rho,P,\bar{P}}^{1,1} \to \cB_{\rho,P,\bar{P}}^{1,1} $  is a well defined endomorphism.   
	Since the spectral radius of $J$ is less than $1$, there exists a unique fixed point $ \alpha \in  \cB_{\rho,P,\bar{P}}^{1,1}$ for which  $\Psi[\alpha ] = \alpha$, see Remark \ref{rem:JContraction}.
	As discussed in  Section~\ref{sec:LP_Overview}, the fixed point of the Lyapunov-Perron operator  provides us with a chart for a local invariant manifold  for the differential equation defined in \eqref{eq:OriginalODE}. 
	By construction $\alpha(0)=0$, hence the origin is contained in the manifold.  
	It follows from the proof of Proposition \ref{prop:StayInsideBall} that points in $M_{\loc}$ converge asymptotically to the origin. 
	
	As \eqref{eq:OriginalODE} is conjugate to \eqref{eq:PreOriginalODE} via the change of variables $\xx \to \xx + \tilde{h}$, it follows that $\tilde{\alpha}(\xx_s) $ is 
	a graph for a local invariant manifold (having a slightly smaller domain)  for the differential equation defined in \eqref{eq:PreOriginalODE}. 
	Furthermore this manifold contains the equilibrium $\tilde{h}$, 
	a point to which trajectories in 
	$\widetilde{M}_{\loc}$ are asymptotically attracted.  
	The error estimates follow by virtue of $\alpha \in  \cB_{\rho,P,\bar{P}}^{1,1}$.  
\end{proof}

As discussed at the end of in Section~\ref{sec:LP_Overview}, the fixed point of the Lyapunov-Perron operator provides us with a chart for the local stable manifold provided we have captured all stable eigenvalues.


\section{Application I: Linear Change of Variables}
\label{sec:SH_Linear}

\subsection{The Swift-Hohenberg Equation}
\label{sec:CoV_Linear_SH}
Consider the Swift-Hohenberg Equation \eqref{eq:SwiftHohenberg}
of Section \ref{sec:SH_eq}.
Since the boundary conditions are Neumann, 
we will expand the spatial variable using Fourier cosine series. 
Proceeding formally (we do not yet specify the norms)
define the space of one-sided sequence of real numbers, denoted
$Y = \R^{\N}$. 
Given a one parameter curve $ a \in C(\R ,Y)$, define a
path of Fourier cosine series by 
\[
u(t,x)= a_0(t) + 2 \sum_{k=1}^\infty a_k(t) \cos(kx). 
\]
Taking the expansion above as an \textit{ansatz},
and plugging it into Equation \eqref{eq:SwiftHohenberg} 
leads to the system of infinitely many coupled
scalar ordinary differential equations 
\begin{align} \label{eq:SH_Fourier}
\dot{a}_k = (- \beta_1 k^4 -\beta_2 k^2 +1) a_k - (a*a*a)_k.
\end{align}
Here, the discrete convolution $*$ for $ a,b\in Y$ is defined by 
\[
(a * b)_k = \sum_{\substack{k_1 + k_2 = k \\ k_1,k_2 \in \Z}} a_{|k_1|} b_{|k_2|} .
\]

We endow $Y$ with the ``analytic'' norm 
corresponding to cosine series with geometrically decaying coefficients.   
So, for $a \in Y$ let 
\[
| a |_{\ell_\nu^1} := 
\sum_{k=0}^\infty | a_k | \omega_k(\nu) ,
\] 
where 
\[
\omega_k(\nu) = \omega_k := 
\begin{cases}
1 & k = 0 \\
2 \nu^k & k \geq 1.
\end{cases} 
\]
With $\nu > 1$ define  
\[
\ell_\nu^1 = \left\{ a \in Y \, : \, |a|_{\ell_\nu^1} < \infty \right\},
\]
and note that $\ell_\nu^1$ is a commutative Banach algebra, in the sense that 
\[
\| a * b \|_{\nu}^1 \leq \| a \|_{\nu}^1 \, \| b \|_\nu^1, 
\qquad\text{for all } a, b \in \ell_\nu^1.
\]
We rewrite \eqref{eq:SH_Fourier} as a (densely defined) 
vector field $F \colon \ell_\nu^1 \to \ell_\nu^1$
given by 
\begin{align}
F(a) := \mathfrak{L} a - a * a * a, \label{eq:SH_Fourier_Condensed}
\end{align}
where $\mathfrak{L}$ is the diagonal linear operator 
\begin{equation}\label{e:defmathfrakL}
\mathfrak{L}(a)_k := (- \beta_1 k^4 - \beta_2 k^2 + 1) a_k,
\qquad\text{for all } k \geq 0.
\end{equation}
Fix some $N \in \N$ and  define a Galerkin projection $ \pi_N : \ell_\nu^1 \to \R^{N+1} \subseteq \ell_\nu^1$ by 
\begin{align}
 \pi_N (a) := ( a_0 , a_1 \dots a_{N-1},a_N , 0,0,0, \dots). \label{eq:GalerkinProjection}
\end{align}
We define the Galerkin projection of $F$ by $F_N := \pi_N\circ F \circ \pi_N$. 

\begin{remark}[Normal form] \label{rem:normalForm}
To enter into the notational framework established in Section \ref{sec:Notation} 
we define a change of variables conjugating the differential 
Equation   \eqref{eq:SH_Fourier} to one of the type given in
Equation \eqref{eq:PreOriginalODE}. 
Note that  \eqref{eq:SH_Fourier_Condensed}  has the desired form 
at the homogeneous equilibrium solution $ 0 \in \ell^1_\nu$, but that 
a change of variables is required when $a$ is non-trivial.
After performing the change of variables,  we will bound 
the constants needed to satisfy the hypotheses of 
Theorem \ref{prop:UniqueStableManifold}. 
\end{remark}


\begin{remark}[First order data]\label{r:morseindex} 
We exploit the  extensive literature on computer assisted 
 proofs for equilibrium solutions to partial differential equations, 
 and provide computer assisted proofs for the existence, 
 local uniqueness, and bounds on the accuracy of the numerical approximation. 
Such techniques rely on  solving the finite dimensional problem $F_N(\bar{a})=0$, 
and use an implicit function type argument to show that there is a 
point $ \tilde{a} \in \ell_\nu^1$ close to $\bar{a}$ for which $F( \tilde{a})=0$. 
We use the techniques described in \cite{HLM,berg2017introduction}.
Similar ideas are used to solve the linearized equations at $\tilde{a}$,
providing enclosures of the necessary eigendata.
The Morse index of the stationary point $\tilde{a}$,
denoted $n_u$, is 
established rigorously using a straightforward implementation 
based on the ideas and
techniques from~\cite{JBJF3D,BGLV}.  
\end{remark}

In a more theoretical setting we would use the sectorial nature 
of $\mathcal{L}$
to decompose $\ell_\nu^1$ as a Cartesian product of eigenspaces of 
$DF(\tilde{a})$. In the more constructive setting of the present work 
we do not have direct access to this data. 
Instead, we numerically compute approximate eigenspaces 
associated with the Galerkin projection.  
Suppose then that $A_N^\dagger \in \Mat(\R^{N+1},\R^{N+1})$ 
is a matrix of real numbers having that 
$A_N^\dagger \approx D F_N ( \bar{a})$.  

Assume for the moment (this assumption will have to be checked in 
practice) that $ A_N^\dagger$ has $n_u$ unstable eigenvalues
(i.e.\ it captures the correct Morse index, see Remark~\ref{r:morseindex}).
Let $\{ \mu_{k'}\}_{k'=1'}^{n_u'}$ denote positive numbers
approximating the unstable eigenvalues of $A_N^\dagger$, and 
$\{ \mu_{k}\}_{k=1}^{n_f}$ with $n_f=N+1-n_u$ denote negative numbers 
approximating the stable eigenvalues. 
Without loss of generality,  suppose that these numbers
are ordered as
\[
\mu_{n_u'} \geq \dots \geq \mu_{1'} > 0 >  \mu_1 \geq \dots \geq   \mu_{n_f}.
\]

\begin{remark}[Gradient structure] \label{rem:grad}
The Swift-Hohenberg PDE is a gradient system, hence $A_N^\dagger$ 
has real eigenvalues with $N+1$ linearly independent eigenvectors.  
Indeed, this is most easily established by working 
with the slightly adapted $\widetilde{F}$ rather than $F$ directly, where
\[
  \widetilde{F}(a)_k = \begin{cases}
  F(a)_0 / 2 & \text{for } k=0,\\
  F(a)_k & \text{for } k\geq 1,
  \end{cases}
\]
so that $D\widetilde{F}_N(\bar{a})$ is symmetric with respect to the standard inner product on $\R^{N+1}$. However, this is a minor technical point. 
\end{remark}



Consider now the 
Swift-Hohenberg equation 
at parameter values such that $m_u = 1$, and 
choose a decomposition of the stable eigenspace
having $m_s =2$. 
We decompose $X$ into subspaces
\begin{align*}
X_{1'} &:= \R^{n_u'} 
&
 X_{1} &:= \R^{n_f} 
 &
 X_{2} &:= \{a \in \ell_\nu^1 : a_k =0 \mbox{ for } k \leq N \},
\end{align*} 
and have  that
$X_u := X_{1'}$ and $X_s := X_1 \times X_2$ and $ X = X_u \times X_s$. 
We sometimes employ 
the notational shorthand $X_f := X_1$ and $X_\infty:= X_2$.  

Note that the map $ \pi_N$ defined in \eqref{eq:GalerkinProjection}, is
the projection $ \pi_N: X \to X_N \subseteq X$ where 
$ X_N := X_{1'} \times X_1 \cong \R^{N+1}$.
Define $\pi_\infty : X \to X_\infty$ by $ \pi_\infty x := x - \pi_N x$.  
A Schauder basis $ \{ \hat{e}_n  \}_{n \in \N}$ for $ X $ is given by
\begin{align*}
X_{1'} &:= \mbox{span} 
\{\hat{e}_{0} , \dots , \hat{e}_{n_u-1} \}
&
X_{1} &:= \mbox{span} 
\{\hat{e}_{n_u} , \dots , \hat{e}_{N} \}
&
X_{2} &:= \overline{\mbox{span} 
	\{\hat{e}_{N+1} ,\hat{e}_{N+2},  \dots  \} },
\end{align*}
so that every $\phi \in X$ has a unique representation 
 $\phi =  \sum_{n=0}^\infty \phi_n \hat{e}_{n}$.

We are now ready to construct 
a linear change of variables from $X$ to $ \ell_\nu^1$. 
Fix $Q_u \in \Mat(\R ^{n_u}, \R ^{N+1})$ and $Q_f \in \Mat(\R ^{n_f}, \R ^{N+1})$ 
as matrices whose columns are 
numerical approximations of unstable/stable eigenvectors of $A_N^\dagger$.
For $ \phi = (\phi_u,\phi_f , \phi_\infty ) \in X_u \times X_f \times X_\infty$, 
define the linear map $Q: X \to \ell_\nu^1$ by 
\begin{align} \label{eq:Linear_Qdef}
Q(\phi) &=  Q_u \phi_u + Q_f \phi_f + \phi_\infty.
\end{align}
We endow $X$ with a Banach space structure as follows. 
Let $\phi^N=\pi_N \phi$ and let $Q^N$ be the  $(N+1) \times (N+1)$ 
invertible matrix given by $Q^N = [Q_u, Q_f]$.  
Define the transformation $Q \colon X \to \ell_\nu^1$ by
\[
[Q \phi]_n = \begin{cases}
[Q^N \phi^N]_n & 0 \leq n \leq N, \\
\phi_n & n > N+1,
\end{cases}
\]
for $\phi \in X$.
Denote the columns of $Q$ by $q_n$, $n \in \mathbb{N}$.  Note that $q_n = e_n$ 
when $n \geq N+1$ and that $q_n = Q^N_n$, the $n$-the column of $Q^N$, 
for $0 \leq n \leq N$.  
Define the norm on $X$ by  
\begin{align} \label{eq:X_norm} 
| \phi |_X & := 
\sum_{n = 0}^N \left|\phi_n Q \hat{e}_{n} \right|_{\ell_{\nu}^1} \\
&= \sum_{n = 0}^N |\phi_n| | q_n |_{\ell_\nu^1} 
+ \sum_{n = N+1}^\infty |\phi_n|  \omega_n   
\nonumber \\
& = \sum_{n = 0}^N |\phi_n| | q_n |_{\ell_\nu^1}
 + | \phi_\infty |_{\ell_\nu^1}. \nonumber 
\end{align} 
Note that $  | \phi |_X = \sum_{\bi \in \bI} | \phi_\bi |$ for $\phi  \in X$, so 
that with this norm, $X$ satisfies the hypotheses of 
Proposition  \ref{prop:ComputeHhat}.

We also require explicit formulas for 
the induced norms on several collections of 
operators in $ \cL(X,X)$, $\cL(X,\ell^1_\nu)$ and $\cL(\ell_\nu^1,X)$.
Suppose that $M^N$ is a $(N+1) \times (N+1)$ matrix and define 
the linear operator $M \colon X \to X$ by 
\[
\left[M \phi\right]_n = \begin{cases}
[M^N \phi^N]_n & 0 \leq n \leq N , \\
0 & n \geq N+1 .
\end{cases}
\] 
A standard calculation shows that  
\begin{equation} \label{eq:Norm_X_X}
\| M \|_{\cL(X, X)} = \sup_{| \phi |_X = 1} \left| M \phi \right|_{X} 
\leq  \max_{0 \leq k \leq N} 
\frac{| M^N_k|_{X}}{| q_k |_{\ell_\nu^1}},
\end{equation}
where $M^N_k$ denotes the $k$-th column of $M^N$. 
Similarly, for $\Omega^N$ an $(N+1) \times (N+1)$ matrix 
define the linear operator $\Omega \colon X \to \ell_{\nu}^1$ by   
\[
\left[\Omega \phi\right]_n = \begin{cases}
[\Omega^N \phi^N]_n & 0 \leq n \leq N , \\
\phi_n & n \geq N+1 .
\end{cases}
\] 
Again, a standard calculation shows that
\begin{equation} \label{eq:Norm_X_ell}
\| \Omega \|_{\cL(X, \ell_\nu^1)} 
 = \sup_{| \phi |_X = 1} \left| \Omega \phi \right|_{\ell_\nu^1} 
\leq \max\left(\max_{0 \leq k \leq N} 
 \frac{| \Omega^N_k|_{\ell_\nu^1}}{| q_k |_{\ell_\nu^1}} 
\, , 1 \right)  ,
\end{equation}
where $\Omega^N_k$ denotes the $k$-th column of $\Omega^N$.
From this it follows that $\| Q \|_{\cL(X, \ell_\nu^1)} = 1$. 

To compute the norm of $Q^{-1} :\ell_\nu^1 \to X$, let 
$B^N$ denote the matrix inverse of $Q^N$.
The action of $Q^{-1} $  is expressed as
\[
\left[Q^{-1} a \right]_n = \begin{cases}
\left[B^N a^N \right]_n & 0 \leq n \leq N ,\\
a_n & n \geq N+1 .
\end{cases}
\]
Then
\begin{equation} \label{eq:Norm_ell_X}
\| Q^{-1} \|_{\cL(\ell_\nu^1, X)} 
= \sup_{| a |_{\ell_\nu^1} = 1} | Q^{-1} a |_{X} 
\leq
\max \left(
\max_{0 \leq k \leq N}   \frac{| B_k^N |_X}{\omega_k} \, ,1 
\right) .
\end{equation}
Now, 
for any $ \bi \in \both{I}$, we define projection maps $ \pi_{\both{i}} : X \to X_{\both{i}}$. 
Again,  $ \pi_\infty$ coincides with its usual definition. 
By our choice of norm on $X$, we have $ \| \pi_\bi \|_{\cL(X,X_\bi)} = 1$.  
Recalling the definitions of $p_u, p_s, p_\bi$ in Equation
\eqref{eq:ProjectionBound}, we have that $ p_u = p_s = p_\bi =1 $.  
Lastly, we define $\Lambda$ by 
\begin{align*}
		\Lambda_{1'} &:= \mbox{diag}\{\mu_{n_u} ,\dots , \mu_{1'} \} ,&
	\Lambda_{1} &:= \mbox{diag}\{\mu_{1} , \dots , \mu_{n_f} \} ,&
	\Lambda_{2} &:=  \fL \circ  \pi_\infty    .
\end{align*}

We show that the norm on $X$, 
as defined above, is well aligned with the semigroup $ e^{\Lambda t}$. 
Fix a point $  \phi =(\phi_u,\phi_f , \phi_\infty ) \in X$ and write $\phi_u = ( \phi_{0} , \dots , \phi_{n_u-1} ) $ and $ \phi_f = ( \phi_{n_u} , \dots , \phi_{N})$ and $ \phi_\infty  = ( \phi_{N+1}, \phi_{N+2} , \dots )$.   
Then for $ t \in \R$ we have 
\begin{alignat*}{1}
e^{\Lambda_{1'}t} \phi_u &= 
\sum_{1 \leq k \leq n_u}  e^{\mu_{k'}t}  \phi_{k-1} \hat{e}_{k-1}, 
\\
e^{\Lambda_{1}t} \phi_f &= 
	\sum_{1 \leq k \leq n_f} 
	e^{\mu_{k}t}  \phi_{k+n_u-1}  \hat{e}_{k+n_u-1},
\\
	e^{\Lambda_{2}t} \phi_\infty &= 
	\sum_{k=N+1}^\infty 
	e^{ (- \beta_1 k^4 - \beta_2 k^2 + 1)t}  \phi_{k}   \hat{e}_{k}.
\end{alignat*}
Define $ \lambda_{1'}   $, $\lambda_1$, and 
$ \lambda_2   $  as 
\begin{align}
\lambda_{1'} & := \mbox{Re } \mu_{1'} ,
&
\lambda_1 &:=  \mbox{Re } \mu_1,
&
\lambda_2 &:=  - \beta_1 (N+1)^4 - \beta_2 (N+1)^2 + 1 .
\end{align}
It follows that $ \lambda_{1'} \leq \mbox{Re } \mu_{k'}$ for $ 1' \leq k' \leq n_{u}'$, and 
$ \lambda_1 \geq \mbox{Re } \mu_{k}$ for $ 1 \leq k \leq n_{f}$, and 
$\lambda_2 \geq (- \beta_1 k^4 -\beta_2 k^2 +1) $ for $ k \geq N+1$.
Choose $N$ sufficiently large so that $-\beta_1 k^4 -\beta_2 k^2 +1$ is negative and decreasing for $k\geq N+1$.  
Then 
\begin{align*}
| e^{\Lambda_{1'}t} \phi_u |_X 
&\leq \sum_{0 \leq k \leq n_u-1}  e^{\lambda_{1'} t}  | Q \phi_{k} |_{\ell_\nu^1} , 
& \mbox { for } t \leq 0 , \\
| e^{\Lambda_{1}t} \phi_f  |_X
&\leq \;
\sum_{n_u \leq k \leq N} \; e^{\lambda_{1} t}  | Q \phi_{k} |_{\ell_\nu^1} , 
& \mbox { for } t \geq 0 , \\
| e^{\Lambda_{2}t} \phi_\infty |_X 
&\leq \;
\sum_{k=N+1}^\infty \;
e^{ \lambda_2 t} | Q \phi_{k} |_{\ell_\nu^1}
& \mbox { for } t \geq 0 .
\end{align*}
From Equation~\eqref{eq:X_norm}, we have that 
\eqref{eq:stableEigenvalueEstimate}  and 
\eqref{eq:UnstableEigenvalueEstimate} 
are satisfied.

\subsection{Bounds for the Linear Change of Coordinates }
\label{sec:Linear_COV_Estimates}
The estimates necessary for completing the argument are obtained following 
 the instructions outlined below, which summarizes the discussion of the 
 previous sections.
	\begin{enumerate}
		\item  \label{estimate:ChangeOfVariables} 
		For $U \subset X$, 
		define a change of variables $ K: U \to \ell_\nu^1$ 
		such that $K(0) = \bar{a}$. \\
		For the equilibrium $\tilde{ h} = K^{-1}(\tilde{ a})$, 
		obtain bounds $ | \pi_\bi \tilde{ h}| \leq \eps_{\bi}$ for $ \bi \in \bI$.
		\item \label{estimate:ConjugateSystem}
		Pull back the vector field from $\ell_{\nu}^1$ to $U$, creating the 
		conjugate differential equation 
		\[
		\dot{\xx} = DK(\xx)^{-1} F(K(\xx)).
		\]
		Define   $\tilde{\cN} \in C_{loc}^2 (U,X)$ 
		as  $\tilde{\cN}(\xx) :=DK(\xx)^{-1} F(K(\xx)) -\Lambda \xx $.
		\item \label{estimate:SecondDerivative}
		Obtain constants $ \tilde{C}_\both{j}^{\both{ik}}(r_s,r_u) $ which  
		bound $\| \tilde{\cN}_\both{j}^{\both{ik}}\|_{(r_s + \eps_s ,r_u+\eps_u)}$ 
		for $ \both{i},\both{j},\both{k} \in \both{I}$. 
		\item \label{estimate:Z1}
		Obtain constants $ \tilde{D}_\both{j}^\both{i} $ 
		 which bound $\| \tilde{\cN}_{\both{j}}^{\both{i}}(0) \|$ for  
		 $ \both{i},\both{j} \in \both{I}$. 
		\item  \label{estimate:Semigroup}
		Obtain constants $C_s,\lambda_s$ which satisfy equation
		\eqref{eq:TOTALstableEigenvalueEstimate} to bound
		 $e^{(\Lambda_s + L_s^s )t}$.   
	\end{enumerate}

In the remainder of this section we explain how to follow the outline 
above, arriving at a linear change of coordinates $K$. 
The results of the a calculation  are presented in Section \ref{sec:Linear_Results}.

\subsubsection{Estimate \ref{estimate:ChangeOfVariables} -- Defining a Change of Variables} \label{sec:defChangeVar}

Define the affine change of coordinates $K: X \to \ell_\nu^1$ by
\begin{align}
K(\phi) &:= \bar{a} + Q \phi.  \label{eq:LinearCoV}
\end{align} 
Let $| \bar{a} - \tilde{a} |_{\ell_\nu^1} \leq \epsilon$ be 
a bound on the distance between 
the approximate solution and true equilibrium solutions,
and define $\eps_\both{i} := \eps  \| \pi_{\bi} Q^{-1}  \|_{\cL(\ell_\nu^1, X_\bi)}$ 
for $ \both{i} \in \both{I}$ as needed in Proposition 
\ref{prop:Initial_Constant_Bounds}.

\subsubsection{Estimate \ref{estimate:ConjugateSystem} -- 
Defining the Conjugate Differential Equation}

Applying the change of coordinates defined in \eqref{eq:LinearCoV}
to the Swift-Hohenberg equation 
leads to 
\begin{equation}
\dot{\phi} =  \Lambda \phi + \tilde{\cN}(\phi) 
\qquad
\text{with}\qquad \tilde{\cN}(\phi) := DK(\phi)^{-1} F( K(\phi))  - \Lambda \phi .
\label{eq:def_G_ConjugateSystem}
\end{equation}
We note that the form of $ \tilde{\cN}$ as given is not easy to work with,
and expand $ \tilde{\cN}$ into an affine part and a purely nonlinear part. 
Define functions $E,R: X \to \ell_\nu^1$ as  
\begin{align*}
E(\phi) &:= F(\bar{a}) + DF(\bar{a}) Q \phi - Q  \Lambda \phi ,
& 
R(\phi) &:= 	-3 \bar{a} * (Q \phi)^{*2} - (Q\phi)^{*3} .
\end{align*} 
Then $E + R = F \circ K - DK  \cdot  \Lambda$, where
$DK(\phi) = Q$ for all $\phi \in X$. 
It follows that $\tilde{\cN}(\phi) = Q^{-1}  \left( 
E(\phi) + R(\phi) 
\right)$.

\subsubsection{Estimate \ref{estimate:SecondDerivative} --
Bounding $\tilde{\cN}_\both{k}^{\both{ij}} $}

All second derivatives of $E$ are zero. 
Hence $ \partial_\bi \partial_\bj \pi_\bk \tilde{\cN} = \tilde{\cN}_\both{k}^{\both{ij}}    
= ( Q^{-1} R)_\both{k}^{\both{ij}}  $ for $\both{i},\both{j},\both{k} \in \both{I}$. 
 For $ \phi \in X$, define 
\begin{align}  
\kQ 		:= Q \phi = Q_f  \phi_f + Q_u \phi_u + \phi_\infty,
\end{align} 
and note that each term in $R$ itself contains a term of the form 
$\kQ *\kQ$.  Set 
\[
\kQ^2 := \kQ * \kQ \qquad\text{and}\qquad \kQ^3 := \kQ * \kQ * \kQ.
\]
Then $R(\phi) =  -3 \bar{a} * \kQ^2- \kQ^3$.

The derivatives of $\kQ$ are 
\begin{align*}
\partial_{f} \kQ \cdot h_f	&= Q_f h_f, &
\partial_{u}  \kQ	\cdot h_u &= Q_u h_u, &
\partial_{\infty}  \kQ \cdot h_\infty&= h_\infty ,
\end{align*}
where $h_f \in X_f$, $h_u \in X_u$ and $h_\infty \in X_\infty$. 
Since $\|Q\|_{\cL(X,\ell_\nu^1)} =1$, we have
$\| \partial_{\both{i}} \kQ \|_{\cL(X,\ell_\nu^1)}  = 1 $ 
for $  \both{i} \in \both{I}$. 
As $\partial_\bi \kQ$ is a linear operator, the second derivatives
$\partial_{\both{ij}}\kQ$ vanish for all $ \both{i},\both{j} \in \both{I}$.

The derivatives of $\kQ^2$ and $\kQ^3$ are given by
\[
\partial_{\both{i}\both{j}} \kQ^2 = 2 \partial_{\both{i}} \kQ * \partial_{\both{j}} \kQ 
\qquad\text{and}\qquad
\partial_{\both{i}\both{j}} \kQ^3 = 6 \kQ* \partial_{\both{i}} \kQ * 
\partial_{\both{j}} \kQ,
\]
so that 
\[
\partial_{\both{i}\both{j}} R = 
-6 (\bar{a} + \kQ) * \partial_{i} \kQ * \partial_{j} \kQ.
\]
Recall that  $\| \partial_{\both{i}} \kQ \|_{\cL(X,\ell_\nu^1)}  = 1 $ 
for all $  \both{i} \in \both{I}$.  
Fixing $ \phi = (\phi_u,\phi_s) \in B_u(r_u) \times B_s(r_s)$ 
with $r_s =( r_f , r_\infty)$ gives $|Q \phi |  \leq  r_u + r_f + r_\infty $.
Define  
\begin{align}
C_\both{k}^{\both{ij}} &
:= 
6 \| \pi_\bk Q^{-1} \|_{\cL(\ell_\nu^1 , X)} 
  \left( | \bar{a}| + r_u + r_f + r_\infty + \epsilon_u + \epsilon_f + \epsilon_\infty \right).
\end{align}
Then $ \| \tilde{\cN}_\both{k}^{\both{ij}}  \|_{(r_s + \epsilon_s ,r_u + \epsilon_u)} 
\leq {C}_\both{k}^{\both{ij}} $ for $\both{i},\both{j},\both{k} \in \both{I}$.   

\subsubsection{Estimate \ref{estimate:Z1} --  Bounding $\tilde{\cN}_\both{j}^{\both{i}} (0) $ } 
\label{sec:LinearLbound}
Since $\partial_{\both{i}} R(0)  = 0$ and 
$ \partial_\phi DK(\phi)^{-1} E(\phi)= Q^{-1} DF(\bar{a})Q  - \Lambda $,  we have
\begin{align*}
\tilde{\cN}_\both{j}^{\both{i}} (0)  &=   \pi_{\both{j}} 
\left( 
Q^{-1} DF(\bar{a})Q  - \Lambda 
\right) \pi_{\both{i}} .
\end{align*}
Approximate $DF(\bar{a})$  by the 
operator $ A^{\dagger} : \ell_\nu^1 \to \ell_\nu^1$  defined by
\begin{align*}
(A^\dagger v)_k &:= 
\begin{cases}
(A^\dagger_N v)_k  & k \leq N \\
( \mathfrak{L} v)_k & k > N ,
\end{cases}
\end{align*}
for $v \in \ell_\nu^1$.
We bound $\tilde{\cN}_\both{j}^{\both{i}} (0)$  
by adding and subtracting $Q^{-1} A^{\dagger}Q$ to obtain
\begin{align} \label{eq:Dtildebound_initial_triangle}
\left\|
\tilde{\cN}_\both{j}^{\both{i}} (0)   \right\|_{\cL(X,X)}  
&\leq
\left\|  \pi_{\both{j}} 
Q^{-1} \left(   DF(\bar{a})  -  A^{\dagger} 
\right) Q   \pi_{\both{i}} 
 \right\|_{\cL(X,X)}  
+
\left\|
 \pi_{\both{j}} 
\left( 
 Q^{-1}A^{\dagger} Q  - \Lambda 
 \right) \pi_{\both{i}} 
 \right\|_{\cL(X,X)}   .
\end{align}
To bound the right summand in \eqref{eq:Dtildebound_initial_triangle}, note that 
$\pi_{\both{j}} 
\left( 
Q^{-1}A^{\dagger} Q  - \Lambda 
\right) \pi_{\both{i}}$
vanishes when either $\both{i}=\infty$ or $\both{j}=\infty$, hence
the right-summand in~\eqref{eq:Dtildebound_initial_triangle} 
is computed directly using \eqref{eq:Norm_X_X}.
The left summand in~\eqref{eq:Dtildebound_initial_triangle} 
is bounded by considering four cases, depending on whether $\bi$ or $\bj$ 
equals $\infty$. 
Each of these terms involves 
\begin{align}
\left( DF ( \bar{a} ) h - A^\dagger h \right)_k &= 
\begin{cases}
- 3 ( \bar{a} * \bar{a} * \pi_\infty h )_k  
+
((DF_N(\bar{a})- A_N^{\dagger}) \pi_N h)_k
& 0 \leq k \leq N \\
-3 ( \bar{a } * \bar{a} * h)_k & k \geq N+1.
\end{cases}
\label{eq:LinearCoV_Z1}
\end{align}

For the case $\bi = \infty$ and $ \bj = \infty$,  
since $ \ell_\nu^1$ is a Banach algebra and $ \pi_\infty$ projects onto 
the modes $ k \geq N+1$, we use
 \eqref{eq:LinearCoV_Z1}  and obtain
\[
\left| \pi_\infty 	\left( DF ( \bar{a} ) - A^\dagger \right) h \right| \leq 3 | \bar{a} * \bar{a} |_{\ell_{\nu}^1} | h|_{\ell_\nu^1} .
\]
Hence $ \left\|  \pi_{\infty}   \left(  DF(\bar{a})  -  A^{\dagger} 
\right)  
\right\|_{\cL(\ell_\nu^1,\ell_\nu^1)} \leq 3 | \bar{a} * \bar{a} |_{\ell_{\nu}^1} $.
Define
\begin{align} \label{eq:Dtilde_Case1}
\tilde{D}_\infty^\bi := 3 | \bar{a} * \bar{a} |_{\ell_{\nu}^1},
\end{align}
so that $\|\tilde{\cN}_\infty^{\both{i}} (0)\|_{\cL(X,X)} \leq \tilde{D}_\infty^\bi $ for all $ \bi \in \bI$.

For the case  $ \bi \neq \infty$ and $ \bj \neq \infty$, we note that
the operator  $ \pi_{\both{j}} 
 \left( Q^{-1}  DF(\bar{a}) Q -  \Lambda  
\right)    \pi_{\both{i}}   $ is 
represented by an $(N+1) \times (N+1)$ matrix and explicitly bound the norm.
Define 
\begin{align} \label{eq:Dtilde_Case2}
\tilde{D}_\bj^\bi &:= 
\left\|   \pi_{\both{j}} 
 \left(  Q^{-1} DF(\bar{a}) Q  -  \Lambda 
\right)   \pi_\bi
\right\|_{\cL(X,X)}.
\end{align}
It follows that $\|\tilde{\cN}_\both{j}^{\both{i}} (0)\| \leq \tilde{D}_\bj^\bi $ 
for all $ \bi ,\bj\in \bI - \{ \infty \}$. 
 
 For the case  $ \bi = \infty$ and $ \bj \neq \infty$,
it follows from~\eqref{eq:LinearCoV_Z1} that
\[
   \pi_{\bj}  [DF ( \bar{a} ) - A^\dagger ]_k =0\
   \qquad\text{for } k > 3N,
\]
where we recall that the subscript $k$ denotes the $k$-th column.
Since $Q \pi_\infty= \pi_\infty$, 
using the appropriate analogue of 
\eqref{eq:Norm_X_X} for a matrix of a larger size, we set  
\begin{equation} \label{eq:Dtilde_Case3}
  \tilde{D}_\bj^\infty :=  
  \max_{N+1 \leq k \leq 3N} \frac{| \pi_{\bj} Q^{-1} [DF ( \bar{a} ) - A^\dagger ]_k |_X}{\omega_k}.
\end{equation} 
It follows that 
$\|\tilde{\cN}_\both{j}^{\infty} (0)\| \leq \tilde{D}^\infty_\bj $ 
for all $ \bj \in \bI - \{ \infty  \}$. 

For the case $\bi \neq \infty$ and $ \bj = \infty$, 
we note that since $ \pi_\infty Q^{-1} = \pi_\infty$ 
and $ \pi_\infty A^{\dagger} \pi_N =0$, we have 
\[
 \pi_{\both{j}} 
Q^{-1} \left(   DF(\bar{a})  -  A^{\dagger} 
\right) Q   \pi_{\both{i}}  
= 
 \pi_{\infty} 
  DF(\bar{a})  Q   \pi_{\both{i}}.
\]
Recalling the formula in \eqref{eq:Norm_X_X}, we set 
\begin{equation} \label{eq:Dtilde_Case4}
\tilde{D}_\infty^\bi :=  
\max_{0 \leq k \leq N} 
\frac{| [  \pi_{\infty} 
	DF(\bar{a})  Q   \pi_{\both{i}}  ]_k |_X}
{ | q_k|_{\ell_\nu^1} }.
\end{equation} 
It follows that $\|\tilde{\cN}_\both{j}^{\infty} (0)\| \leq \tilde{D}^\infty_\bj $ 
for all $ \bj \in \bI - \{ \infty  \}$. 
With $\tilde{D}_\both{j}^\both{i}$ as in Equations \eqref{eq:Dtilde_Case1}, \eqref{eq:Dtilde_Case2},  \eqref{eq:Dtilde_Case3} and \eqref{eq:Dtilde_Case4},
we have bounds on 
$\|\tilde{\cN}_\both{j}^{\both{i}} (0)\|_{\cL( X,X)} $ for all $ \bi , \bj \in \bI$. 

\subsubsection{Estimate \ref{estimate:Semigroup} --  Semigroup Bounds}
\label{sec:ImplementSemiGroupBound_Linear}

To find  $ C_s$ and $ \lambda_s$ as needed in 
\eqref{eq:TOTALstableEigenvalueEstimate}, 
we use Proposition \ref{prop:FastSlowExponentialEstimate} and Remark \ref{rem:SemiGroupEll1}. 
Define $ D_{\bj}^{\bi}   := \tilde{D}_\bj^\bi + \tilde{C}_\bj^{\bi l} \epsilon_l +
\tilde{C}_\bj^{\bi l'} \epsilon_{l'}  $ for $ \bi , \bj \in \bI$ as in  Proposition \ref{prop:Initial_Constant_Bounds}, and 
let 
\begin{align*}
\mu_1 &:= \lambda_1  
&
\delta_a &:=  D_f^f 
&
\delta_b &:=  D_f^\infty  
&
& \\
\mu_\infty &:= \lambda_2  = \lambda_\infty 
&
\delta_c &:=  D_\infty^f  
&
\delta_d &:=  D_\infty^\infty  
&
\varepsilon &:= \sum_{ \tilde{\mu}_k \in \sigma(\Lambda_1) }  \frac{ |\mu_\infty|^{-1}    }{1 - |\mu_\infty|^{-1}   ( \delta_d   + | \tilde{\mu}_k| )}. 
\end{align*}
Note that $\| \Lambda_\infty ^{-1} \| = |\mu_\infty|^{-1}$.   
Assume that the spectral gap conditions 
\begin{align}
1 &> |\mu_\infty|^{-1}  \Big( \delta_d   +  \sup_{\tilde{\mu}_k \in \sigma(\Lambda_1)}  | \tilde{\mu}_k| \Big),
&
\mu_1   &> \mu_\infty + \delta_d + \varepsilon \delta_b \delta_c (1 + \varepsilon^2 \delta_b \delta_c)  ,
\end{align}
are satisfied.  (These must be checked in explicit examples).
It then follows from 
Proposition \ref{prop:FastSlowExponentialEstimate} 
and Remark \ref{rem:SemiGroupEll1} that 
\[
\| e^{(\Lambda_s + L_s^s) t} \| \leq C_s e^{\lambda_s t},
\]
where 
\begin{align*}
C_s &:=  (1+\varepsilon \delta_b)^2(1+\varepsilon \delta_c)^2  \\
\lambda_s &:=  \mu_1 + \delta_a C_s + \Delta\\
\Delta &:= \varepsilon \delta_b \delta_c   \max \left\{ 1 + \varepsilon \delta_c(1 + \varepsilon \delta_b ) , \varepsilon \delta_b ( 2 + \varepsilon^2 \delta_b \delta_c)  \right\}.
\end{align*}	


\subsection{Numerical Results} 
\label{sec:Linear_Results}
	 

Following the steps given in Section \ref{sec:Linear_COV_Estimates} 
allows us to prove a variety of stable manifold theorems.
In Theorem \ref{prop:CAP_Linear} below 
we present one such result, for the equilibrium displayed in 
Figure \ref{fig:SlowManifoldEndPoints}.
Here we choose $\rho_f$, 
the radius of the domain $B_s(\rho) \subseteq X_f \times X_\infty$
projected into the finite dimensional subspace $X_f$ , as large as possible. 
A number of additional results are presented 
in Section \ref{sec:ConclusionAndNumerics}.
	
\begin{theorem} \label{prop:CAP_Linear} 
Consider the Swift-Hohenberg Equation \eqref{eq:SwiftHohenberg}
with parameters $ \beta_1 = 0.05$, and  $\beta_2 = - 0.35$. 
Let $\nu = 1.001$ and
suppose that $\bar{a} \in \ell_\nu^1$ is an approximate equilibrium solution,
$\eps = 1.61 \times 10^{-14}$ close in the $\ell_\nu^1$
norm  to a true equilibrium solution. 
Fixing the Galerkin projection dimension at
$N=30$,  and following the instructions described 
in Section \ref{sec:defChangeVar}, 
we bound $\epsilon_s \leq  10^{-14} \cdot  ( 4.97, 1.61 ) $.
Let  $ \rho = 
\left( 2.2 \times 
10^{-2}  
 ,    10^{-5} 
\right)$, and 
define $B_s(\rho-\epsilon_s)$ as in Definition \ref{def:Ball},
and $I$,$I'$, and $\mathbf{I} = I \cup I'$  as in Remark \ref{rem:primes}.
Let 
\begin{equation*}
P = \left(
\begin{matrix}
0.153,  & 1.38 \times 10^{-5}  
\end{matrix}
\right) 
\quad \quad \mbox{and} \quad \quad 
\bar{P}  = \left(
\begin{matrix}
16.9 \times 10^{-0}   &	1.37 \times 10^{-3}   \\ 
1.37 \times 10^{-3}   &	2.14 \times 10^{-4}  
\end{matrix}
\right),
\end{equation*}
be tensors as in Definition \ref{def:Ball_of_Functions}.

Then, there exists a unique $\tilde{\alpha} \in C^{1,1}(B_s(\rho - \eps_s), X_u)$, 
such that the local stable manifold of 
$\tilde{a} \in \ell_\nu^1$ is given by 
\[
\xx_s \mapsto K\left(  \xx_s   , \tilde{\alpha}(\xx_s ) \right),
\] 
for $K$ as given in \eqref{eq:LinearCoV}.
Moreover, $\tilde{\alpha}$  has 
\begin{align*}
| \tilde{\alpha}_{i'}(\xi)  | &\leq 
3.36 \times 10^{-3}
&\|  \tilde{\alpha}_{i'}^i  (\xi)\| &\leq P_{i'}^i 
& \Lip(\partial_i   \tilde{\alpha})_{i'}^{j}  \leq \bar{P}_{i'}^{ij} ,
\end{align*}
for all $ \xi \in B_s(\rho - \eps_s)$, $ i ,j \in I$, $ i' \in I'$ and $ \bi \in \bI$.

\end{theorem}

\begin{proof} 
	In script \texttt{main.m} we calculate all of the constants and verify all of the hypotheses in Theorem \ref{prop:UniqueStableManifold}. 
	In particular we have a contraction constant  $\| J \|  < 0.356 $. 
	The entire computation took about 4 seconds and was run on MATLAB 2019a with INTLAB on a i7-8750H processor. 
\end{proof}


\section{Application II: Nonlinear Change of Variables}
\label{sec:SH_Nonlinear}

In this section we improve the approximation of the stable manifolds 
in certain directions, by making the nonlinear change of coordinates 
discussed in Section~\ref{sec:goodCoordinates}.  Again, 
we consider the example of the Swift-Hohenberg Equation \eqref{eq:SwiftHohenberg}.
We employ the notation established in Section \ref{sec:CoV_Linear_SH},
with some minor adjustments. 
In particular, we use $m_{u} =1$ and $ m_{s} =3$. 
Recalling the notation of~Section~\ref{sec:goodCoordinates},
set $n_u=m_{\unst}$, $n_\theta := m_{\slow}$, $n_f= m_{\fast}+m_{\slow}$,
and $N=n_u+n_f-1$,
and define
\begin{align*}
X_{1'} &:= \R^{n_u} 
&
X_{1} &:= \R^{n_\theta} 
&
X_{2} &:= \R^{n_f-n_\theta} 
&
X_{3} &:= \{a \in \ell_\nu^1 : a_k =0 \mbox{ for } k \leq N \} .
\end{align*} 
We write $X_u := X_{1'}$ and $X_s := X_1 \times X_2 \times X_3$ 
and $ X = X_u \times X_s$, and use the notational shorthand 
$X_\theta := X_1$ (slow stable), $X_f := X_2$ (fast but finite stable) 
and $X_\infty:= X_3$ (stable tail).   
The map $ \pi_N$, as defined in \eqref{eq:GalerkinProjection}, 
is a projection operator $ \pi_N: X \to X_N\subseteq X$,
 where we define $X_N := X_{1'} \times X_1 \times X_2 \cong \R^{N+1}$. 
Define $\pi_\infty : X \to X_\infty$ by $ \pi_\infty x := x - \pi_N x$, and 
$\Lambda$ as  
\begin{align*}
\Lambda_{1'} &:= \mbox{diag}\{\mu_{n_u'},\dots,\mu_{1'} \} ,&
\Lambda_{1} &:= \mbox{diag}\{\mu_{1},\dots,\mu_{n_\theta} \} ,&
\Lambda_{2} &:= \mbox{diag}\{\mu_{n_\theta+1},\dots,\mu_{n_f} \}, &
\Lambda_{3} &:=  \fL \circ  \pi_\infty,
\end{align*}
with $\mu$ defined in Section~\ref{sec:CoV_Linear_SH}, and
$\fL$ defined in~\eqref{e:defmathfrakL}.
Define $ \lambda_{\bi}$ for $\bi \in \bI$ by 
\begin{align}
\lambda_{1'} & := \mu_{1'}
&
\lambda_1 &:= \mu_1,
&
\lambda_2 &:= \mu_{n_{\theta }+1},
&
\lambda_3 &:=  - \beta_1 (N+1)^4 - \beta_2 (N+1)^2 + 1 .
\end{align} 
Repeating the argument given at the end of 
Section~\ref{sec:CoV_Linear_SH} in this context gives that 
the inequalities of Equations \eqref{eq:stableEigenvalueEstimate}  and 
\eqref{eq:UnstableEigenvalueEstimate} 
are satisfied. 
We now follow the scheme for stable manifold validation outlined in
Section \ref{sec:Linear_COV_Estimates}.

\subsection{Estimate \ref{estimate:ChangeOfVariables} -- Defining a Change of Variables}

Using the parameterization method, 
and the good coordinates discussed in Section \ref{sec:goodCoordinates}, 
we approximate a slow stable manifold and finite dimensional 
invariant normal bundles  
\begin{align*}
P & \colon [-1,1]^{n_{\theta}} \to X_N, \\
Q_f(\theta)  & \colon [-1,1]^{n_{\theta}} \to \Mat( \R^{n_f-n_\theta}, X_N) \\
Q_u(\theta) & \colon [-1,1]^{n_{\theta}} \to \Mat( \R^{n_u},X_N).
\end{align*}
These are chosen to approximately solve~\eqref{e:defrelP}--\eqref{e:defrelQs}. 
The error terms  
\begin{subequations}
\label{eq:Approximate_Conjugacy1}
\begin{align}
 E_{\theta} &: [-1,1]^{n_{\theta}} \to \ell^1_\nu  &
 E_{f} &: [-1,1]^{n_{\theta}} \to \cL(X_{f},\ell^1_\nu) \\
 E_{u} &: [-1,1]^{n_{\theta}} \to \cL(X_{u},\ell^1_\nu) &
 E_{\infty} &: [-1,1]^{n_{\theta}} \to \cL(X_{\infty},\ell^1_\nu),
\end{align}
\end{subequations}
are defined by 
\begin{subequations}
\label{eq:Approximate_Conjugacy2}
\begin{align} 
 E_{\theta}(\theta)&:=F(P(\theta)) - DP(\theta) \Lambda_{\theta} \theta    \\
 E_{f}(\theta)  &:=DF(P(\theta))Q_f(\theta) - DQ_f(\theta) \Lambda_{\theta} \theta - Q_f(\theta) \Lambda_{f}\\
 E_{u}(\theta) &:=DF(P(\theta))Q_u(\theta) - DQ_u(\theta) \Lambda_{\theta} \theta - Q_u(\theta) \Lambda_{u}     \\
 E_{\infty}(\theta)   &:= DF(P(\theta)) \pi_\infty-   \Lambda_{\infty}    .
\end{align}
\end{subequations}

Define $U := B(r_s + \eps_s , r_u + \eps_u) \subseteq  X_u \times [-1,1]^{n_{\theta}}  \times X_f  \times X_\infty$,
a normal frame bundle $Q: [-1,1]^{n_{\theta}} \to \cL(X/X_1,\ell_\nu^1)$, and a local  diffeomorphism $K : U \subseteq  X \to \ell^1_{\nu}$ by 
\begin{align}
Q(\theta)\phi &:= Q_f(\theta) \phi_f + Q_u(\theta) \phi_u + \phi_\infty \\
K(\theta,\phi) &:= P(\theta) + Q(\theta)\phi .  \label{eq:K_def}
\end{align}  	
	We define the norm $| \cdot |_X $ as in  \eqref{eq:X_norm} relative to the linear map $Q_0 :X \to \ell_\nu^1$  defined by 
\begin{equation} \label{eq:Q0_Nonlinear_Norm}
	Q_0 \cdot (h_\theta, h_\phi) :=
	DK(0,0) \cdot (h_\theta, h_\phi) =
	\partial_\theta P(0) h_\theta +Q(0) h_\phi, 
\end{equation}
where $ h_\theta \in X_\theta $ and $h_\phi \in X_u   \times X_f  \times X_\infty$.
	
While we do not have an  explicit expression for the inverse function $K^{-1}$, 
we can bound the norm of $\tilde{ h} = K^{-1}(\tilde{ a})$ as follows.
Note that $ K^{-1}(a) = Q_0^{-1} ( a - \bar{a}) + \cO(|a-\bar{a}|^2)$. 	
If   $| \bar{a} - \tilde{a} |_{\ell_\nu^1} \leq \epsilon$ bounds the distance between the approximate and true solutions, 
we apply standard techniques from rigorous numerics (cf Remark \ref{r:morseindex})
to bound $ |\pi_\bi \tilde{h}| \leq \eps_\both{i}$ 
for $ \both{i} \in \both{I}$ as needed in Proposition \ref{prop:Initial_Constant_Bounds},
in terms of $\eps$,  $\| \pi_{\bi} Q_0^{-1} \|$, and the polynomial coefficients of 
$K(\theta,\phi)$.

\subsection{Estimate \ref{estimate:ConjugateSystem} -- Defining the Conjugate Differential Equation}

Applying the coordinate change of Equation \eqref{eq:K_def}
to the Swift-Hohenberg equation leads to
\begin{align}
 \dot{\xx}  &=  \Lambda \xx + \tilde{\cN}(\xx) 
\label{eq:def_G_ConjugateSystem2},
&
\tilde{\cN}(\xx)  &:= DK (\xx)^{-1}  F( K(\xx) )  - \Lambda \xx ,
\end{align}
for $ \xx \in U$.
We now perform a Taylor expansion of  $ F(K(\xx))$ in $ \xx \in U$. 
To simplify the  notation, for $\xx = ( \theta , \phi)$ where 
$\theta \in [-1,1]^{n_{\theta}}$ and $ \phi \in X_{u} \times X_{f} \times X_\infty$, 
define
\begin{align} \label{eq:PQshorthand}
\kP  &:=  P(\theta) & \kQ &:=  Q(\theta)\phi . 
\end{align}
Starting from  \eqref{eq:SH_Fourier_Condensed}, expand $F(K(\theta,\phi))$ as
\begin{align*}
F(K(\theta,\phi)) &= \fL[ \kP  + \kQ ] - (\kP +\kQ )^{3}   \\
& = \big( \fL \kP - \kP^{3} \big)+  \big( \fL \kQ   -3 \kP^{2} *\kQ \big) - 3\kP*\kQ^{2}- \kQ^{3}  ,
\end{align*} 
where the powers denote products of convolutions.
Note that for $a , h \in \ell_\nu^1$,  the derivative of $F$ is given by
\[
DF(a) \cdot h = \fL h - 3 ( a * a* h ),
\]
so that
\begin{align*}
F( \kP ) &= \fL \kP  - \kP ^{3}  ,
&
DF(\kP) \cdot \kQ &= 
\fL  \kQ - 3 ( \kP^{2}  * \kQ).
\end{align*}
Defining a remainder term $R : U \subseteq X \to \ell^1_{\nu}$ by
\begin{align}
\kR =
R(\theta,\phi) &:= -3 P(\theta) * (Q(\theta)\phi )* (Q(\theta)\phi )  - (Q(\theta)\phi )* (Q(\theta)\phi )* (Q(\theta)\phi ) = -3\kP * \kQ^2 - \kQ^3, 
\label{eq:Remainder_Definition}
\end{align}
simplifies $F(K(\theta,\phi))$ as  
\begin{align}\label{eq:F_phi_taylor_expansion}
F(K(\theta,\phi))  &= F( \kP) + DF(\kP) \cdot \kQ + \kR.
\end{align}

The (approximate) conjugacy relations in \eqref{eq:Approximate_Conjugacy2}  
(approximately) linearize the non-remainder components in \eqref{eq:F_phi_taylor_expansion}. 
More precisely, we have that 
\begin{align*}
F(P(\theta)) + DF(P(\theta)) \left[ Q_f(\theta) \phi_f + Q_u(\theta) \phi_u +\phi_\infty \right] 
&=		  E_{\theta}(\theta) + DP(\theta) \Lambda_{\theta} \theta \\
& \qquad +E_{f}(\theta)\phi_f + DQ_f(\theta) (\Lambda_{\theta} \theta, \phi_f) + 
	Q_f(\theta) \Lambda_{f} \phi_f \nonumber \\
& \qquad +E_{u}(\theta)\phi_u+ DQ_u(\theta) (\Lambda_{\theta} \theta, \phi_u) +  
	Q_u(\theta) \Lambda_{u} \phi_u    \\
& \qquad +E_{\infty}(\theta)\phi_{\infty} +  \Lambda_{\infty} \phi_\infty \\
& = E(\theta,\phi)  +
 DK(\theta, \phi_f, \phi_u, \phi_\infty) 
\left(
\begin{array}{c}
\Lambda_{\theta} \theta  \\
\Lambda_{f} \phi_{f} \\
\Lambda_{u} \phi_{u} \\
\Lambda_{\infty} \phi_{\infty} \\
\end{array}
\right), 	 
\end{align*} 
where $E: U \to \ell_\nu^1$ is defined by
\begin{align}
 E(\theta,\phi) := E_{\theta}(\theta) + E_{f}(\theta)\phi_f + E_{u}(\theta)\phi_u + E_\infty(\theta) \phi_\infty. 
 \label{eq:NonlinearEDef}
\end{align}
It follows that for $ \xx \in U$, we have
\begin{align*}
DK   (\xx)^{-1}  F(K(\xx) )
	&= DK(\xx)^{-1} \big( E(\xx)  +
	DK(\xx) \Lambda  \xx+ R(\xx) \big) \\
	&=  \Lambda \xx + DK(\xx)^{-1} \left( E(\xx)  +  R(\xx) \right). 	
\end{align*}
Thus,  the differential equation is decomposed
 into a diagonalized part and nonlinear error terms. 
It follows that 
\begin{align}
	\tilde{\cN}(\theta,\phi)  = DK(\theta,\phi)^{-1} 
	\left(
	E(\theta,\phi) + R(\theta,\phi)
	\right).
	\label{e:tildecN}
\end{align}

\subsection{Estimate \ref{estimate:SecondDerivative} -- Bounding $\tilde{\cN}_\both{k}^{\both{ij}} $}
\label{sec:NCOVSecondDerivitive}


Throughout this section, consider 
points in the ball $ (\theta , \phi) \in U = B(r_s+\epsilon_s ,r_u + \epsilon_u) $,
and assume that $ | \phi_u | \leq r_u + \eps_u$,  $ | \phi_f | \leq r_f + \eps_f$,  and  $ | \phi_\infty | \leq r_\infty + \eps_\infty$. 
Additionally, choosing $\delta_\theta \in (0,1] $ such that if $| \theta |_X \leq r_\theta + \eps_\theta$, we have that 
$ (\theta)_k  \leq  \delta_\theta$ for all components $ 1 \leq k \leq n_\theta$, 
whereby $ U = B(r_s+\epsilon_s ,r_u + \epsilon_u)  \subseteq   X_u \times [-\delta_\theta,\delta_\theta]^{n_{\theta}}  \times X_f  \times X_\infty$. 

\subsubsection{Bounding the Derivatives of $DK$ and its Inverse}
\label{sec:BoundingDK}
Fix  $h = ( h_\theta, h_f, h_u , h_\infty) \in X_\theta \times X_f \times X_u \times X_\infty$. 
We have that 
\begin{align}
DK(\theta,\phi)  \cdot h  &= \big( \partial_\theta P(\theta) +\partial_\theta Q_f(\theta) \phi_f +  \partial_\theta Q_u(\theta) \phi_u \big) h_\theta  + 
Q_f(\theta) h_f +
Q_u(\theta) h_u +
h_\infty.
\end{align}
Define the maps
\begin{align*}
A_0(\theta) \cdot h  &:= \partial_{\theta} P(\theta) h_\theta + Q_f (\theta) h_f + Q_u(\theta) h_u + h_\infty ,
\\
A_1(\theta,\phi) \cdot h &:=  \partial_{\theta} Q_f(\theta) \phi_f h_\theta  +  \partial_{\theta} Q_u(\theta) \phi_u  h_\theta.
\end{align*}
Then $ DK = A_0 + A_1$. 

The norm of $ A_1$ is controlled by taking $ |\phi|$ small. 
Assume $A_0(\theta)$ is invertible for all 
$ \theta \in [-\delta_\theta , \delta_\theta]^{n_{\theta}}$ with inverse $ B(\theta) := A_0(\theta)^{-1}$.  
Indeed, the action of the  operator 
$A_0(\theta):X_N \times X_\infty \to \ell_\nu^1 \cong X_N \times X_\infty$ 
leaves both subspaces  $X_N$ and $ X_\infty$ invariant. 
The action of the operator $A_0(\theta)$ in the finite dimensional component is
represented by a  polynomial in $\theta$ with $(N+1) \times (N+1)$ 
matrix coefficients.  Its action in the infinite dimensional component is 
precisely the identity map. 
Hence the operator $B(\theta)=A_0(\theta)^{-1}$ is an infinite power series 
in $\theta$, with Taylor coefficients defined recursively by power matching. 
We compute finitely many of these coefficients by solving the recursion 
relations.
 
The inverse $DK^{-1}: \ell_\nu^1 \to X$ now has 
\begin{align*}
DK(\theta,\phi)^{-1} 
&= B(\theta) \big(  I +  A_1(\theta, \phi) B(\theta) \big) ^{-1}.
\end{align*} 
Bounds on the derivatives of $DK(\theta,\phi)^{-1} $ are obtained by 
the product rule. We first compute finitely many terms in the power series 
expansion of $B(\theta)$, and bound the Taylor remainder and its derivatives
using a Neumann series argument similar to the one given below
to bound $\big(  I +  A_1(\theta, \phi) B(\theta) \big) ^{-1}$. 
Indeed, for $\phi$ sufficiently small the Neumann series provides the bound
\begin{align*}
\|  \big(  I +  A_1(\theta, \phi) B(\theta) \big) ^{-1} \| 
&\leq \frac{ 1 }{1 - \|   A_1(\theta, \phi) B(\theta) \|_{\cL(\ell_\nu^1 , \ell_\nu^1)} } \\
&\leq \left[
1 -  
( |\phi_f| + |\phi_u| )
\| \partial_\theta Q(\theta) \|_{\cL(X_{\theta} \otimes X,\ell_\nu^1)} 
  \|  B(\theta) \|_{\cL(\ell_\nu^1 , X)} 
\right]^{-1}.
\end{align*} 

Derivatives of $ \big(  I +  A_1(\theta, \phi) B(\theta) \big) ^{-1}$ 
are bound using the fact that for any smooth path of invertible matrices,
it holds that
\[
\frac{\partial Y^{-1}}{\partial t} = - Y^{-1} \frac{\partial Y}{\partial t} Y^{-1}.
\] 
Applying the product rule gives
\begin{align*}
\frac{\partial^2 Y^{-1}}{\partial t \partial s}
&=   Y^{-1} \left( 
\frac{\partial Y}{\partial s} Y^{-1}
\frac{\partial Y}{\partial t} 
-   \frac{\partial^2 Y}{\partial t \partial s}  
+     \frac{\partial Y}{\partial t}		
Y^{-1} \frac{\partial Y}{\partial s}
\right) Y^{-1}.
\end{align*}
Hence, to bound the derivatives of 
$ \big(  I +  A_1(\theta, \phi) B(\theta) \big)^{-1}$, 
it suffices to bound the inverse and the derivatives of 
$I +  A_1(\theta, \phi) B(\theta)$. 

For fixed $(\theta , \phi ) \in U$ and $\bi \in \bI$, we see that the 
nontrivial first derivatives 
$\partial_\bi A_1 (\theta,\phi) : X \otimes X_{\bi} \to \ell_\nu^1$ are given by 
\begin{align*}
\partial_{\theta} A_1(\theta, \phi) &=  \partial_{\theta\theta} Q_f(\theta) \phi_f +  \partial_{\theta\theta} Q_u(\theta) \phi_u , &
\partial_{\star} A_1(\theta, \phi)   &=  \partial_{\theta}Q_\star(\theta)    \quad\text{for } \star \in \{f,u\}. 
\end{align*}
For fixed $ ( \theta , \phi ) \in U$, and $ \bi,\bj \in \bI$,
compute the nontrivial second derivatives 
$\partial_{\bi} \partial_{\bj} A_1 (\theta,\phi): X \otimes X_{\bi} \otimes X_{\bj} \to \ell_\nu^1$,
 by 
\begin{align*}
\partial_{\theta\theta} A_1(\theta, \phi) 
 &= \partial_{\theta\theta\theta}Q_f(\theta) \phi_f +  \partial_{\theta\theta\theta} Q_u(\theta) \phi_u  , & 
\partial_{\theta \star} A_1(\theta, \phi)  &=  \partial_{\theta\theta} Q_\star(\theta)    
\quad\text{for } \star \in \{f,u\}. 
\end{align*}
Note  that $\partial_\infty DK^{-1}=0$.
Furthermore, $\pi_{\infty}DK^{-1} = \pi_{\infty}$, so that 
$\pi_\infty \partial_{\bi} (DK^{-1}) =0$ for all $\bi \in \bI $. 
Then bounds on $DK^{-1}$  and its derivatives follow from bounds on 
\begin{align} \label{eq:finalDKiBounds}
  \| \pi_{\circ} B(\theta) \|_{\cL( \ell_\nu^1 , X)} && 
\left\| \pi_{\circ}\tfrac{\partial^k}{\partial \theta^k} B(\theta) \right\|_{\cL\left( X_{\theta}^{\otimes k} \otimes \ell_\nu^1 ,X \right)} &&
\left\| \pi_{\circ} \tfrac{\partial^k}{\partial \theta^k} Q(\theta) \right\|_{\cL\left( X \otimes X_{\theta}^{\otimes k},\ell_\nu^1\right)}  ,
\end{align}
where  $\pi_{\circ} \in \{ \pi_N , \pi_\infty \}$ and  $k=1,2,3$. 
Since we have either explicit expressions (we may take a supremum 
over $ \theta \in [-\delta_\theta,\delta_\theta]^{n_{\theta}}$ using interval arithmetic) 
or explicit bounds for each of these, we obtain the necessary explicit  
bounds on  $ DK^{-1}$  and its derivatives.
Note that bounds on 
$\pi_\bk DK(\theta,\phi)^{-1} =  
\pi_\bk  B(\theta) \big(  I +  A_1(\theta, \phi) B(\theta) \big) ^{-1}$
are improved by bounding $\| \pi_\bk B(\theta) \|_{\cL( \ell_\nu^1 , X)}$ for $ \bk \in \bI$, 
and likewise for the derivatives.

\subsubsection{Bounding $E$}
\label{sec:NL_E_bounds}
To bound $E: U \to \ell_\nu^1$ defined in \eqref{eq:NonlinearEDef},
see also~\eqref{eq:Approximate_Conjugacy1} and~\eqref{eq:Approximate_Conjugacy2},
we note first that these bounds  are
 calculated in the $ |\cdot |_{\ell_\nu^1}$ norm, 
whereas   bound on $ E_{f}, E_{u},E_\infty$  are calculated in the $ \| \cdot  \|_{\cL( X,\ell_\nu^1)}$ norm.
We have that 
\begin{align*}
	\partial_\theta E(\theta,\phi)  \cdot h &= 
	\Big(
	\partial_\theta E_{\theta} (\theta)+ 
	\partial_\theta E_{f}(\theta)\phi_f + 
	\partial_\theta E_{u}(\theta)\phi_u + 
	\partial_\theta E_{\infty}(\theta) \phi_\infty 
	\Big)\cdot h_\theta.
\end{align*}
The other first derivatives of $E$ are
\begin{align*}
\partial_{\star}  E(\theta,\phi)  \cdot h &=E_{\star} (\theta)  \cdot h_f,
\qquad\text{for }  \star \in \{f,u,\infty\}.
\end{align*} 
The nontrivial second derivatives of $E$ are
\begin{align*}
\partial_{\theta \theta} E(\theta,\phi) \cdot (h^1,h^2) &= 
\left(
\partial_{\theta \theta} E_{\theta} + 
\partial_{\theta \theta} E_{f}\phi_f + 
\partial_{\theta \theta} E_{u}\phi_u + 
\partial_{\theta \theta} E_{\infty}\phi_\infty
\right) \cdot ( h^1_\theta , h^2_\theta) ,
\\
\partial_{\theta \star} E(\theta,\phi) \cdot (h^1,h^2) &= 
\partial_\theta E_{\star}(\theta) \cdot (h^1_\theta,h^2_\star) , 
\qquad\text{for } \star \in \{f,u,\infty\}.
\end{align*}
Recall that we have an explicit finite dimensional polynomial representation 
for the functions $E_{\theta} $, $E_{f}$ and $ E_{u}$.  
For $E_\infty$ and its derivatives we have
\begin{align*}
E_\infty(\theta) \cdot \phi_\infty 
&=  -3 P(\theta) * P(\theta) *\phi_\infty\\
\partial_{\theta} E_\infty(\theta)  \cdot (\phi_\infty,h_\theta)
&=  -6 \left( \partial_{\theta} P(\theta) h_\theta \right)* P(\theta) *\phi_\infty\\
\partial_{\theta \theta }E_\infty(\theta) \cdot (\phi_\infty,h_\theta^1,h_\theta^2)
&= -6 (\partial_{\theta \theta } P(\theta) \cdot (h_\theta^1,h_\theta^2)) * P(\theta) *\phi_\infty
-6 (\partial_{\theta} P(\theta)  h_\theta^1 )  *(\partial_{\theta} P(\theta)  h_\theta^2 )*\phi_\infty .
\end{align*}

Using the bounds on $|\phi|$, the explicit expressions for the polynomials $P$, $Q$,  and the expressions above, we obtain bounds on $E$ over all of $U \subseteq X$.  
In summary, we have bounds on $E$ and its derivatives, and bound 
\begin{align} \label{eq:LfinalEBounds}
\left\| \pi_{\circ} \tfrac{\partial^k}{\partial \theta^k} E_{\theta} (\theta) \right\|_{\cL\left( X_{\theta}^{\otimes k},\ell_\nu^1\right)} &&
\left\| \pi_{\circ}  \tfrac{\partial^k}{\partial \theta^k} E_{\star} (\theta)  \right\|_{\cL\left( X \otimes X_{\theta}^{\otimes k},\ell_\nu^1\right)},
\end{align}
where $\pi_{\circ} \in \{ \pi_N , \pi_\infty \}$, 
$ \star \in \{ u,f,\infty \}$, and the supremum is taken
over $ \theta \in [-\delta_\theta,\delta_\theta]^{n_{\theta}}$.  
Here for $k=0,1,2$,  $X^{\otimes k}$ is the $k$-fold tensor product of 
$ X$, and $X^{\otimes 0}$ is the trivial vector space.

\subsubsection{Bounding $R$}
Recalling \eqref{eq:PQshorthand} and \eqref{eq:Remainder_Definition},
we have 
\begin{align*} 
\kP  &:=  P(\theta) , 
& 
\kQ &:=  Q_f(\theta) \phi_f + Q_u(\theta) \phi_u + \phi_\infty ,
&
\kR &:= -3 \kP * \kQ^{2} - \kQ^{3}.
\end{align*}
To calculate bounds on $R(\theta,\phi)=\kR$ and its derivatives, 
we start by calculating  the  derivatives of~$\kQ$.  These are 
\begin{align*}
\partial_{\theta} \kQ \cdot h&= \left(\partial_\theta Q_f  \phi_f 
+\partial_\theta Q_u  \phi_u \right) \cdot h_\theta,
& 
\partial_{\star} \kQ\cdot h	&= Q_\star \cdot h_\star
\quad\text{for } \star \in \{f,u\},
&
\partial_{\infty} \kQ\cdot h &= h_\infty.
\end{align*} 
The nonvanishing second derivatives of $ \kQ$ are given by
\begin{align*}
\partial_{\theta \theta} \kQ \cdot (h^1,h^2)	&= \left(\partial_{\theta \theta} Q_f \phi_f + \partial_{\theta \theta } Q_u \phi_u  \right)\cdot (h_{\theta}^1 ,  h_{\theta}^2 ), &
\partial_{\star\theta} \kQ \cdot (h^1,h^2) &= \partial_{\theta} Q_\star \cdot( h_\theta^1 ,h_\star^2 ) \quad\text{for } \star \in \{f,u\}.
\end{align*} 
The only nonvanishing derivatives of $\kP$ are with respect to $\theta$. 
Then, bounds on $ \cC$, $\kQ^3$,$ \kP*\kQ^2$, and their partial derivatives are
obtained using the product rule.

Using that $ \kR = -3 \kP * \cC- \kQ * \cC$, 
we have expressions for all of the first and second derivatives of $R$. 
Hence, to bound $R$  and its derivatives, it suffices to bound
\begin{align} \label{eq:finalRBounds}
\left\| \pi_{\circ} \tfrac{\partial^k}{\partial \theta^k} P (\theta) \right\|_{\cL\left( X_{\theta}^{\otimes k},\ell_\nu^1\right)} &&
\left\| \pi_{\circ}  \tfrac{\partial^k}{\partial \theta^k} Q_\star (\theta)  \right\|_{\cL\left( X \otimes X_{\theta}^{\otimes k},\ell_\nu^1\right)} ,
\end{align}
where we take $\pi_{\circ} \in \{ \pi_N , \pi_\infty \}$, $ \star \in \{ u,f \}$, $k=0,1,2$,
and the supremum over $ \theta \in [-\delta_\theta,\delta_\theta]^{n_{\theta}}$. 
 The rest of the bounds follow by applying the product rule (as detailed above),
  the Banach algebra property of~$ \ell_\nu^1$, and the 
  bounds on $| \phi |$ which result from restricting 
  to the ball $B(r_s+\epsilon_s,r_u + \epsilon_u )$.

\subsubsection{Bounding $\tilde{\cN}$}
The derivatives of $ \tilde{\cN}  = DK^{-1} (E +R)$ are calculated using the product rule. 
Exploiting the formulas derived in Section \ref{sec:NCOVSecondDerivitive} 
facilitates implementation of the constants $ \tilde{C}^{\bi\bk}_\bj$
bounding $\| \tilde{\cN}_\bj^{\bi\bk} \|_{(r_s + \epsilon_s,r_u + \epsilon_u)}$,
for $\bi,\bj,\bk \in \bI$ needed to apply Proposition \ref{prop:Initial_Constant_Bounds}.

\subsection{Estimate \ref{estimate:Z1} --  Bounding    $\tilde{\cN}_\both{j}^{\both{i}} (0) $ } 
 
 We now compute a tensor $\tilde{D}$ 
 bounding $\|\tilde{\cN}(0) \|$, as needed in 
 Proposition \ref{prop:Initial_Constant_Bounds}.
We infer from the computations in Section~\ref{sec:NCOVSecondDerivitive}
that $\cC(\theta,0)=0$, $D \cC(\theta,0)=0$, 
$ D(\kQ*\cC) =0$,  and $ D(\kP*\cC) =0$ when $ \phi =0$.
Hence $DR(\theta,0) =0$. 
Since $R(\theta,0) =0$ as well,
we infer that
\begin{align} \label{eq:Nonlinear_Z1}
 \partial_\bi \tilde{\cN}(0)  &=    DK(0)^{-1}  \partial_\bi E(0,0)
 +
 (\partial_\bi DK(0)^{-1}) E(0,0)
 \qquad\text{for } \bi \in \bI.
\end{align}

The first summand in \eqref{eq:Nonlinear_Z1} is 
similar to the term studied in  Section \ref{sec:LinearLbound}. 
To see this, starting from \eqref{eq:NonlinearEDef},
compute the  first derivatives of $E$ at $(\theta,\phi)=(0,0)$
to obtain
\begin{align*}
\partial_\theta E(0,0) \cdot h&= \partial_\theta E_{\theta}(0) \cdot h_\theta,  &
\partial_\star  E (0,0) \cdot h&=E_{\star}(0) \cdot h_\star 
\quad\text{for } \star \in \{f,u,\infty\}.
\end{align*}
We deduce from the definition of $E$ in  \eqref{eq:Approximate_Conjugacy2}
and the substitution $P(0) = \bar{a}$, that 
\begin{align*}
\partial_\theta E_{\theta}(0) \pi_\theta&=  \left( DF(\bar{a}) \partial_\theta P(0) - \partial_\theta  P(0) \Lambda_\theta \right)  \pi_\theta , \\
E_{\star}(0) \pi_\star&=  \left( DF(\bar{a}) Q_\star (0) - Q_\star  \Lambda_\star \right)  \pi_\star 
\quad\qquad\qquad \text{for } \star \in \{f,u,\infty\}.
\end{align*}
Using $Q_0$ as defined in \eqref{eq:Q0_Nonlinear_Norm},
we obtain the  simplification  
\[
\partial_\bi E(0,0)h 
= \left( DF(\bar{a}) Q_0 - Q_0 \Lambda \right) \pi_\bi
\qquad \text{for } \bi \in \bI.
\]
Finally, the first summand in \eqref{eq:Nonlinear_Z1} simplifies to 
\[
DK(0,0)^{-1} \partial_\bi E(0,0) = \left(Q_0^{-1} DF( \bar{a}) Q_0 - \Lambda \right) \pi_\bi 
\qquad \text{for } \bi \in \bI.
\]
We then bound $\| \pi_\bj \left(Q_0^{-1} DF( \bar{a}) Q_0 - \Lambda \right) \pi_{\bi} 
 \|_{\cL(X,X)}$  as in  Section \ref{sec:LinearLbound}, with the trivial addition 
 that the projection map $ \pi_\theta$ must also be considered. 

To bound the second summand in \eqref{eq:Nonlinear_Z1},
note that  $ E(0,0) = E_{\theta}(0)$, for which we have an explicit expression. 
From a calculation in the same vein as in Section \ref{sec:BoundingDK}, we obtain
\begin{align*}
\left( \partial_\bi DK(0)^{-1}  \right) E(0,0) &= 
- Q_0^{-1} \left(
\partial_{\both{i}} 
DK(0) 
\right)
Q_0^{-1} 
E_\theta(0) .
\end{align*}
Then 
\begin{align*}
\partial_\theta DK(0) &= \partial_\theta A_0(0) ,  &
\partial_\star DK(0)  &= \partial_\theta Q_{\star}(0)  
\quad\text{for } \star \in \{f,u,\infty\}.
\end{align*}
The norm  $| E_{\theta}(0) |_{\ell_\nu^1}$ is quite small in practice, 
and it suffices to obtain a rough bound on the norm of $ \partial_\bi DK(0)^{-1}$. 
Thus, for $ \bi,\bj \in \bI$ we bound the components of~\eqref{eq:Nonlinear_Z1}
as
\begin{align*}
\tilde{ D}_{\bj}^{\bi} &:= 
\| 
\pi_{\bj} \left(Q_0^{-1} DF( \bar{a}) Q_0 - \Lambda \right) \pi_\bi \|_{\cL(X,X)} 
+
\left\| 
\pi_\bj Q_0^{-1} 
\right\|_{\cL(\ell_\nu^1,X)}
\left\|
\partial_{\both{i}} 
DK(0)  
\right\|_{\cL(X_\bi \otimes X,\ell_\nu^1)}
\left| 
\pi_N 
Q_0^{-1} 
E_\theta(0)
\right|_X .
\end{align*}
There are some additional cancellations, as $\pi_\bj Q_0^{-1} 
\left(
\partial_{\both{i}} 
DK(0) 
\right) =0$ when $ \bi = \infty$ or $\bj = \infty$.

\subsection{Estimate \ref{estimate:Semigroup} --  Semigroup Bounds}

The constants $ C_s$ and $ \lambda_s$ are obtained by applying Theorem \ref{prop:FastSlowExponentialEstimate} as in Section \ref{sec:ImplementSemiGroupBound_Linear}. 
The only difference is that $X_s$ is decomposed into 3 subspaces in Section \ref{sec:SH_Nonlinear} (as opposed to 2 subspaces in the linear case). 
We argue as follows.
Define $D_{\bj}^{\bi}   := \tilde{D}_\bj^\bi + \tilde{C}_\bj^{\bi l} \epsilon_l +
\tilde{C}_\bj^{\bi l'} \epsilon_{l'}  $ as in  Proposition \ref{prop:Initial_Constant_Bounds},
and
	\begin{align*}
	\mu_1 &:= \lambda_1 
	&
	\delta_a &:= \max_{1 \leq i \leq m_s-1} \sum_{1 \leq j \leq m_s -1} D_j^i
	&
	\delta_b &:=   \sum_{1 \leq j \leq m_s -1} D_j^{m_s} ,
	\\
		\mu_\infty &:= \lambda_3 = \lambda_\infty 
	&
	\delta_c &:= \max_{1 \leq i \leq m_s-1}   D_{m_s}^i 
	&
    \delta_d &:= D_{m_s}^{m_s}.
	\end{align*}
The rest of the computation for $ C_s$ and $ \lambda_s$ are exactly
 as described in Section \ref{sec:ImplementSemiGroupBound_Linear}.

\subsection{Conclusion and Numerical Results} \label{sec:ConclusionAndNumerics}

We recall that the parameter $ \rho = (\rho_\theta, \rho_f, \rho_\infty)$ determines the size of the domain
\[
B_s(\rho) = \left\{ 
(\xx_\theta,\xx_f,\xx_\infty) \in X_s : |\xx_\theta| \leq \rho_\theta, |\xx_f| \leq \rho_f,|\xx_\infty| \leq \rho_\infty 
\right\},
\]
for the candidate charts $\alpha \in \cB_{\rho,P,\bar{P}}$,
 where $X_s$ is decomposed in terms of the eigenspaces 
$X_\theta$, $X_f$, and $X_\infty$ of $\Lambda_s$ corresponding to the slow stable eigenvalues, the fast-but-finite stable eigenvalues, and the remaining infinite stable eigenvalues respectively.  
This parameter $\rho$ has a significant impact on nearly every aspect of our analysis.

For a given application it may be advantageous to choose certain components of 
$\rho = (\rho_\theta, \rho_f , \rho_\infty)$ large and others small.  
For example, we generically expect connecting orbits to have a larger projection 
into the slow-stable subspace $X_\theta$ and a smaller projection into the other 
stable subspaces. 
In Theorem \ref{prop:NonlinearComputerAssistedProof}, we present one such 
result, taking $\rho_\theta$
as large as possible. 
The parameters are the same as the ones used to produce Figure \ref{fig:SlowManifoldEndPoints}. 
This nonlinear approximation of the stable manifold produces significantly better error estimates than a linear approximation:
the $C^0$ error bounds in Theorem \ref{prop:NonlinearComputerAssistedProof} are of size $7.43 \times 10^{-12}$, 
whereas the approximate manifold in Theorem \ref{prop:CAP_Linear}
 has $C^0$ error bounds of $3.36 \times 10^{-3}$. 


\begin{theorem} \label{prop:NonlinearComputerAssistedProof} 
	Consider the Swift-Hohenberg Equation \eqref{eq:SwiftHohenberg}
with parameters $ \beta_1 = 0.05$, and  $\beta_2 = - 0.35$. 
Let $\nu = 1.001$ and
suppose that $\bar{a} \in \ell_\nu^1$ is an approximate equilibrium solution,
$\eps = 1.61 \times 10^{-14}$ close in the $\ell_\nu^1$
norm  to a true equilibrium solution. 	
	Using the techniques discussed in Section \ref{sec:goodCoordinates}, we compute a slow stable manifold and finite dimensional (un)stable bundles, represented by Taylor polynomials of degree 20.
	Fixing the Galerkin projection dimension at
$N=30$,  and following the instructions described 
in Section \ref{sec:defChangeVar}, 
we bound
$ \epsilon_s \leq  10^{-14} \cdot ( 1.85 , 4.51, 1.61 ) $.
Let
\begin{align*}
\rho &= 
\left(
\begin{matrix}
3.18 \times 10^{-2}   &    10^{-6}    &   10^{-10} 
\end{matrix}
\right),
\end{align*}
and 
\begin{align*}
P &= \left(
\begin{matrix}
9.43 \times 10^{-11}  \\ 
4.41 \times 10^{-6}  \\  
3.31 \times 10^{-6}
\end{matrix}
\right)
&
\bar{P} &= \left(
\begin{matrix}
1.30 \times 10^{-9} &   5.60\times 10^{-5}  	& 1.04\times 10^{-4} \\
5.60\times 10^{-5} 	&   2.72\times 10^{-0}  & 8.20\times 10^{-4} \\
1.04\times 10^{-4} 	&   8.20\times 10^{-4}  	& 1.41\times 10^{-4} \\
\end{matrix}
\right),
\end{align*}
be tensors as in Definition \ref{def:Ball_of_Functions}.
Define 
$B_s(\rho-\epsilon_s)$ as in Definition \ref{def:Ball},
$I$,$I'$, and $\mathbf{I} = I \cup I'$  as in Remark \ref{rem:primes}.

Then, there exists a unique
$\tilde{\alpha} \in C^{1,1}(B_s(\rho - \eps_s), X_u)$ 
so that the local stable manifold of $ \tilde{a} \in \ell_\nu^1$
 is given by 
 \[
 \xx_s \mapsto K\left(  \xx_s   , \tilde{\alpha}(\xx_s  )  \right),
 \] 
 for $K$ as in \eqref{eq:K_def}. Moreover, $\tilde{\alpha}$  has 
 \begin{align*}
 | \tilde{\alpha}_{i'}(\xi)  | &\leq  7.43 \times 10^{-12}
 &\|  \tilde{\alpha}_{i'}^i  (\xi)\| &\leq P_{i'}^i 
 & \Lip (\partial_i   \tilde{\alpha})_{i'}^{j}  \leq \bar{P}_{i'}^{ij} ,
 \end{align*}
 for all $ \xi \in B_s(\rho - \eps_s)$ and $ i ,j \in I$, $ i' \in I'$ and $ \bi \in \bI$.
\end{theorem}

\begin{proof}
In script \texttt{main\_NL.m} we calculate all of the constants and verify all of the hypotheses in Theorem  \ref{prop:UniqueStableManifold}. 
In particular we have a contraction constant  $\| J \| < 5.86 \times 10^{-6}  $.
It takes approximately 11 seconds to construct the slow-stable manifold and normal bundles,  
23 seconds to compute the bounds detailed in Section \ref{sec:SH_Nonlinear}, 
and 12 seconds to compute all the bounds in Sections  \ref{sec:ExponentialTracking}-\ref{sec:Contraction} needed to validate the stable manifold. 
These we run on MATLAB 2019a with INTLAB on a i7-8750H processor. 

\end{proof}
 

 The nonlinear approximation in Theorem \ref{prop:NonlinearComputerAssistedProof} 
 is optimized to  produce a larger validated part of the manifold in the direction of the slow stable eigenvector, as this is where we would generically expect to find connecting orbits. 
Note that in Theorem \ref{prop:NonlinearComputerAssistedProof}  the gap between 
eigenvalues of $ \Lambda_{1'}$, $\Lambda_1$ and $\Lambda_2$ is not very large:
\begin{align*}
\lambda_{1'} &=    1.01,
&
 \lambda_1 &= -1.41,
 &
  \lambda_2 &= -1.99,
  &
  \lambda_3 &=  -4.58  \times  10^{4} .
\end{align*}
We took the slow-stable eigenspace to be one dimensional. If a particular application required a stable manifold which was wider along the second slowest stable eigendirection, we could increase $\rho_f$ at a cost of also increasing $P$, $\bar{P}$, etc. 
These error estimates could be improved somewhat by splitting $X_f$ into two subspaces. 
Moreover,  we could significantly increase the radius of our approximation along the second slowest stable eigendirection by using a higher dimensional slow stable manifold.     

From the classical theory \cite{chicone2006ODE}, we expect our derivative 
bound $ P \geq \|D\alpha\| $ to be at least as large as the ratio between the derivative of the nonlinearity and the spectral gap, roughly 
\[
|P| 
\gtrsim
\frac{\| D\cN \|}{\lambda_u - \lambda_s}
\gtrsim 
\frac{\|L \| + \|D^2 \cN \|   \rho }{\lambda_u - \lambda_s}.
\]
We expect that 
this bound should increase linearly with $\rho$,
 and be bounded below by $\|L\|$, 
 the error from not perfectly splitting $X_u \times X_s$ into eigenspaces. 
This scaling is observed in Figure  \ref{fig:ErrorScaling}, where  we display
the error bounds in Theorem \ref{prop:CAP_Linear} and Theorem \ref{prop:NonlinearComputerAssistedProof} as functions of $\rho$. 
The nonlinear approximation maintains small error bounds,
despite taking $\rho_\theta$ large.  This is because the change of 
variables prepares the nonlinearity so
that $\| \partial_\theta D \cN\|$  is small. 
 Note that one should be mindful in comparing the two graphs in 
 Figure~\ref{fig:ErrorScaling}, as in Theorem \ref{prop:CAP_Linear} 
 we split $X_s = X_f \times X_\infty$ with $ \dim(X_f) = N$, and in Theorem \ref{prop:NonlinearComputerAssistedProof} we split $X_s = X_\theta \times X_f \times X_\infty$ with $ \dim(X_\theta)=1$ and $ \dim(X_f) = N-1$. 
 

 \begin{figure}
 	\centering
 	\includegraphics[width=.4\linewidth]{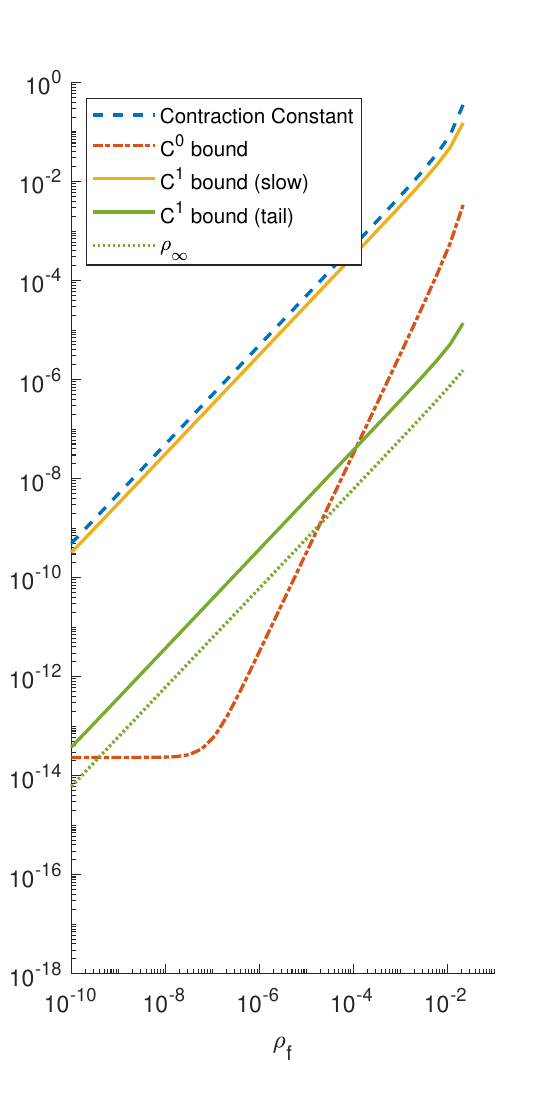} ~
 	\includegraphics[width=.4\linewidth]{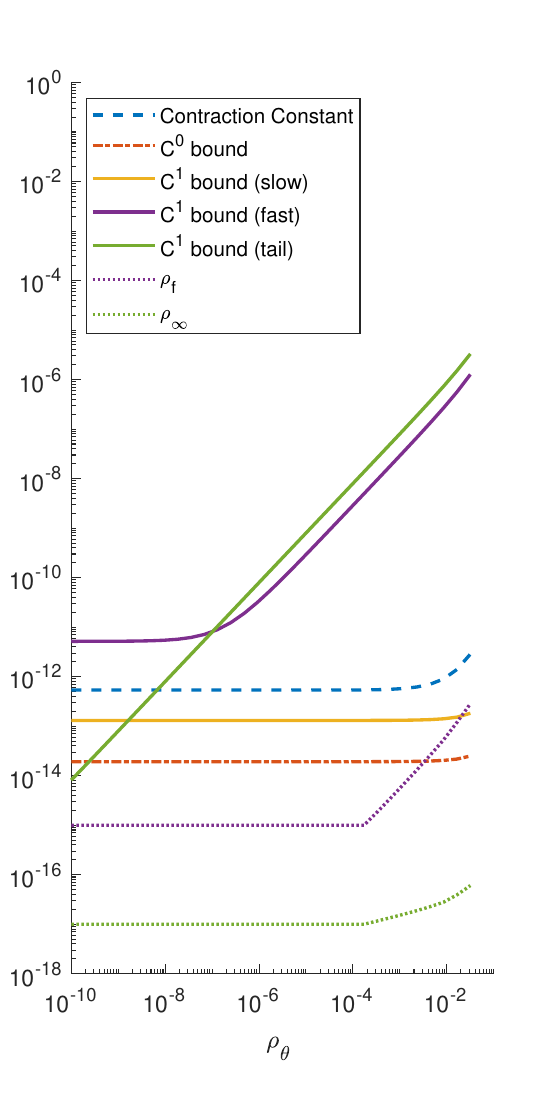}
 	\caption{ (Left) Using the estimates from Section \ref{sec:SH_Linear}, the bounds produced by a computer assisted proof for a range of radii $\rho_f \in [10^{-10},0.022]$, with $ \rho_\infty$ chosen to be as small as possible. 
 		(Right) Using the estimates from Section \ref{sec:SH_Nonlinear}, the bounds produced for a range of radii $\rho_\theta \in [10^{-10},0.0318]$, with $\rho_f $  and  $ \rho_\infty$ chosen to be as small as possible.  
 		Note that the nonlinear approximation yields smaller $C^0$ error bounds (red dash-dotted lines).
 	}
 	\label{fig:ErrorScaling}
 \end{figure}

When using the linear approximation we see that  
 for a large range of $\rho_f$, the contraction constant, the tensor~$P$,
  and the minimal choice of $ \rho_\infty$, all scale linearly with $\rho_f$. 
 The $C^0$ error of the manifold, given by 
 $|\tilde{\alpha}_{i'}|\leq  P_{i'}^i ( \rho_i + \epsilon_i) + \eps_{i'}$ in Theorem \ref{prop:UniqueStableManifold},  
 is dominated by the error in validating the equilibrium until $\rho_f \approx 10^{-7}$, where it 
 begins to scale quadratically with $ \rho_f$. 
 The $C^{1,1}$ error bounds on the norm of the components of $\bar{P}$ do not improve much for $ \rho < 10^{-3}$, and increase quite rapidly for $ \rho_f > 10^{-2}$. 
 
For the nonlinear approximation, the error in validating the equilibrium dominates 
the $C^0$ bound until $ \rho_\theta \approx 10^{-2}$, 
the point after which $P_u^\theta$  increases marginally. 
The contraction constant scales similarly, begining to increase 
around $ \rho_\theta \approx 10^{-3}$. 
 The $C^1$ bounds in the $X_f$ and $X_\infty$ subspaces
 are bounded below by the accuracy of the decomposition into
 eigenspaces of $DF(\bar{a})$, and increase linearly with $ \rho_\theta$.  
 For the whole range of admissible $\rho_\theta$, both $ \rho_f$ and $ \rho_\infty$ 
 can be taken exceedingly small, without contributing significantly
 to the overall error.

We do not expect to validate a global stable manifold with the Lyapunov-Perron approach;  
if $\rho$ is too large, the various hypotheses of Theorem \ref{prop:UniqueStableManifold} 
may no longer be satisfied.  
 For example, we may be unable to prove the image of $\Psi$ is contained within  
 $\cB_{\rho,P}^{0,1} $ or $\cB_{\rho,P,\bar{P}}^{1,1}$, as detailed in Theorems  \ref{prop:Endomorphism} or Theorem \ref{prop:Endomorphism_2}. 
 Other causes for failure 
 would be if 
 $\|J\|>1$ whereby $\Psi$ is not a contraction mapping, or 
 if we are unable to prove solutions  $x( t, \xi, \alpha)$ are contained within $B_s(\rho)$ for all $ t\geq 0$ as required by Proposition \ref{prop:StayInsideBall}. 
When using a linear approximation, many of these hypotheses all 
simultaneously 
 fail for larger values of $\rho$. 
In contrast, for the nonlinear approximation in Section \ref{sec:SH_Nonlinear}, the dominant limiting factor 
is the condition  $ \gamma_0 = \lambda_s + C_s \hat{  \cH} <0$ as required in Proposition \ref{prop:StayInsideBall}. 
Overall, the framework developed in Sections \ref{sec:Notation} - \ref{sec:Contraction} allow us to leverage our estimates on our approximate stable manifold made in Sections \ref{sec:SH_Linear}-\ref{sec:SH_Nonlinear}.  
%
%


%
%
%
%
%
%
%
%

\bibliographystyle{amsplain}
\bibliography{BIBInvariantManifold}

\providecommand{\bysame}{\leavevmode\hbox to3em{\hrulefill}\thinspace}
\providecommand{\MR}{\relax\ifhmode\unskip\space\fi MR }
\providecommand{\MRhref}[2]{%
  \href{http://www.ams.org/mathscinet-getitem?mr=#1}{#2}
}
\providecommand{\href}[2]{#2}
\begin{thebibliography}{10}

\bibitem{MR2679365}
Gianni Arioli and Hans Koch, \emph{Computer-assisted methods for the study of
  stationary solutions in dissipative systems, applied to the
  {K}uramoto-{S}ivashinski equation}, Arch. Ration. Mech. Anal. \textbf{197}
  (2010), no.~3, 1033--1051. \MR{2679365}

\bibitem{MR2852213}
\bysame, \emph{Non-symmetric low-index solutions for a symmetric boundary value
  problem}, J. Differential Equations \textbf{252} (2012), no.~1, 448--458.
  \MR{2852213}

\bibitem{MR3281845}
\bysame, \emph{Existence and stability of traveling pulse solutions of the
  {F}itz{H}ugh-{N}agumo equation}, Nonlinear Anal. \textbf{113} (2015), 51--70.
  \MR{3281845}

\bibitem{MR3377562}
Ferenc~A. Bartha and Warwick Tucker, \emph{Fixed points of a destabilized
  {K}uramoto-{S}ivashinsky equation}, Appl. Math. Comput. \textbf{266} (2015),
  339--349. \MR{3377562}

\bibitem{bates1998existence}
Peter~W Bates, Kening Lu, and Chongchun Zeng, \emph{Existence and persistence
  of invariant manifolds for semiflows in {B}anach space}, vol. 645, American
  Mathematical Soc., 1998.

\bibitem{MR3437754}
Maxime Breden, Jean-Philippe Lessard, and Jason~D. Mireles~James,
  \emph{Computation of maximal local (un)stable manifold patches by the
  parameterization method}, Indag. Math. (N.S.) \textbf{27} (2016), no.~1,
  340--367. \MR{3437754}

\bibitem{MR3832468}
Thomas Breunung and George Haller, \emph{Explicit backbone curves from spectral
  submanifolds of forced-damped nonlinear mechanical systems}, Proc. A.
  \textbf{474} (2018), no.~2213, 20180083, 25. \MR{3832468}

\bibitem{cabre2003parameterizationI}
Xavier Cabr{\'e}, Ernest Fontich, and Rafael de~la Llave, \emph{The
  parameterization method for invariant manifolds {I}: manifolds associated to
  non-resonant subspaces}, Indiana University mathematics journal (2003),
  283--328.

\bibitem{cabre2003parameterizationII}
\bysame, \emph{The parameterization method for invariant manifolds {II}:
  regularity with respect to parameters}, Indiana University mathematics
  journal (2003), 329--360.

\bibitem{cabre2005parameterizationIII}
Xavier Cabr{\'e}, Ernest Fontich, and Rafael De~La~Llave, \emph{The
  parameterization method for invariant manifolds {III}: overview and
  applications}, Journal of Differential Equations \textbf{218} (2005), no.~2,
  444--515.

\bibitem{capinski2011cone}
Maciej~J Capi{\'n}ski and Piotr Zgliczy{\'n}ski, \emph{Cone conditions and
  covering relations for topologically normally hyperbolic invariant
  manifolds}, Discrete {\&} Continuous Dynamical Systems - A \textbf{30}
  (2011), no.~3, 641--670.

\bibitem{capinski2015geometric}
\bysame, \emph{Geometric proof for normally hyperbolic invariant manifolds},
  Journal of Differential Equations \textbf{259} (2015), no.~11, 6215--6286.

\bibitem{MR3871613}
Roberto Castelli, Jean-Philippe Lessard, and Jason~D. Mireles~James,
  \emph{Parameterization of invariant manifolds for periodic orbits ({II}): a
  posteriori analysis and computer assisted error bounds}, J. Dynam.
  Differential Equations \textbf{30} (2018), no.~4, 1525--1581. \MR{3871613}

\bibitem{cheng2020stable}
Hongyu Cheng and Rafael de~la Llave, \emph{Stable manifolds to bounded
  solutions in possibly ill-posed {PDE}s}, Journal of Differential Equations
  \textbf{268} (2020), no.~8, 4830--4899.

\bibitem{chicone2006ODE}
Carmen Chicone, \emph{Ordinary differential equations with applications},
  Springer, 2006.

\bibitem{chow1988invariant}
Shui-Nee Chow and Kening Lu, \emph{Invariant manifolds for flows in {B}anach
  spaces}, Journal of Differential equations \textbf{74} (1988), no.~2,
  285--317.

\bibitem{MR3773757}
Jacek Cyranka and Thomas Wanner, \emph{Computer-assisted proof of heteroclinic
  connections in the one-dimensional {O}hta-{K}awasaki {M}odel}, SIAM J. Appl.
  Dyn. Syst. \textbf{17} (2018), no.~1, 694--731. \MR{3773757}

\bibitem{MR2136516}
Sarah Day, Yasuaki Hiraoka, Konstantin Mischaikow, and Toshiyuki Ogawa,
  \emph{Rigorous numerics for global dynamics: a study of the
  {S}wift-{H}ohenberg equation}, SIAM J. Appl. Dyn. Syst. \textbf{4} (2005),
  no.~1, 1--31. \MR{2136516}

\bibitem{Llave2016connecting}
Rafael de~la Llave and Jason~D Mireles~James, \emph{Connecting orbits for
  compact infinite dimensional maps: computer assisted proofs of existence},
  SIAM Journal on Applied Dynamical Systems \textbf{15} (2016), no.~2,
  1268--1323.

\bibitem{MR727816}
J.-P. Eckmann, H.~Koch, and P.~Wittwer, \emph{A computer-assisted proof of
  universality for area-preserving maps}, Mem. Amer. Math. Soc. \textbf{47}
  (1984), no.~289, vi+122. \MR{727816}

\bibitem{MR883539}
Jean-Pierre Eckmann and Peter Wittwer, \emph{A complete proof of the
  {F}eigenbaum conjectures}, J. Statist. Phys. \textbf{46} (1987), no.~3-4,
  455--475. \MR{883539}

\bibitem{eldering2013}
Jaap Eldering, \emph{Normally hyperbolic invariant manifolds; the noncompact
  case}, Springer, Atlantis Press, 2013.

\bibitem{MR1453709}
Zbigniew Galias, \emph{Positive topological entropy of {C}hua's circuit: a
  computer assisted proof}, Internat. J. Bifur. Chaos Appl. Sci. Engrg.
  \textbf{7} (1997), no.~2, 331--349. \MR{1453709}

\bibitem{MR1661345}
Zbigniew Galias and Piotr Zgliczy\'{n}ski, \emph{Chaos in the {L}orenz
  equations for classical parameter values. {A} computer assisted proof},
  Proceedings of the {C}onference ``{T}opological {M}ethods in {D}ifferential
  {E}quations and {D}ynamical {S}ystems'' ({K}rak\'{o}w-{P}rzegorza\l y, 1996),
  no.~36, 1998, pp.~209--210. \MR{1661345}

\bibitem{MR2776917}
Marcio Gameiro and Jean-Philippe Lessard, \emph{Rigorous computation of smooth
  branches of equilibria for the three dimensional {C}ahn-{H}illiard equation},
  Numer. Math. \textbf{117} (2011), no.~4, 753--778. \MR{2776917}

\bibitem{MR3990999}
Javier G\'{o}mez-Serrano, \emph{Computer-assisted proofs in {PDE}: a survey},
  SeMA J. \textbf{76} (2019), no.~3, 459--484. \MR{3990999}

\bibitem{groothedde2017parameterization}
Chris~M Groothedde and JD~Mireles James, \emph{Parameterization method for
  unstable manifolds of delay differential equations}, Journal of Computational
  Dynamics \textbf{4} (2017), no.~1\&2, 21--70.

\bibitem{MR3562433}
George Haller and Sten Ponsioen, \emph{Nonlinear normal modes and spectral
  submanifolds: existence, uniqueness and use in model reduction}, Nonlinear
  Dynam. \textbf{86} (2016), no.~3, 1493--1534. \MR{3562433}

\bibitem{haro2016parameterization}
Alex Haro, Marta Canadell, Jordi-Lluis Figueras, Alejandro Luque, and
  Josep-Maria Mondelo, \emph{The parameterization method for invariant
  manifolds}, Applied mathematical sciences \textbf{195} (2016).

\bibitem{validatedParmMethod_ddes}
Olivier H\'{e}not, J.P. Lessar, and J.~D. Mireles~James, \emph{Parameterization
  of unstable manifolds for ddes: formal series solutions and validated error
  bounds}, (Submitted) (2020).

\bibitem{henry1981geometric}
Daniel Henry, \emph{Geometric theory of semilinear parabolic equations}, vol.
  840, Springer, 1981.

\bibitem{HLM}
A.~Hungria, J.-P. Lessard, and J.D.~Mireles James, \emph{Rigorous numerics for
  analytic solutions of differential equations: the radii polynomial approach},
  Mathematics of Computation \textbf{85} (2016), 1427--1459.

\bibitem{jaquette2020global}
Jonathan Jaquette, Jean-Philippe Lessard, and Akitoshi Takayasu, \emph{Global
  dynamics in nonconservative nonlinear {S}chr{\"o}dinger equations}, arXiv
  preprint arXiv:2012.09734 (2020).

\bibitem{MR2773294}
Tomas Johnson and Warwick Tucker, \emph{A note on the convergence of
  parametrised non-resonant invariant manifolds}, Qual. Theory Dyn. Syst.
  \textbf{10} (2011), no.~1, 107--121. \MR{2773294}

\bibitem{MR3800257}
Florian Kogelbauer and George Haller, \emph{Rigorous model reduction for a
  damped-forced nonlinear beam model: an infinite-dimensional analysis}, J.
  Nonlinear Sci. \textbf{28} (2018), no.~3, 1109--1150. \MR{3800257}

\bibitem{lakshmikantham1988stability}
Vangipuram Lakshmikantham, Srinivasa Leela, and Anatoli{\u \i}~Andreevich
  Martynyuk, \emph{Stability analysis of nonlinear systems}, CRC Press, 1988.

\bibitem{MR759197}
Oscar~E. Lanford, III, \emph{Computer-assisted proofs in analysis}, vol. 124,
  1984, Mathematical physics, VII (Boulder, Colo., 1983), pp.~465--470.
  \MR{759197}

\bibitem{MR775044}
\bysame, \emph{A shorter proof of the existence of the {F}eigenbaum fixed
  point}, Comm. Math. Phys. \textbf{96} (1984), no.~4, 521--538. \MR{775044}

\bibitem{MR3735860}
J.~D. Mireles~James, \emph{Fourier-{T}aylor approximation of unstable manifolds
  for compact maps: numerical implementation and computer-assisted error
  bounds}, Found. Comput. Math. \textbf{17} (2017), no.~6, 1467--1523.
  \MR{3735860}

\bibitem{MR3792792}
\bysame, \emph{Validated numerics for equilibria of analytic vector fields:
  invariant manifolds and connecting orbits}, Rigorous numerics in dynamics,
  Proc. Sympos. Appl. Math., vol.~74, Amer. Math. Soc., Providence, RI, 2018,
  pp.~27--80. \MR{3792792}

\bibitem{MR1276767}
Konstantin Mischaikow and Marian Mrozek, \emph{Chaos in the {L}orenz equations:
  a computer-assisted proof}, Bull. Amer. Math. Soc. (N.S.) \textbf{32} (1995),
  no.~1, 66--72. \MR{1276767}

\bibitem{MR1459392}
\bysame, \emph{Chaos in the {L}orenz equations: a computer assisted proof.
  {II}. {D}etails}, Math. Comp. \textbf{67} (1998), no.~223, 1023--1046.
  \MR{1459392}

\bibitem{MR1808460}
Konstantin Mischaikow, Marian Mrozek, and Andrzej Szymczak, \emph{Chaos in the
  {L}orenz equations: a computer assisted proof. {III}. {C}lassical parameter
  values}, vol. 169, 2001, Special issue in celebration of Jack K. Hale's 70th
  birthday, Part 3 (Atlanta, GA/Lisbon, 1998), pp.~17--56. \MR{1808460}

\bibitem{MR944817}
Mitsuhiro~T. Nakao, \emph{A numerical approach to the proof of existence of
  solutions for elliptic problems}, Japan J. Appl. Math. \textbf{5} (1988),
  no.~2, 313--332. \MR{944817}

\bibitem{MR3971222}
Mitsuhiro~T. Nakao, Michael Plum, and Yoshitaka Watanabe, \emph{Numerical
  verification methods and computer-assisted proofs for partial differential
  equations}, Springer Series in Computational Mathematics, vol.~53, Springer,
  Singapore, [2019] \copyright 2019. \MR{3971222}

\bibitem{MR1100582}
M.~Plum, \emph{Computer-assisted existence proofs for two-point boundary value
  problems}, Computing \textbf{46} (1991), no.~1, 19--34. \MR{1100582}

\bibitem{MR1131109}
Michael Plum, \emph{Verified existence and inclusion results for two-point
  boundary value problems}, Contributions to computer arithmetic and
  self-validating numerical methods ({B}asel, 1989), IMACS Ann. Comput. Appl.
  Math., vol.~7, Baltzer, Basel, 1990, pp.~341--355. \MR{1131109}

\bibitem{reinhardt2019fourier}
Christian Reinhardt and JD~Mireles James, \emph{{F}ourier--{T}aylor
  parameterization of unstable manifolds for parabolic partial differential
  equations: Formalism, implementation and rigorous validation}, Indagationes
  Mathematicae \textbf{30} (2019), no.~1, 39--80.

\bibitem{sell2002dynamics}
George~R Sell and Yuncheng You, \emph{Dynamics of evolutionary equations}, vol.
  143, Springer Science \& Business Media, 2002.

\bibitem{MR3672647}
Robert Szalai, David Ehrhardt, and George Haller, \emph{Nonlinear model
  identification and spectral submanifolds for multi-degree-of-freedom
  mechanical vibrations}, Proc. A. \textbf{473} (2017), no.~2202, 20160759, 19.
  \MR{3672647}

\bibitem{takayasu2019rigorous}
Akitoshi Takayasu, Jean-Philippe Lessard, Jonathan Jaquette, and Hisashi
  Okamoto, \emph{Rigorous numerics for nonlinear heat equations in the complex
  plane of time}, arXiv preprint arXiv:1910.12472 (2019).

\bibitem{MR1701385}
Warwick Tucker, \emph{The {L}orenz attractor exists}, C. R. Acad. Sci. Paris
  S\'{e}r. I Math. \textbf{328} (1999), no.~12, 1197--1202. \MR{1701385}

\bibitem{MR1870856}
\bysame, \emph{A rigorous {ODE} solver and {S}male's 14th problem}, Found.
  Comput. Math. \textbf{2} (2002), no.~1, 53--117. \MR{1870856}

\bibitem{MR3541499}
J.~B. van~den Berg and J.~D. Mireles~James, \emph{Parameterization of
  slow-stable manifolds and their invariant vector bundles: theory and
  numerical implementation}, Discrete Contin. Dyn. Syst. \textbf{36} (2016),
  no.~9, 4637--4664. \MR{3541499}

\bibitem{parmChristian}
J.~B. van~den Berg, J.~D. Mireles~James, and Christian Reinhardt,
  \emph{Computing (un)stable manifolds with validated error bounds:
  non-resonant and resonant spectra}, Journal of Nonlinear Science \textbf{26}
  (2016), 1055--1095.

\bibitem{MR2443030}
Jan~Bouwe van~den Berg and Jean-Philippe Lessard, \emph{Chaotic braided
  solutions via rigorous numerics: chaos in the {S}wift-{H}ohenberg equation},
  SIAM J. Appl. Dyn. Syst. \textbf{7} (2008), no.~3, 988--1031. \MR{2443030}

\bibitem{MR3444942}
\bysame, \emph{Rigorous numerics in dynamics}, Notices Amer. Math. Soc.
  \textbf{62} (2015), no.~9, 1057--1061. \MR{3444942}

\bibitem{MR3518609}
Jan~Bouwe van~den Berg, Jason~D. Mireles~James, and Christian Reinhardt,
  \emph{Computing (un)stable manifolds with validated error bounds:
  non-resonant and resonant spectra}, J. Nonlinear Sci. \textbf{26} (2016),
  no.~4, 1055--1095. \MR{3518609}

\bibitem{MR3904424}
Jan~Bouwe van~den Berg and J.~F. Williams, \emph{Rigorously computing symmetric
  stationary states of the {O}hta-{K}awasaki problem in three dimensions}, SIAM
  J. Math. Anal. \textbf{51} (2019), no.~1, 131--158. \MR{3904424}

\bibitem{berg2017introduction}
J.B. van~den Berg, \emph{Introduction to rigorous numerics in dynamics: general
  functional analytic setup and an example that forces chaos}, Rigorous
  numerics in dynamics, Proc. Sympos. Appl. Math., vol.~74, Amer. Math. Soc.,
  Providence, RI, 2018, pp.~1--25.

\bibitem{BGLV}
J.B. van~den Berg, M.~Gameiro, J.-P. Lessard, and R.C. van~der Vorst,
  \emph{Towards computational {M}orse-{F}loer homology: forcing results for
  connecting orbits by computing relative indices of critical points}, 2020, In
  preparation.

\bibitem{StableManifoldCode}
J.B. {v}an~{d}en Berg, J.~Jaquette, and J.D. Mireles~James, \emph{Matlab codes
  of ``{V}alidated numerical approximation of stable manifolds for parabolic
  partial differential equations''},
  \url{https://github.com/JCJaquette/Validated-Numerical-Approximation-of-Stable-Manifolds-for-Parabolic-PDEs},
  2020.

\bibitem{JBJF3D}
J.B. van~den Berg and J.F. Williams, \emph{Optimal periodic structures with
  general space group symmetries in the {O}hta-{K}awasaki problem}, 2019,
  Preprint.

\bibitem{danielKS_chaosProof}
Daniel Wilczak and Piotr Zgliczy\'{n}ski, \emph{A geometric method for
  infinite-dimensional chaos: symbolic dynamics for the
  {K}uramoto-{S}ivashinsky pde on the line}, Journal of Differential Equations
  \textbf{269} (2020), no.~10, 8509--8548.

\bibitem{zgliczy2009covering}
Piotr Zgliczynski, \emph{Covering relations, cone conditions and the stable
  manifold theorem}, Journal of Differential Equations \textbf{246} (2009),
  no.~5, 1774--1819.

\bibitem{MR1838755}
Piotr Zgliczy\'{n}ski and Konstantin Mischaikow, \emph{Rigorous numerics for
  partial differential equations: the {K}uramoto-{S}ivashinsky equation},
  Found. Comput. Math. \textbf{1} (2001), no.~3, 255--288. \MR{1838755}

\end{thebibliography}

\appendix

\section{General Strategy for Bootstrapping Gronwall's Inequality }
\label{sec:GeneralBootstrap}

We generalize the bootstrapping argument used in Section \ref{sec:ExponentialTracking} so that it can be applied in Section \ref{sec:LyapunovPerron} and Section \ref{sec:Contraction}. 
To unify the class of functions we wish to bound, and the set of assumptions we make on these functions, we define Condition \ref{anz:IntegralInequality} below. 
In a slight abuse of notation, here we define $ \cB$ to be a tensor, distinct from its previous usage as a ball of functions in Definition \ref{def:Ball_of_Functions}. 


\begin{condition}
	\label{anz:IntegralInequality}
	Fix $ \lambda_1,  \dots, \lambda_{N_{\lambda}} \in \R $, fix $H \in \R^{N_{\lambda}} \otimes \R^{N_{\lambda}} $ and define $ \gamma_k := \lambda_k + H_k^k$ for $ 1 \leq k \leq N_\lambda$. 
	For $ N_\mu \in \N$, fix some $ \mu_k \in \R$ for $1 \leq k \leq N_\mu$.   
	Assume that $ \{ \gamma_j \}_{j=1}^{N_\lambda} \subseteq \{ \mu_k \}_{k=1}^{N_\mu}$, and suppose that both $ \gamma_{k} >  \gamma_{k+1}$ and $\mu_{k} > \mu_{k+1}$.  Assume further that $ \mu_1 > \gamma_1$. 
	
	For $M \in \N$, and  $N_i \in \N $ for $ 1 \leq i \leq M$ and basis elements 
	$ e_{n_i} \in \R^{N_i}$ where $ 1 \leq n_i \leq N_i$, we fix tensors 
	\begin{align*}
		\cA &\in  
		\big(
		\bigotimes_{i=1}^M \R^{N_i} 
		\big) \otimes  \R^{N_\lambda} \otimes \R^{N_\mu} 
,
		&
		\cB &\in \big( 
		\bigotimes_{i=1}^M \R^{N_i} \big)
		\otimes  \R^{N_\lambda}
	\end{align*}
	component-wise by  
	\begin{align*}
	\cA_{j,k} &:=  A_{j,k}^{n_1 \dots n_M} \cdot  e_{n_1} \otimes \dots \otimes e_{n_M}  ,
	&
	\cB_j &:= B_j^{n_1 \dots  n_M} \cdot   e_{n_1} \otimes \dots \otimes e_{n_M} .
	\end{align*}

	For this arrangement of constants, we say that a pair $(u,\omega)$ satisfies
	Condition~\ref{anz:IntegralInequality} on a time interval $[0,T]$
	if the functions $u=(u_j)_{j=1}^{N_\lambda}$ and the positive tensor $\omega  \in \bigotimes_{i=1}^M \R^{N_i} $ satisfy the inequalities
	\begin{align} \label{eq:GeneralIntegralInequality}
	e^{ - \lambda_j t} u_j(t) &\leq 
	\cB_j \omega  + \int_0^t e^{-\lambda_j \tau} 
	\sum_{0 \leq k \leq N_{\mu}}  e^{\mu_k \tau} \cA_{j,k} \omega  \, d \tau 
	+ \int_0^t e^{-\lambda_j \tau} H_j^i u_i(\tau) \, d\tau  
	\qquad \text{for all } t \in [0,T].
	\end{align}
\end{condition}

In all cases where we consider constants satisfying Condition \ref{anz:IntegralInequality}, we take $N_\lambda = m_s$,  and  $ \lambda_1 , \dots , \lambda_{N_\lambda}$ as in \eqref{eq:stableEigenvalueEstimate}, and $H_j^i$ as in Definition \ref{def:H}. 
Hence, the definition of $ \gamma_k$ here coincides with that given in Definition \ref{def:Gamma}. 
For the other variables, we take them in the various sections according to the following table. 

\begin{table}[h!]
	\begin{center}
		\label{tab:table1}
		\begin{tabular}{c|c|c|c} 
			 & Section \ref{sec:ExponentialTracking} & Section \ref{sec:LyapunovPerron} & Section \ref{sec:Contraction} \\
			\hline 
			&&& \\
			$u_j$ &  
			$ | x_j(t,\xi,\alpha) - x_j(t,\zeta,\alpha)|$ &
			$ \| \partial_{i} x_j(t,\eta,\alpha) - \partial_{i}  x_j(t,\zeta,\alpha)\|$   &
			$ | x_j(t , \xi , \alpha) - x_j(t, \xi , \beta)|$ \\[1mm]
			$\omega$ &
			$ | \xi_n - \zeta_n|$   &
			 $ | \eta_l - \zeta_l|$  &
			 $|\xi_{n_1}| \otimes \| \alpha - \beta \|_{n_2', \cE}^{n_3}$  \\
			$\cA_{j,k}$ & 
			0  & 
			$S_j^{nm} G_{m,k_1}^{l} G_{n,k_2}^i$ &
			$C_j^{n_2'} G_{n_3, k}^{n_1}$ \\[1mm]
			$\cB_j$ & $\delta_j^n$ & 0 & 0\\[1mm]
			$\{ \mu_k \}$ & 
			 $\{ \gamma_k\}_{k=0}^{m_s}$ & 
			  $\{ \gamma_k \}_{k=0}^{m_s} \cup \{ \gamma_{k_1} + \gamma_{k_2}  \}_{k_1,k_2=0}^{m_s}$ & 
			  $\{ \gamma_k\}_{k=-1}^{m_s}$ \\
		\end{tabular}
	\end{center}
\end{table}
We note that for $\cA_{j,k}$ in Section \ref{sec:LyapunovPerron} we use a double index $(k_1,k_2)$ to index over the elements of $\{\mu_k \}$. 
For a system given as in Condition \ref{anz:IntegralInequality} we are interested in finding a tensor $\cG$ satisfying Condition \ref{cond:GeneralGbounds} below.  
%
%

\begin{condition}
	\label{cond:GeneralGbounds}
	Given $\mu$ as in Assumption~\ref{anz:IntegralInequality} and a pair $(u,\omega)$ of functions $u=(u_j)_{j=1}^{N_\lambda}$ on $[0,T]$ 
	and a positive tensor $\omega  \in \bigotimes_{i=1}^M \R^{N_i} $, 
	we say  that the tensor $ \cG \in \big( \bigotimes_{i=1}^M \R^{N_i}  \big) \otimes \R^{N_\lambda} \otimes \R^{N_\mu} $ with components
	\[
	\cG_{j,k} :=  G_{j,k}^{n_1 \dots  n_M} e_{n_1} \otimes \dots \otimes e_{n_M},
	\]
	satisfies Condition \ref{cond:GeneralGbounds} if $u_j(t) \leq \sum_{k=1}^{N_\mu} e^{\mu_k t} \cG_{j,k} \omega  $ for all
	$t\in [0,T]$.
\end{condition}

From these two conditions, we can bootstrap our bounds on a tensor $\cG$. 

\begin{proposition}
	\label{prop:GeneralImprovementOperator}
	Assume the pair $(u,\omega)$ satisfies Condition~\ref{anz:IntegralInequality} on $[0,T]$ and assume $ \cG$ satisfies Condition~\ref{cond:GeneralGbounds}. Fix $1 \leq j  \leq N_\lambda $.  
	If  $\cA_{j,k}=0$ and $ \cG_{i,k}=0$ whenever $\mu_k= \gamma_j$, 
	then we have:  
	\begin{align} \label{eq:pre_Tj_General}
	u_j(t) &\leq 
e^{\gamma_j t} 	\cB_j \omega + 
	\sum_{\substack{1 \leq k \leq N_\mu\\ \mu_k \neq \gamma_j} }
	\frac{e^{\mu_k t} - e^{\gamma_j t}}{\mu_k - \gamma_j}
	\Big(\cA_{j,k} +
	\sum_{\substack{1 \leq i \leq N_\lambda \\ i \neq j }}
	H_j^i \cG_{i,k}  	
	\Big) \omega  \qquad\text{for all }
	t\in [0,T].
	\end{align}	
	In other words, define a map $ \cT_{j,k} : \big( \bigotimes_{i=1}^M \R^{N_i}  \big) \otimes \R^{N_\lambda} \otimes \R^{N_\mu} \to 
	\bigotimes_{i=1}^M \R^{N_i}   $  by:
	\begin{align}
	\cT_{j,k}(\cA, \cB,\cG)  := 
	\begin{dcases}
	(\mu_k - \gamma_j)^{-1}  
	\Big( \cA_{j,k} +   \sum_{\substack{1 \leq i \leq N_{\lambda} \\ i \neq j }}
	H_j^i \cG_{i,k} \Big)
	& \mbox{if } \mu_k \neq \gamma_j \\
	\cB_j  - \sum_{\substack{0 \leq m \leq N_\mu \\ \mu_m \neq \gamma_j }} 
	(\mu_m - \gamma_j)^{-1}  
	\Big(
	\cA_{j,m}+
	\sum_{\substack{1 \leq i \leq N_{\lambda} \\ i \neq j }}
	H_j^i \cG_{i,m}
	\Big)
	& \mbox{if } \mu_k=\gamma_j.
	\end{dcases}
	\end{align}
	Then $\cG$ also satisfies  Condition  \ref{cond:GeneralGbounds} if we replace $ \cG_{j,k}$ by  $\cT_{j,k}
	(\cA, \cB,\cG)$ for all $k$. 
\end{proposition}

\begin{proof}[Proof of Proposition \ref{prop:GeneralImprovementOperator}]
	
	Splitting $ H_j^i u_i = \sum_{i\neq j} H^i_j u_i +  H_j^j u_j $, we  write \eqref{eq:GeneralIntegralInequality} as 
	\begin{align*}  
	e^{ - \lambda_j t} u_j(t) &\leq 
	\cB_j \omega 
	+ \int_0^t e^{-\lambda_j \tau} v(\tau,\omega)  d\tau 
	+ \int_0^t e^{-\lambda_j \tau} H_j^j u_j(\tau) d\tau .
	\end{align*}
	where 
	\begin{align*}
	v(\tau,\omega) &= 
	\sum_{\substack{1 \leq k \leq N_\mu\\ \mu_k \neq \gamma_j} }e^{\mu_k \tau} \cA_{j,k} \omega + \sum_{\substack{ 1 \leq i \leq N_\lambda \\ i \neq j }}  H_j^i u_i (\tau) .
	\end{align*}
	By plugging in the bound assumed in Condition \ref{cond:GeneralGbounds},  we obtain 
	\begin{align*}
	v(\tau,\omega) &\leq 
	\sum_{\substack{1 \leq k \leq N_\mu\\ \mu_k \neq \gamma_j} }e^{\mu_k \tau}
	\Big(
	\cA_{j,k} \omega + \sum_{\substack{ 1 \leq i \leq N_\lambda \\ i \neq j }}  H_j^i \cG_{i,k} \omega
	\Big) .
	\end{align*}
	By applying Lemma \ref{prop:Fundamental} we obtain \eqref{eq:pre_Tj_General}.  
\end{proof}

In order to obtain tensors satisfying the requirement that  $\cA_{j,k},  \cG_{i,k}=0$ whenever $\mu_k= \gamma_j$, we define an operator $\cQ_j$ as below.

\begin{proposition}
	\label{prop:GeneralProjection}
	Fix $1 \leq j \leq N_{\lambda}$ and define a map 
	$ \cQ_j : 
	 \big( \bigotimes_{i=1}^M \R_+^{N_i}  \big) \otimes \R^{N_\lambda} \otimes \R^{N_\mu} 
	\to 
	 \big( \bigotimes_{i=1}^M \R_+^{N_i}  \big) \otimes \R^{N_\lambda} \otimes \R^{N_\mu}$ by 
	\[
	\cQ_j
	(\cG )_{i,k}^{n_{1} \dots  n_{M}}   := 
	\begin{cases}
	0 		&	\mbox{if } \mu_k= \gamma_j \\
	G_{i,k}^{n_{1} \dots  n_{M}} + G_{i,(k+1)}^{n_{1} \dots  n_{M}} &	\mbox{if } \mu_{k+1}= \gamma_j, \mbox{and } G_{i,(k+1)}^{n_{1} \dots  n_{M}} >0 \\
	G_{i,k}^{n_{1} \dots  n_{M}} + G_{i,(k-1)}^{n_{1} \dots  n_{M}} &	\mbox{if } \mu_{k-1}= \gamma_j, \mbox{and } G_{i,(k-1)}^{n_{1} \dots  n_{M}} <0 \\
	G_{i,k}^{n_{1} \dots  n_{M}} & \mbox{otherwise.}
	\end{cases}		
	\]
	Then $ \cQ_j
	(\cG )_{i,k} =0$ whenever $ \mu_k = \gamma_j$.
	Furthermore, if   $\cG$ satisfies Condition  \ref{cond:GeneralGbounds} then  $\cQ_j
	(\cG )$ satisfies Condition \ref{cond:GeneralGbounds}. 
\end{proposition}

We are able to generalize Algorithm \ref{alg:BootStrap} as follows.
\begin{algorithm}
	\label{alg:GeneralBootStrap}
	Take as input all the constants in Condition \ref{anz:IntegralInequality},  an input tensor $\widehat{\cG}$  satisfying Condition  \ref{cond:GeneralGbounds}, and a computational parameter $N_{bootstrap}$. 
	The algorithm outputs a tensor $\cG$.

\begin{algorithmic}
	\STATE $\cG \gets \widehat{\cG}$
	\FOR {$1 \leq i \leq  N_{\text{bootstrap}}$} 
	\FOR {$1 \leq j \leq m_s $} \vspace{.1cm}
	\STATE $\cG_{j,k} 	\gets \cT_{j,k}( \cQ_j (\cA), \cB,\cQ_j (\cG))$
	\vspace{.1cm}
	\ENDFOR
	\ENDFOR 
	\RETURN $\cG$ 
\end{algorithmic}
	
\end{algorithm}

\begin{proposition} \label{prop:GeneralAlgorithmResult}
	If the input tensor  $\widehat{\cG}$ to Algorithm \ref{alg:GeneralBootStrap}  satisfies  Condition \ref{cond:GeneralGbounds}, then the output tensor  $\cG$ satisfies Condition \ref{cond:GeneralGbounds}. 
\end{proposition}

The proof of Proposition \ref{prop:GeneralProjection}  follows from the assumption that $\mu_{k} > \mu_{k+1}$. 
The proof of  Proposition \ref{prop:GeneralAlgorithmResult} follows from an induction argument which uses Proposition \ref{prop:GeneralImprovementOperator} for the inductive step. 
Both proofs are left to the reader.


\section{Semigroup Estimates for Fast-Slow Systems}
\label{sec:EstimatingEigenvalues} 

In equation \eqref{eq:TOTALstableEigenvalueEstimate} we require constants $C_s, \lambda_s $ satisfying 
\begin{alignat}{2}\label{e:goalappendix}
|e^{(\Lambda_s + L_s^s)t } \xx_{s} | &\leq  C_s e^{\lambda_s t } |\xx_{s}|,  &\qquad&
t \geq 0, \xx_{s} \in X_s  .
\end{alignat}
Our assumption that  $ \lambda_s <0$, and moreover that $  \gamma_0 = \lambda_s + C_s \hat{\cH} < 0 $, is essential.
In Proposition \ref{prop:StayInsideBall} this is used to prove that solutions $x(t, \xi,\alpha) $ stay inside the ball $B_s(\rho)$ for all $ t\geq 0$. 
While our method of bootstrapping Gronwall's inequality greatly mitigates the effect of these constants $ C_s, \lambda_s$ on our final estimates, for the Lyapunov-Perron operator to be well defined it is essential that we prove $\gamma_0 <0$.  

There are two types of estimates which we will apply to obtain 
pairs $(C_s,\lambda_s)$ satisfying~\eqref{e:goalappendix}.
First, for linear operators $ A,B \in \cL(X,X) $ with $| e^{At} \xx | \leq k e^{\lambda t} |\xx|$ for all $\xx\in X$ and $t \geq 0$, and $ \|B\| <\infty$, we have (the proof is analogous to the one of Proposition~\ref{prop:InitialBound})
\begin{align} \label{eq:SemiGroup_Estimate_1}
| e^{(A+B)t} \xx | 
&\leq k  e^{(\lambda + k \|B\|) t} |\xx|, \qquad\text{for all }
t \geq 0, \xx \in X.
\end{align}
This estimate by itself is not enough, as the largest eigenvalue of $ \Lambda_s$ is often small in comparison with $\| L_s^s \|$. 
For example, in Section \ref{sec:SH_Linear} we showed that $  |e^{\Lambda_i t} \xx_i| \leq e^{\lambda_i t} | \xx_i| $ and  $ \| L_j^i \| \leq D_j^i$ with values 
\begin{align*}
\lambda_1 &= -1.41,
&
\lambda_2 &= -4.58  \times  10^{4} , 
&
D_s^s &= 
\left(
\begin{matrix}
4 \times 10^{-10} & 1.6 \\
1.6 & 5.7
\end{matrix}
\right) .
\end{align*} 
%
Since $  \lambda_1  + \| L_s^s\| > 0 $, just an estimate of the type in \eqref{eq:SemiGroup_Estimate_1} with $A$ the diagonal part of $D_s^s$ and $B$ the off-diagonal part
will not suffice.  
We further note that our estimates for $ D_s^s$  do not improve with a larger Galerkin projection dimension. 
%
%
Hence we want to change basis to diagonalize $\Lambda_s + L_s^s$, at least approximately,
and then take advantage of the identity $e^{PJP^{-1}t}=P e^{Jt}P^{-1}$ in our estimates.  To motivate our construction, we first consider a
$2 \times 2$ matrix  
\begin{align*}
M&=
\left(
\begin{matrix}
\lambda_1 & \delta_b \\
\delta_c  & \lambda_\infty 
\end{matrix}
\right)  .
\end{align*} 
If $\lambda_\infty$ is much larger in absolute value than the other matrix entries, then the eigenvalues of $M$ are approximately given by $ \lambda_1$ and $ \lambda_\infty$. 
In particular, if $| \delta_b \delta_c| < | \lambda_1 \lambda_\infty|$  and $ \lambda_1,\lambda_\infty <0$, then all of the eigenvalues of $M$ have negative real part.
Below in Theorem \ref{prop:FastSlowExponentialEstimate} we prove an analogous theorem where we replace $ \lambda_1$ by a finite dimensional matrix, and $ \lambda_\infty$ by an infinite dimensional linear operator. This is the second type of estimate that we use to find pairs $(C_s,\lambda_s)$ satisfying~\eqref{e:goalappendix}.

%

\begin{theorem}
	\label{prop:FastSlowExponentialEstimate}
	
	Consider  Banach spaces $ \C^N$ and $X_\infty$ with arbitrary norms, and their product  $ \C^N \times X_\infty$ with norm $ | (x_N,x_\infty)| = (|x_N|^p + |x_\infty|^p)^{1/p}$ for any $ 1 \leq p \leq \infty$. 
	
	Consider the linear operators $M,\Lambda ,L : \C^N \times X_\infty \to  \C^N \times X_\infty$   given by
	\begin{align}
	M&= \Lambda+L ,
	&
	\Lambda &= 
	\left(
	\begin{matrix}
	\Lambda_1 & 0\\
	0 & \Lambda_{\infty}
	\end{matrix}
	\right),
	&
	L &= 
	\left(
	\begin{matrix}
	L_1^1 & L_1^\infty \\
	L_\infty^1 & L_{\infty}^\infty
	\end{matrix}
	\right) .
	\label{eq:StiffMatrix}
	\end{align} 
We require $\Lambda$ to be densely defined and $L$ to be bounded.  
	Suppose that $ \Lambda_1$ is \emph{diagonal} and that $\Lambda_\infty$ has a bounded inverse. 
	 
	Fix constants $\mu_1,\mu_\infty,C_1,C_\infty \in \R$ such that for all $t \geq 0$ we have
	\begin{align*}
	\| e^{\Lambda_1 t} \| &\leq C_1 e^{\mu_1 t} , &
    \| e^{\Lambda_\infty t} \| &\leq C_\infty e^{\mu_\infty t} .
	\end{align*}
	Fix constants $\delta_1, \delta_b,\delta_c,\delta_d,\varepsilon >0$ such that 
	\begin{align*}
	\|L_1^1 \| &\leq \delta_a,  
	&
	\| L_1^\infty\| & \leq \delta_b ,
	&
	\| L_\infty^1\| &\leq  \delta_c ,
	&
	\| L_\infty^\infty \| &\leq  \delta_d ,
	\end{align*} 
	and set
	\begin{align*}
	\varepsilon &:= \sum_{ \lambda \in \sigma(\Lambda_1) }  \frac{ \| \Lambda_\infty^{-1} \|    }{1 - \| \Lambda_\infty^{-1} \|   ( \delta_d + | \lambda| )}   .
	\end{align*}
	Assume that the inequalities 
	\begin{align}
	 \| \Lambda_\infty^{-1} \|  \left( \delta_d +  \sup_{\lambda_k \in \sigma(\Lambda_1)}  | \lambda_k| \right) & <1 ,
	&
 \mu_\infty + C_\infty \left(\delta_d+ \varepsilon \delta_b \delta_c (1 + \varepsilon^2 \delta_b \delta_c)\right)  <	\mu_1   , \label{eq:SemiGroupEstimateHypothesis}
	\end{align}
are satisfied. Then we have
	\[
	\| e^{M t} \| \leq C_s e^{\lambda_s t},
	\]
	where 
	\begin{align*}
	C_s &:=  (1+\varepsilon \delta_b)^2(1+\varepsilon \delta_c)^2 \max\{C_1,C_\infty \} \\
	\lambda_s &:=  \mu_1 +  C_s \delta_a + \Delta \max\{C_1,C_\infty \}\\
	\Delta &:= \varepsilon \delta_b \delta_c  \left( 1 + \varepsilon(2 \delta_b +  \delta_c) + \varepsilon^2 \delta_b \delta_c (1 + \varepsilon \delta_b) \right).
	\end{align*}	
\end{theorem}

First we prove a lemma for general Banach spaces which allows us to approximately diagonalize  our matrix. 
When $|\cdot|$ denotes the norm on a Banach space, then by $|\cdot|_*$ we denote the norm on its dual.

\begin{lemma}
	\label{prop:ApproximateDiagonalization}
	For a Banach space $ X_\infty$ consider the linear operator $M_1: \C^N \times X_\infty  \to \C^N \times X_\infty $ defined as 
	\[
	M_1= 
	\left(
	\begin{matrix}
	A & B \\
	C &D
	\end{matrix}
	\right) .
	\]
	Suppose that $ \sigma(A) \cap \sigma(D) = \emptyset$ and that  $A$ has  distinct eigenvalues $\lambda_1 ,\dots, \lambda_N$ with eigenvectors $ v_1 , \dots,  v_N$, and dual eigenvectors $u_1 , \dots , u_N$ (the corresponding eigenvectors of $A^*$). 
	Normalize the vectors so that 
	 $u^*_i v_j =\delta_{ij}$,
	the Kronecker delta.

We define $ W_b :  X_\infty \to \C^N $ and $ W_c : \C^N \to X_\infty$ 	as a sum of products between vectors in their codomains, and dual vectors acting on their domains:
	\begin{align*}
	W_b  &:= \sum_{k=1}^N v_k  \left[ (D^* - \lambda^*_k I_{\infty})^{-1} B^* u^*_k \right],  &
	W_c    &:= \sum_{k=1}^N - \left[ (D - \lambda_k I_{\infty})^{-1} C v_k \right]  u^*_k ,
	\end{align*}
	where $ D^* : X_\infty^* \to X_\infty^*$ and $ B^* :(\C^N)^* \to X_{\infty}^*$ are the dual transformations. 
	Define invertible operators $P_b,P_c: \C^N \times X_\infty  \to \C^N \times X_\infty $   by   
	\begin{align*}
	P_b &= \left(
	\begin{matrix}
	I_{N} & W_b
	\\
	0 & I_{\infty}
	\end{matrix}
	\right)
	&
P_c &= \left(
\begin{matrix}
I_{N} & 0 \\
W_c & I_{\infty}
\end{matrix}
\right) .
	\end{align*}
Then 
	\[
	(P_cP_b)^{-1} M_1 ( P_cP_b) = 
	\left(
	\begin{matrix}
	A & 0 \\
	0 &D
	\end{matrix}
	\right) + E,
	\]
	where 
	\[
	E = 
		\left(
	\begin{matrix}
	 (I_N + W_b W_c) B W_c & B W_c W_b + W_bW_cB (I + W_c W_b) \\
-W_c B W_c & -W_c B ( I_{\infty} + W_c W_b )
	\end{matrix}
	\right) .
	\]
\end{lemma}

\begin{proof}

	First we show that 
	\begin{align}
P_b^{-1} 
\left(
\begin{matrix}
A & B \\
0 & D
\end{matrix}
\right) 
P_b &= 
\left(
\begin{matrix}
A & 0 \\
0 & D
\end{matrix}
\right) ,
&
P_c^{-1} 
\left(
\begin{matrix}
A & 0 \\
C & D
\end{matrix}
\right) 
P_c &= 
\left(
\begin{matrix}
A & 0 \\
0 & D
\end{matrix}
\right) .
\label{eq:SemiGroupCoV}
	\end{align}
	We begin with the second equality in \eqref{eq:SemiGroupCoV}, and calculate 
	\begin{align*}
	P_c^{-1} 
	\left(
	\begin{matrix}
	A & 0 \\
	C & D
	\end{matrix}
	\right) 
	P_c &= \left(
	\begin{matrix}
	A & 0 \\
	-W_cA + C +DW_c & D
	\end{matrix}
	\right) .
	\end{align*}
	We compute the action of $-W_cA + C +DW_c$ on an eigenvector $ v_k$ of $A$ as follows:  
	\begin{align*}
	( 	-W_cA + C +DW_c ) v_k &=  C v_k +( D -\lambda_k I_{\infty} )W_c v_k  .
	\end{align*}
To see that the right hand side is equal to zero, we calculate, using 
$u^*_i v_j =\delta_{ij}$, 
	\[
	W_c v_k = - \left(  D - \lambda_k I_{\infty}  \right)^{-1} C v_k .
	\]
	Since the eigenvectors $ v_1 \dots v_N$ span $\C^N$, 
	then  $-W_c A + C +DW_c  =0$, yielding the desired equality. 
	
	The argument is analogous for the first identity in \eqref{eq:SemiGroupCoV}. 
	Again we begin by calculating
		\begin{align*}
	P_b^{-1} 
	\left(
	\begin{matrix}
	A & B \\
	0 & D
	\end{matrix}
	\right) 
	P_b &= \left(
	\begin{matrix}
	A & AW_b + B -W_bD \\
	0 & D
	\end{matrix}
	\right) .
	\end{align*}
	Hence, we would like to show the map $(AW_b + B -W_bD): X_\infty \to \C^N$ is the zero map, which we do by arguing that $u^*_k (AW_b + B -W_bD) =0$ for all $k$. The latter follows from a calculation similar to the one performed above.
	
Finally, we calculate $(P_cP_b)^{-1} M_1 P_cP_b$ as follows: 
	\begin{align*}
			(P_cP_b)^{-1} M_1 ( P_cP_b) &= 
			P_b^{-1} \left( 
				\left(
			\begin{matrix}
			A & 0 \\
			0 &D
			\end{matrix}
			\right)
			+
			P_c^{-1}
				\left(
			\begin{matrix}
			0 & B \\
			0 & 0
			\end{matrix}
			\right)
			P_c
			 \right)
			 P_b
			 \\
			  &= 
			 P_b^{-1} \left( 
			 \left(
			 \begin{matrix}
			 A & B \\
			 0 &D
			 \end{matrix}
			 \right)
			 +
			 \left(
			 \begin{matrix}
			 B W_c & 0 \\
			 -W_c B W_c & -W_c B
			 \end{matrix}
			 \right)
			 \right)
			 P_b
			 \\
			 &= 
			 \left(
			 \begin{matrix}
			 A & 0 \\
			 0 &D
			 \end{matrix}
			 \right)
			 +
			 \left(
			 \begin{matrix}
			 (I_N + W_b W_c) B W_c & B W_c W_b + W_bW_cB (I + W_c W_b) \\
			 -W_c B W_c & -W_c B ( I_{\infty} + W_c W_b )
			 \end{matrix}
			 \right). \qedhere
	\end{align*}

\end{proof}

\begin{proof}[Proof of Theorem \ref{prop:FastSlowExponentialEstimate}.]
	
	Let $ M = M_1 + M_2 $, where  
\begin{align*}
	M_1 &:= \left(
\begin{matrix}
A&B\\
C&D
\end{matrix}
\right)  :=
	\left(
	\begin{matrix}
	\Lambda_1  & L_1^\infty \\
	L_\infty^1 & \Lambda_\infty + L_\infty^\infty 
	\end{matrix}
	\right), 
	&
	M_2 &:= 
	\left(
	\begin{matrix}
	L_1^1  & 0 \\
	0 &  0
	\end{matrix}
	\right) .
\end{align*}	

We will apply Lemma \ref{prop:ApproximateDiagonalization} to the matrix $M_1$.
Since we have assumed that $\Lambda_1$ is diagonal we may take $u_k=v_k=e_k$, the standard basis vectors in $\C^{N}$.
We begin by proving $ \| W_b\|  \leq \varepsilon \delta_b$ and $\| W_c \| \leq \varepsilon \delta_c$. 
We first calculate
\begin{align*}
	(D-\lambda_k I_\infty)^{-1} =
	( \Lambda_\infty + L_\infty^\infty  - \lambda_k I_{\infty} )^{-1}
	&= ( I_{\infty} + \Lambda_\infty^{-1} ( L_\infty^\infty -\lambda_k I_{\infty}))^{-1}	\Lambda_\infty^{-1} .
\end{align*}
By our hypothesis, we are allowed to apply the Neumann series and we obtain
\begin{align}
\label{eq:SemiGroupD_lambdaBound}
\|	(D-\lambda_k I_\infty)^{-1} \|
& \leq \frac{\| \Lambda_\infty^{-1}\|}{1 - \| \Lambda_\infty^{-1}\|  ( \delta_d + | \lambda_k| )}.
\end{align}
We note that the same estimate holds for the dual operator $(D^*-\lambda_k^* I_\infty)^{-1}$. 

We now show that $\|W_b\| \leq \varepsilon \delta_b$. Namely, by using that $\|u^*_k\|_{(\C^N)^*}=\|v_k\|_{\C^N}=1$ we find that
\begin{align*}
	\| W_b \| &= \sup_{ x \in X_\infty, \| x\| =1} 
\Big\| \sum_{ \lambda_k \in \sigma(\Lambda_1) } 
v_k \left[(D^* - \lambda^*_k I_{\infty} )^{-1} B^* u^T_k\right] x
 \Big\|_{\C^N} \\
 &\leq \sup_{ x \in X_\infty, \| x\| =1} 
\sum_{ \lambda_k \in \sigma(\Lambda_1) } 
 \Big|  \left[(D^* - \lambda^*_k I_{\infty} )^{-1} B^* u^T_k\right] x
 \Big|  \\
 &\leq   
 \sum_{ \lambda_k \in \sigma(\Lambda_1) } 
 \Big\|   (D^* - \lambda^*_k I_{\infty} )^{-1} B^* 
 \Big\|_{\cL( (\C^{N})^* ,X_\infty^*) }  \\
 &\leq \| B^* \| \sum_{ \lambda_k \in \sigma(\Lambda_1) }  \frac{ \| \Lambda_\infty^{-1} \|    }{1 - \| \Lambda_\infty^{-1} \|   ( \delta_d  + | \lambda_k| )}   .
\end{align*}
Hence, by plugging in $\| B^*\| = \|L_1^\infty\|$ we obtain $ \| W_b \|  \leq \varepsilon \delta_b $. The proof of the estimate
$\| W_c\| \leq \varepsilon \delta_c$
is analogous.
Next, we note that
\begin{align*}
\|P_b\| ,\|P_b^{-1}\| &\leq 1+\varepsilon \delta_b &
\|P_c\| ,\|P_c^{-1}\| &\leq 1+\varepsilon \delta_c .
\end{align*}
By Lemma \ref{prop:ApproximateDiagonalization} we have
	\begin{align} \label{eq:SemiGroup_Conjugate}
			(P_cP_b)^{-1}( M_1 + M_2)( P_cP_b) 
		&= M_3 +M_4 + (P_b P_b)^{-1} M_2 (P_c P_b),
	\end{align}
	where 
	\begin{align*}
M_3 &:= 
\left(
\begin{matrix}
\Lambda_1   & 0 \\
0 & \Lambda_\infty + L_\infty^\infty -W_c L_1^\infty  ( I_{\infty} + W_c W_b )
\end{matrix}
\right) , \\
M_4 &:= 
\left(
\begin{matrix}
(I_N + W_b W_c) L_1^\infty  W_c & L_1^\infty  W_c W_b + W_bW_c L_1^\infty  (Id + W_c W_b) \\
-W_c L_1^\infty  W_c & 0
\end{matrix}
\right).
	\end{align*}
	For $ ( x_N , x_\infty) \in \C^N \times X_\infty$ we see that
	\begin{align*}
		e^{M_3 t} (x_N , x_\infty) &= \left(e^{\Lambda_1 t}x_N, e^{(\Lambda_\infty + L_\infty^\infty -W_c L_1^\infty  ( I_{\infty} + W_c W_b )) t} x_\infty \right).
	\end{align*}
We also have $ \| L_\infty^\infty -W_c L_1^\infty  ( I_{\infty} + W_c W_b ) \| \leq  \delta_d + \varepsilon \delta_b \delta_c (1 + \varepsilon_b \varepsilon_c)$. 
By applying the estimate \eqref{eq:SemiGroup_Estimate_1} 
	we obtain, for all $t \geq 0$,
	\begin{align*}
		\| e^{\Lambda_1 t} x_N \| & \leq C_1 e^{\mu_1 t}\| x_N \| , \\
		\| e^{(\Lambda_\infty + L_\infty^\infty -W_c L_1^\infty  ( I_{\infty} + W_c W_b )) t} x_\infty \| & \leq C_\infty e^{(\mu_\infty + C_\infty[ \delta_d + \varepsilon \delta_b \delta_c (1 + \varepsilon_b \varepsilon_c)]) t} \|x_\infty\|. 
	\end{align*}	
From our assumption in \eqref{eq:SemiGroupEstimateHypothesis} that
	$
	\mu_1   > \mu_\infty + C_\infty [\delta_d + \varepsilon \delta_b \delta_c (1 + \varepsilon^2 \delta_b \delta_c)]
	$,
we obtain, for any $p$-norm, $1 \leq p \leq \infty$, on the product $ \C^N \times X_\infty$,  
	\[
	\| e^{ M_3 t } (x_N,x_\infty) \| \leq  \max \{ C_1, C_\infty\} e^{\mu_1 t} \| (x_N,x_\infty) \|.
	\]	
	We may estimate the norm of the components of $M_4$ as
\begin{align*}
\|(I_N + W_b W_c) L_1^\infty  W_c \| & \leq \varepsilon \delta_b \delta_c ( 1 + \varepsilon^2 \delta_b \delta_c)  ,
\\
\|-W_c L_1^\infty  W_c \| & \leq  
 \varepsilon^2 \delta_b \delta_c^2 ,
\\
\| L_1^\infty  W_c W_b + W_bW_c L_1^\infty  (Id + W_c W_b)\| &\leq 
\varepsilon^2 \delta_b^2 \delta_c ( 2 + \varepsilon^2 \delta_b \delta_c) .
\end{align*}
	We then obtain the bound 
	\[
	\|M_4\| \leq  \Delta := \varepsilon \delta_b \delta_c  \left( 1 + \varepsilon(2 \delta_b +  \delta_c) + \varepsilon^2 \delta_b \delta_c (1 + \varepsilon \delta_b) \right)
	\]
 by summing the component bounds. 
  	
	We now perform the final estimate.
By using 
\eqref{eq:SemiGroup_Conjugate} we obtain 
	\begin{align*}
		 e^{Mt} &= 		
		 (P_cP_b)
		 \exp
		 \left\{ 
		 \left[M_3+M_4 + (P_cP_b)^{-1}M_2(P_cP_b)\right]
		 t
		 \right\}
		 (P_cP_b)^{-1} .
	\end{align*}
By then applying \eqref{eq:SemiGroup_Estimate_1} to the sum of $ M_3$ and the bounded operator $M_4 + (P_cP_b)^{-1}M_2(P_cP_b)$ we obtain,
with $C_{1,\infty} := \max \{C_1,C_\infty \}  $,
	\begin{align*}
		\| e^{Mt}\| &\leq
				\|P_cP_b\|
					\cdot 
				\|(P_cP_b)^{-1}\|
	C_{1,\infty} 
		 \exp
		 \left\{ \mu_1 +  C_{1,\infty} 
		 	\left\|
		 	M_4 
		 	+
		 	 (P_cP_b)^{-1}M_2 (P_cP_b)
		 	 \right\|
		 	t
		 	\right\} .
	\end{align*}
Defining $C_s =  \max \{ C_1, C_\infty\} (1+\varepsilon \delta_b)^2 ( 1 + \varepsilon \delta_c)^2
$ and plugging in our bounds, we finally infer 
	\[
\|	e^{Mt} \| \leq C_s
 e^{( \mu_1 + C_s \delta_a+ \Delta \max \{ C_1, C_\infty\}   )t} .
  \qedhere
	\]

	%
\end{proof}

\begin{remark}
	\label{rem:SemiGroupEll1}
	If we use the $p=1$ norm for the product space $ \C^N \times X_\infty$ then our bound for $ \Delta $ can be sharpened to
		\[
	\|M_4\| \leq \varepsilon \delta_b \delta_c   \max \left\{ 1 + \varepsilon \delta_c(1 + \varepsilon \delta_b ) , \varepsilon \delta_b ( 2 + \varepsilon^2 \delta_b \delta_c)  \right\}. 
	\]
\end{remark}

\end{document}